\documentclass[11pt]{amsart}

\usepackage{amscd}
\usepackage{xypic}  %commu
\usepackage{amssymb}
\usepackage{amsthm}
\usepackage{epsfig}
\usepackage{paralist, comment}

\usepackage{tikz}

\usepackage[all,cmtip]{xy}
\usepackage{amsmath}
\usepackage{xcolor}

\newtheorem{thm}{Theorem}[section] 

\newtheorem{lemma}[thm]{Lemma}

\newtheorem{proposition}[thm]{Proposition}

\newtheorem{claim}[thm]{Claim}

\newtheorem{corollary}[thm]{Corollary}

\newtheorem{case}[thm]{Case}
 
  \newtheorem{definition-remark}[thm]{Definition-Remark}
  \theoremstyle{definition}
    \newtheorem{example}[thm]{Example}
     \newtheorem{remark}[thm]{Remark}

\def\geq{\geqslant}
\def\leq{\leqslant}

\def\ker{\operatorname{ker}}

\def\ker{\operatorname{ker}}

\def\min{\operatorname{min}}
\def\im{\operatorname{im}}

\def\max{\operatorname{max}}

\def\c1{\operatorname{c_1}}
\def\c2{\operatorname{c_2}}

\def\NE{\operatorname{\overline{NE}}}
\def\Num{\operatorname{\rm Num}}
\def\Div{\operatorname{\rm Div}}

\def\Alb{\operatorname{\rm Alb}}
\def\Aut{\operatorname{\rm Aut}}

\def\CC{{\mathbb C}}

\def\ZZ{{\mathbb Z}}
\def\NN{{\mathbb N}}
\def\QQ{{\mathbb Q}}
\def\RR{{\mathbb R}}
\def\FF{{\mathbb F}}
\def\PP{{\mathbb P}}

\def\C{{\mathcal C}}

\def\L{{\mathcal L}}

\def\O{{\mathcal O}}

\def\E{{\mathcal E}}

\def\F{{\mathcal F}}

\def\C{{\mathcal C}}

\def\FF{{\mathbb F}}

\def\c{\mathfrak{c}}

\def\cong{\simeq}

\def\+{\oplus}               % direct sum
\def\*{\otimes}                  % tensor product
\def\Aut{\operatorname{Aut}}

\def\Pic{\operatorname{Pic}}
\def\NS{\operatorname{NS}}
\def\Num{\operatorname{Num}}

\begin{document}

\title{The classification of complex algebraic surfaces}

\author[C.~Ciliberto]{Ciro Ciliberto}
\address{Ciro Ciliberto, Dipartimento di Matematica, Universit\`a di Roma Tor Vergata, Via della Ricerca Scientifica, 00173 Roma, Italy}
\email{cilibert@mat.uniroma2.it}

%\keywords{}

%\subjclass{}

 \begin{abstract} These notes contain the classification of complex algebraic surfaces following Mori's theory. They includes the $P_{12}$--Theorem, Sarkisov's programme in the surface case and Noether--Castelnuovo's Theorem in its classical version.
 \end{abstract}

\maketitle

\tableofcontents

\section{Introduction} \label{sec:intro}

This text grew up from the notes of a graduate course I  gave at the University of Roma ``Tor Vergata'' in the academic year 2018--19. The subject is the classification of complex algebraic surfaces following Mori's theory. It includes the $P_{12}$--Theorem, Sarkisov's programme in the surface case and Noether--Castelnuovo's Theorem in its classical version. I am indebted to the first chapter of Matzuki's book  \cite {Matz}, though I simplified and streamlined his exposition in some points. For the $P_{12}$--Theorem I partly followed the unpublished paper \cite {F} by P. Francia and the recent paper \cite {CL} by Catanese--Li. The classification of bielliptic surfaces follows the short and elegant exposition by Bombieri--Mumford in \cite {BM}. The classical version of Noether--Castelnuovo's Theorem follows Calabri notes \cite {Cal}. 

This text is intended for people who have a good knowledge of basic algebraic geometry. Generally speaking, acquaintance with  basic parts of books like \cite {GrHa, Hart} should be sufficient. In any case I tried to make the text as self--contained as possible. For this reason a first chapter is devoted to some preliminaries.

\section{Some preliminaries}\label{sec:prel}

This section  is a non--exhaustive collection of some preliminaries which  we will later use. The expert reader may skip this section.

In what follows, by a \emph{surface} we will mean a smooth, irreducible projective complex variety of dimension 2.

\subsection{Projective morphisms}\label{ssec:proj}

We will use the following  basic:

\begin {proposition}[cf. \cite {Hart}, Chapt. III, Prop. 7.12]\label{prop:ev}
Let $X,Y$ be schemes and $g: Y\to X$ a morphism. Let $\E$ be a vector bundle on $X$. To give a morphism $f: Y\to \PP(\E)$ such that the diagram
\[
\xymatrix{ 
&Y\ar_{g}[d]\ar[r]^f&\PP(\E)\ar^{\pi}[d] \\
&X\ar^{\text{id}}[r] &  X&
}
\]
commutes, is equivalent to give a line bundle $L$ on $Y$ and a surjective map of sheaves $g^*(\E)\twoheadrightarrow L$. 
\end{proposition}

In case $X$ is a point, this expresses the well known fact that to give a morphism of a variety $Y$ to a projective space of dimension $n$ is equivalent to give a line bundle $L$ on $Y$ which is base point free and $\dim(|L|)\geq n$.  

\subsection {Basic invariants}  Consider the exponential exact sequence
\begin{equation}\label{eq:exp}
0\longrightarrow \ZZ\longrightarrow \O_S\stackrel{2\pi i {\rm exp}\cdot  }\longrightarrow \O^*_S\longrightarrow 0
\end{equation}
which induces the exact cohomology sequence
\begin{equation*}\label{exp}
0\to H^ 1(S,\ZZ)\longrightarrow H^ 1(S, \O_S) \longrightarrow H^ 1(S,\O^ *_S)\stackrel {\rm ch} \longrightarrow H^ 2(S,\ZZ)\longrightarrow H^ 2(S, \O_S)
\end{equation*}
One sets $H^ 1(S,\O^ *_S)=\Pic(S)$, the \emph{Picard group} of $S$, which is the group of line bundles on $S$, modulo isomorphisms.
One denotes by $q(S)$ (often simply by $q$) the dimension of the vector space $H^ 1(S, \O_S)\cong H^0(S, \Omega^ 1_S)$: this is called the \emph{irregularity} of $S$. It is well known that $H^ 1(S,\ZZ)$ is a maximal rank  lattice inside $H^ 1(S, \O_S)$, so that 
\[
\Pic^ 0(S):=\frac {H^ 1(S, \O_S)} {H^ 1(S,\ZZ)}
\] is a complex torus of dimension $q$ called the \emph{Picard variety} of $S$. Note that the \emph{first Betti number} $b_1(S)$, i.e., the rank of $H^ 1(S,\ZZ)$, is $2q$.

The map ${\rm ch}$ is called the \emph{Chern class}. One sets $NS(S):=\im ({\rm ch})$, which is called the \emph{N\`eron--Severi group} of $S$, and is a finitely generated abelian group. Then we set $\Num(S)=NS(S)/\equiv$, where $\equiv$ is the \emph{numerical equivalence}. This is a finitely generated free abelian group, whose rank is denoted by $\rho(S)$ (often simply by $\rho$).  Then we denote  $N^1(S):=\Num(S)\otimes _\ZZ\RR$, a real vector space of dimension $\rho$. 

We denote by $Z_1(S)$, or by $\Div(S)$, the free abelian group generated by all divisors on $S$. There is an obvious epimorphism $\lambda:Z_1(S)\to \Pic(S)$, which sends a divisor $D$ to the line bundle $\O_S(D)$. One has $\lambda(D)=\lambda(D')$ if and only if $D$ and $D$' are \emph{linearly equivalent}, denoted by $D\sim D'$. The epimorphism $\lambda$ induces an epimorphism $\Lambda: (Z_1(S)/\equiv)\to \Num(S)$, which is  an isomorphism. We set $N_1(S)=(Z_1(S)/\equiv)\otimes_\ZZ\RR$, which is isomorphic to $N^1(S)$, so it has also dimension $\rho$. In fact we will often identify $Z_1(S)$ and  $\Num(S)$, and $N_1(S)$ and $N^1(S)$, via the isomorphism $\Lambda$. So we will speak of numerical equivalence class of divisors or equivalently of line bundles. 

A divisor $D$ in $N_1(S)$ is called \emph{effective}, or a \emph{curve}, if $D=\sum_{i=1}^nm_iD_i$, where $D_1,\ldots, D_n$ are distinct, irreducible curves on $S$ and $m_i\geqslant 0$ for $1\leqslant i\leqslant n$.
We define $\NE(S)$ to be the closure in $N_1(S)$ of the convex cone  generated by all  {curves} on $S$. This is called the \emph{Kleiman--Mori cone} of $S$. The cone $\NE(S)$  is \emph{strictly convex} (see \S \ref {ssec:cones} below). 

On $\Pic(S)$ there is the bilinear form determined by the \emph{intersection product}. It induces bilinear forms (also called \emph{intersection products}) on $\NS(S)$,  on $N^1(S)$ and on $N_1(S)$. The intersection product is non--degenerate on $N^1(S)$ and on $N_1(S)$. 

The dimension of $H^ 2(S, \O_S)\cong H^0(S, \Omega^2_S)$ is denoted by $p_g(S)$ (or simply by $p_g$) and is called the \emph{geometric genus} of $S$, and $p_a(S)=p_g(S)-q(S)$ (also denoted by $p_a$) is called the \emph{arithmetic genus} of $S$. Note that
\[
\chi(S)=h^ 0(S, \O_S)-h^ 1(S, \O_S)+h^ 2(S, \O_S)=p_g(S)-q(S)+1=p_a(S)+1
\]
(also denoted by $\chi$). 

The line bundle $\Omega^2_S$ is the \emph{canonical bundle} of $S$, and one denotes it by $K_S$. Often $K_S$ denotes also any \emph{canonical divisor}, i.e., any divisor such that  $\Omega^2_S\cong \O_S(K_S)$. Note that $p_g-1=\dim (|K_S|)$. Moreover, if $p_g=0$, then the map {\rm ch} is surjective and $NS(S)=H^2(S,\ZZ)$, i.e., any 2--cycle on $S$ is homologous to an algebraic one: one expresses this by saying that \emph{any 2--cycle is algebraic}.  

The bundles $K_S^{\otimes n}$ with $n$ a positive integer, are called \emph{pluricanonical bundles} and $P_n(S):=h^ 0(S, K_S^{\otimes n})$ is called the \emph{$n$--th plurigenus} of $S$ (also denoted by $P_n$). 

The \emph{Betti numbers} $b_i(S)$ (also denoted by $b_i$) are related to the above invariants. First, by Poincar\'e duality, one has $b_i=b_{4-i}$, hence $b_0=b_4=1$ and $b_1=b_3=2q$. Moreover, since $NS(S)\subseteq H^ 2(S,\ZZ)$, one has $\rho\leqslant b_2$. We denote by $e(S)$ (or simply by $e$) the \emph{Euler-Poincar\'e characteristic} of $S$.

Recall that Hodge decomposition gives 
\[
H^ 2 (S,\CC)=H^{2,0}\oplus H^{1,1} \oplus H^{2,0}
\]
with $H^{2,0}=\overline{H^{2,0}}$ and $H^{2,0}\cong H^ 0(\Omega^2_S)$. Hence $b_2\geqslant 2p_g$. 

The \emph{Kodaira dimension} $\kappa(S)$ of $S$ (sometimes simply denoted by $\kappa$)  is defined in the following way:\\
\begin{inparaenum}
\item [$\bullet$] $\kappa=-\infty$ if $P_n=0$ for all $n\in \NN$;\\
\item [$\bullet$] $\kappa=\max\dim\{\varphi_{|nK_S|}(S)\}_{n\in \NN}\in \{0,1,2\}$ if there is an $n\in \NN$ such that $P_n>0$.
\end{inparaenum}

\begin{thm}\label{thm:kod}  If $S$ is a surface for which there is an $n\in \NN$ such that $P_n>0$, then $\kappa(S)=k$ if and only if there are two positive real numbers $a,b$ and there is an $n_0\in \NN$  such that for all $n\in \NN$ one has
\[
an^ k\leqslant h^ 0(nn_0K_S)\leqslant bn^k.
\]
\end{thm}

For the proof, see \cite[Thm. 10.2]{Iit}.

Note that the group $\Pic(S)$ is multiplicative the group operation being the tensor product of line bundles, whereas $\Num(S)$ is additive. We will often abuse notation and use the additive notation also for the elements of $\Pic(S)$, i.e., for line bundles. For example we will denote by $nK_S$ the pluricanonical bundles.

\subsection{The ramification formula}\label{ssec:ram}

A local computation proves the following:

\begin{thm}[Ramification Formula]\label{thm:ram}
Let $X,Y$ be smooth, irreducible, projective varieties of the same dimension $n$ and let $f: X\to Y$ be a surjective morphism. Then
\[
K_X=f^ *(K_Y)+R
\]
where $R$ is the locus where $df$ has rank smaller than $n$, with its natural scheme structure.
\end{thm}

It should be remarked that a classical result by Zariski--Nagata, called the \emph{purity of the branch locus theorem}, proves that in the set up of the above theorem $R$ is a divisor, called the \emph{ramification divisor} of $f$. The divisor $B:=f(R)$ is called the \emph{branch divisor} of $f$.

The \emph{Riemann--Hurwitz formula} is the particular case of the above theorem in the case in which $X$ and $Y$ are curves. 

\subsection{Basic formulas} 

\emph{Adjunction formula} says that if $C$ is a curve on the surface $S$, then
\[
K_S\otimes \O_S(C)_{|C}\cong \omega_C
\]
where $\omega_C$ is the \emph{dualizing sheaf} of $C$. In particular
\[
(K_S+C)\cdot C=2p_a(C)-2.
\]
The \emph{arithmetic genus} of a curve $C$ is defined as 
\[
p_a(C)=1-\chi(\O_C)=h^1(\O_C)-h^0(\O_C)+1.
\]
One has $p_a(C)\geqslant 0$ as soon as $C$ is \emph{algebraically connected}, i.e., if $h^0(\O_C)=1$. Note that if $C$ is reduced and topologically connected, then it is also algebraically connected.
The \emph{geometric genus} $p_g(C)$, also denoted by $g(C)$ (or simply by $g$ if there is no danger of confusion), of a reduced curve $C$ is the arithmetic genus of the normalization of $C$. One has $p_g(C)\geqslant 0$ if $C$ is irreducible and one may have   $p_g(C)< 0$ only if $C$ is reducible. 

\emph{Noether's formula} says that
\begin{equation}\label{eq:NNN}
12\chi(S)=K^2_S+e(S).
\end{equation}

\emph{Riemann--Roch Theorem} says that if $L$ is a line bundle on $S$, then
\[
\chi(L):= h^ 0(L)-h^ 1(L)+h^ 2(L)=\chi(S)+ \frac {L\cdot (L-K_S)}2.
\]

\emph{Serre duality} says that
\[
h^ i(L)=h^ {2-i}( K_S-L), \quad \text{for} \quad 0\leqslant i\leqslant 2. 
\]

{Riemann--Roch Theorem} for vector bundles on a smooth irreducible curve $C$ of genus $g$, says that if $\mathcal E$ is a rank $k$ vector bundle on $C$, one has
\[
\chi(\mathcal E)=h^ 0(\mathcal E)-h^ 1(\mathcal E)= \deg(\mathcal E)+k(1-g),
\]
where $\deg(\mathcal E):=\deg(\wedge ^k\mathcal E)$. 

As a general reference, see \cite [Chapt. V, \S 1]{Hart} and \cite [App. A, \S 4]{Hart}.

\subsection{Ample line bundles} A line bundle $L$ on $S$ is \emph{ample} if there is a positive multiple $nL$ of $L$ which is \emph{very ample}, i.e., the map $\phi_{nL}$ determined by $|nL|$  is an isomorphism of $S$ onto its image. 

\begin{thm}[Nakai--Moishezon Theorem]\label{thm:NM} A line bundle $L$ on $S$ is ample if and only if $L^ 2>0$ and $L\cdot C>0$ for every non--zero curve on $S$.
\end{thm} 

\begin{thm} [Kleiman's Criterion] \label{thm:Kl} A line bundle $L$ on $S$ is ample if and only if $L\cdot C>0$ for every $C\in \NE(S)-\{0\}$. 
\end{thm} 

For the proofs, see \cite [Chapt. 1, \S 1.2 B and Thm. 1.4.23]{Laz}.

A line bundle $L$ on $S$ is said to be \emph{nef} if $L\cdot C \geqslant 0$  for every curve on $X$. Note that $L$ is nef if and only if $L\cdot C\geqslant 0$ for every $C\in \NE(S)$. From Kleiman's Criterion follows that if $N$ is ample and $L$ is nef, then $L+N$ is ample.  This implies that if $L$ is nef, then $L^2\geqslant 0$. Indeed, if $N$ is ample, for any positive rational number $\varepsilon$ one has that the class of $L+\varepsilon N$ in $NS(S)$ is ample, hence $(L+\varepsilon N)^2>0$ and letting $\varepsilon$ go to 0, proves our assertion. 

\begin{thm}[Kodaira Vanishing Theorem]\label{thm:kvt} If $L$ is an ample line bundle on a smooth projective variety $X$ of dimension $n$, then
$h^i(K_X+L)=0$, for $1\leqslant i\leqslant n$. 
\end{thm}

For the proof, see \cite [Chapt. 5, \S 4.2]{Laz}.

There is a powerful generalization of Kodaira Vanishing Theorem which is useful to record. Let $X$ be a smooth projective variety of dimension $n$ and $L$ a line bundle on $X$. As in the surface case, $L$ is said to be \emph{nef} if for any curve $C$ on $X$ one has $L\cdot C\geq 0$. Moreover, $L$ is said to be \emph{big} if, in addition, $L^n>0$. 

\begin{thm}[Kawamata--Viehweg Vanishing Theorem]\label{thm:kvt} If $L$ is a big and nef line bundle on $X$, then
$h^i(K_S+L)=0$, for $1\leqslant i\leqslant n$. 
\end{thm}

For the proof, see \cite {Ka, V}.

\subsection{Hodge Index Theorem}

\begin{thm}[Hodge Index Theorem]\label{thm:hit} Let $S$ be a surface. The intersection product on $\Num(S)$ has signature $(1,\rho-1)$. In particular, if $D,E$ are divisor classes such that $D^2>0$, then 
\[\
D^ 2\cdot E^2\leqslant (D\cdot E)^2
\]
and equality holds if and $D$ and $E$ are numerically equivalent.
\end{thm}

For the proof, see \cite[Chapt. V, Thm. 1.9]{Hart}.

\begin{corollary}\label{cor:hit} Let $S$, $X$ be algebraic surfaces and $f:S\to X$ a morphism. If $E$ is a curve on $S$ contracted to a point by $f$, i.e., $f(E)$ is a point of $X$, then $E^2<0$. 
\end{corollary}
\begin{proof} Let $A$ be a very ample line bundle on $X$ and set $H=f^*(A)$. Then we have $H^ 2>0$ and $H\cdot E=0$. By the Hodge Index Theorem we have $E^ 2\leqslant 0$ and equality cannot hold because   $E$ is not numerically equivalent to $H$. \end{proof}

\subsection{Blow--up}\label{ssec:bu} Let $S$ be a smooth surface and $x\in S$ a point. The \emph{blow--up} of $S$ at $x$ is a morphims $\pi: \widetilde S\to S$ with a curve $E\cong \PP^1$ such that $E^ 2=-1$ (called the \emph{exceptional $(-1)$--curve of the blow--up} such that $\pi(E)=x$ and $\pi: S-E\to S-\{x\}$ is an isomorphism, hence $\pi$ is a birational morphism. It enjoies the following \emph{universal property}: For any morphism $f:X\to S$  such that the schematic counterimage of $x\in S$ in $X$ is a divisor, there is a morphism $g: X\to \widetilde S$ such that $f=\pi\circ g$ (see \cite [Chapt. II, Thm. 7.14]{Hart}).

Basic properties of the blow--up $\pi: \widetilde S\to S$ are the following \cite [Chapt. V, Prop. 3.2, 3.3, 3.4, 3.5]{Hart}:\\
\begin{inparaenum}
\item [$\bullet$] $\Pic(\widetilde S)=\Pic(S)\oplus \ZZ[E]$;\\
\item [$\bullet$] if $C,D\in \Pic(S)$, then $\pi^*(C)\cdot \pi^*(D)=C\cdot D$;\\
\item [$\bullet$]  if $C\in \Pic(S)$, then $\pi^*(C)\cdot E=0$;\\
\item [$\bullet$] $K_{\widetilde S}=\pi^ *(K_S)+E$;\\
\item  [$\bullet$] $\chi(S)=\chi(\widetilde S)$;\\
\item  [$\bullet$] $P_n(S)=P_n(\widetilde S)$, for all $n\in \NN$, hence $\kappa(S)=\kappa(\widetilde S)$;\\
\item  [$\bullet$] $q(S)=q(\widetilde S)$.
\end{inparaenum}

\begin{thm}[Castelnuovo's Contractibility Theorem] \label{thm:cct} Let $S$ be a surface and $E$ a smooth, irreducible curve on $S$. Then there is a morphism $\pi: S\to X$, with $X$ a smooth surface and a point $x\in X$ such that $S$ is the blow--up of $X$ at $x$ with exceptional divisor $E$ if and only if $E\cong \PP^1$ and $E^ 2=-1$, or equivalently $E^2=K_S\cdot E=-1$.
\end{thm}

For the proof, see \cite [Chapt. V, Thm. 5.7]{Hart}.

\subsection{Rational and birational maps}\label{ssec:rm} Let $S,S'$ be smooth, irreducible, projective surfaces. A birational map $f:S\dasharrow S'$ is determined by an isomorphism between two Zariski open subsets of $S$ and $S'$. 
Then we can consider the \emph{set of definition} of $f$, which is the maximal open subset of $S$ where $f$ is well definied, i.e., it is a morphism. The complement of the set of definition consists of finitely many points of $S$, called the \emph{indeterminacy points} of $f$. 

The following main properties of birational maps have to be recalled:\\
\begin {inparaenum}
\item [$\bullet$] the indeterminacy locus of a birational map consists of isolated points (see \cite [Rappel II.4]{bea};\\
\item [$\bullet$] if $f:S\to S'$ is a birational morphism (i.e., the indeterminacy set is empty), then $f$ is the composite of finitely many blow--ups (see \cite [Thm. II.7]{bea});\\
\item [$\bullet$] [Resolution of indeterminacies] if $f:S\dasharrow S'$ is a birational map, then there is a commutative diagram
\begin{equation*}
\xymatrix{ 
&X\ar_{p}[d]\ar^{q}[dr] &\\
&S \ar@{-->}[r]^{f} &  S'&
}
\end{equation*}
with $X$ a surface and $p,q$ birational morphisms (see \cite [Thm. II.7]{bea};\\
\item [$\bullet$] accordingly, if $f:S\dasharrow S'$ is a birational map, then $f$ is composite of finitely many blow--ups and inverse of blow--ups.\end{inparaenum}

By taking into account the properties of blow--ups, we have that, if $f: S\dasharrow S'$ is a birational map, one has:\\
\begin{inparaenum}
\item [$\bullet$] $q(S)=q(S')$;\\
\item [$\bullet$] $P_n(S)=P_n(S')$, for all $n\in \NN$, hence $\kappa(S)=\kappa(S')$;\\
\item [$\bullet$] $\chi(S)=\chi(S')$.\\
\end{inparaenum}

If $f: S\dasharrow C$ ia a dominant rational map, where $S$ is a surface and $C$ a smooth curve, then we still have the resolution of the indeterminacies, i.e., a commutative diagram 
\begin{equation*}
\xymatrix{ 
&X\ar_{p}[d]\ar^{q}[dr] &\\
&S\ar@{-->}[r]^{f} &  C&
}
\end{equation*} 
where $p$ and $q$ are  rational morphisms, the former one birational. An important remark is the following: if $C$ is not rational, then $f$ has no indeterminacies. Indeed, assuming that $p$ is composed with the minimal number of blow--ups, we see that the exceptional curve of the last blow--up has to be mapped dominantly on $C$, because it cannot be contracted to a point. This proves that $C$ is unirational, hence rational, by L\"uroth's theorem.

\subsection{The relative canonical sheaf} \label {ssec:rcs}

Let $S$ be a surface, $C$ a smooth projective curve, and $f:S\to C$ a surjective morphism with connected fibres of genus $g$. Consider the \emph{relative canonical line bundle}
\[
\omega_{S|C}= K_S\otimes f^ *(K_C^\vee)
\]
and its image
\[
f_*\omega_{S|C}=f_*(K_S)\otimes K_C^\vee
\]
which is a rank $g$ vector bundle on $C$.

\begin{thm}\label{thm:rcs} In the above setting, if $f$ is relatively minimal, i.e., no fibre of $f$ contains a $(-1)$--curve,  the singular fibres of $f$ are reduced and have at most nodes as singularities, one has
\[
\deg(f_*\omega_{S|C})> 0
\]
unless  $f$ is \emph{isotrivial}, i.e., unless the fibres of $f$ are all isomorphic.
\end{thm}

For the proof, see \cite [Theorem (17.3)] {BPHV}. 

\subsection{Cones}\label{ssec:cones} A \emph{cone} in a euclidean space $\RR^n$ is a set $\C$ such that if $x,y\in \C$ then also $\lambda x+\mu y\in \C$, for all $\lambda,\mu\geqslant 0$. Note that a cone is convex. The cone will be said to be \emph{strictly convex} if it contains no non--zero vector subspace of $\RR^n$.

A hyperplane of $\RR^n$, with equation $f(x)=0$, will be said to be a  \emph{supporting hyperplane} of $\C$ if $f(z)\geqslant 0$ for all $z\in \C$. 
A \emph{face} $F$ of $C$ is the intersection of $\C$ with a supporting hyperplane, i.e., there is a supporting hyperplane with equation $f(x)=0$ such that $F=\{z\in \C: f(z)=0\}$.  A \emph{ray} is a half--line of the form $\RR^+x$, for $x\in \C-\{0\}$. A ray $R$ is said to be \emph{extremal} if it is a face, in which case clearly $x+y\in R$ implies that $x,y\in R$.

A cone $\C$ is said to be \emph{polyhedral} if it is the convex one generated by finitely many rays, i.e., if there are rays $R_1,\ldots, R_h$ such that $\C=R_1+\ldots +R_h$. 

\subsection{Complete intersections}\label{ssec:ci} A more general form of \emph{adjunction formula} says that if $X$ is a smooth variety of dimension $n$ and if $Y$ is an effective divisor on $X$, then
\[
\omega_Y\cong(\omega_X\otimes \O_X(Y))_{|Y},
\]
where $\omega_X$ and $\omega_Y$ are the dualizing sheaves of $X$ and $Y$. 

If $n\geq 3$, if $Y$ is smooth and $L:=\O_X(Y)$ is ample (or even big and nef), from the exact sequence
\[
0\to L^ \vee\to \O_X\to \O_Y\to 0
\]
and by Kodaira Vanishing Theorem (or by Kawamata--Viehweg Vanishing Theorem), we have
\[
0=H^1(L^ \vee)\to H^1(\O_X)\to H^1(\O_Y)\to H^2(L^ \vee)=0
\]
hence $h^1(\O_X)= h^1(\O_Y)$. 

Iterated application of the adjunction formula implies that if $Y\subset X$ is the $n-h$ dimensional, complete intersection of divisors $Y_1,\ldots, Y_h$, one has 
\[
\omega_Y\cong(\omega_X\otimes \O_X(Y_1+\ldots+Y_h))_{|Y}.
\]

Similarly, iterated application of the Kodaira Vanishing Theorem tells us that if $Y_1,\ldots, Y_h$ are ample and if $Y$ is smooth of dimension at least 2, then $h^1(\O_X)= h^1(\O_Y)$.

Since $\omega_{\PP^ n}\cong \O_{\PP^n}(-n-1)$, if $Y$ is the complete intersection of dimension $n-h$ of hypersurfaces of $\PP^n$ of degrees $d_1,\ldots, d_h$ (in this case $Y$ is said to be a complete intersection of \emph{type} $(d_1,\ldots, d_h)$), one has
\[
\omega_Y\cong\O_Y(d_1+\cdots+ d_h-n-1).
\]

Since $h^1(\O_{\PP^ n})=0$, if $Y$ is a smooth complete intersection of dimension at least 2, then  $h^1(\O_Y)=0$. 

Let now $X$ be a surface of degree $d$ in $\PP^3$, with \emph{ordinary singularities}, i.e., with at most the following singularities: a curve $\Gamma$ of double points generically the transverse intersection of two branches, with at most finitely many pinch points, with $\Gamma$ having at most finitely many triple points as singularities, with three independent tangent lines, which are triple points also for $X$. Then the normalization $\nu: S\to X$ is smooth. A suitable application of the adjunction formula implies that the canonical bundle $\omega_S$ of $S$ is the pull--back to $S$ of $\mathcal I_{\gamma|\PP^3}(d-4)$, off the pull back of $\Gamma$ on $S$. This implies that for any $n\in \ZZ$, there is a map
\[
r_n: H^ 0(\mathcal I_{\Gamma|\PP^3}(n))\to H^0(\omega_S\otimes \O_S(n-d+4))
\]
where $\O_S(m):=\nu^*(\O_X(m))$ for any $m\in \ZZ$. The map $r_n$ is surjective for all $n\in \ZZ$. 

For this, see \cite[ Mumford's appendix to Chapt. III]{Zar2}. 

\subsection{Stein factorization}\label {ssec:stein} Let $f: X\to Y$ be a dominant projective morphism of noetherian schemes. Then there is a commutative diagram
\begin{equation*}
\xymatrix{ 
&X\ar_{f'}[d]\ar^{f}[dr] &\\
&Y\ar^{g}[r] &  S'&
}
\end{equation*}
where $f'$ is a projective morphism with connected fibres and $g$ is a finite morphism.

See \cite [Chapt. III, Cor. 11.5]{Hart}.

\subsection{Abelian varieties} 

\begin{proposition}\label{prop:pp} Let $A$ be a complex torus, $X$ a variety and $f: X\to A$ an \'etale cover. Then $X$ is a complex torus.
\end{proposition}

\begin{proof}  Let $n$ be the dimension of $A$ and $X$. The \'etale cover $f$ corresponds to a subgroup $G$ of finite index of $\pi_1(A)=H_1(A,\ZZ):=\Lambda\cong \ZZ^{2n}$. Then $A\cong \CC^n/\Lambda$. Hence $X\cong \CC^n/G$. \end{proof}

\begin{thm}[Poincar\'e Complete Reducibility Theorem] \label{thm:pcrt} Let $X$ be an abelian variety and $Y$ an abelian subvariety. Then there is an abelian subvariety $Z$ of $X$ such that $Y\cap Z$ is finite and $Y+Z=X$, i.e., there is an isogeny
\[
(y,z)\in Y\times Z\to y+z\in X.
\]
Moreover there is a morphism $\phi: X\to Z$ such that $Y$ is contained in a fibre of $\phi$. 
\end{thm}

For the proof, see \cite [p. 173]{Mum}.

\subsection{The Albanese variety}\label{ssec:alb} Let $X$ be a smooth, projective, irreducible variety. We set $q(X):=h^ 0(\Omega^ 1_X)$ (also simply denoted by $q$), which is the \emph{irregularity} of $X$.

There is an injection
\[
\int: H_1(X,\ZZ)\to H^ 0(\Omega^ 1_X)^\vee 
\]
sending a cycle $\gamma$ to the linear map
\[
\int_\gamma: \omega \in H^ 0(\Omega^ 1_X)\to \int_\gamma\omega\in \CC.
\]
The injection $\int$ realizes $H_1(X,\ZZ)$ as a lattice of maximal rank in $H^ 0(\Omega^ 1_X)^\vee$, in particular $b_1(X)=2q(X)$. The $q$--dimensional complex torus 
\[
\Alb(X):=\frac {H^ 0(\Omega^ 1_X)^\vee} {H_1(X,\ZZ)}
\]
is called the \emph{Albanese variety} of $X$. 

Fix a point $p\in X$. Then we have the well defined morphism
\[
\alpha_X: x\in X\to \int_p^ x \in \frac {H^ 0(\Omega^ 1_X)^\vee} {H_1(X,\ZZ)}
=\Alb(X)
\]
which is called the \emph{Albanese morphism} of $X$ with base point $p$, often simply denoted by $\alpha$. By changing the base point $p$, $\alpha_X$ changes by a translation in $\Alb(X)$. The Albanese morphism enjoyes the following \emph{universal property}: for every complex torus $A$ and for every morphism $\beta: X\to A$, there is a commutative diagram
\begin{equation*}
\xymatrix{ 
&X\ar_{\alpha}[d]\ar^{\beta}[dr] &\\
&\Alb(X)\ar^{\gamma}[r] &  A&
}
\end{equation*}

Moreover the map
\[
\alpha^*: H^0(\Omega^1_{\Alb(X)})\to H^0(\Omega^1_{X})
\]
is an isomorphism. This implies that $\alpha(X)$ spans $\Alb(X)$ as a complex torus.

If $X=C$ is a curve, then $\Alb(C)$ is denoted by $J(C)$ and called the \emph{jacobian variety} of $C$ and the map $\alpha_C$ is an embedding in this case. 

\begin{proposition}\label{prop:tr} Let $X,Y$ be smooth, irreducible, projective varieties and $f: X\to Y$ a surjective morphism. Then $q(X)\geq q(Y)$.
\end{proposition}
\begin{proof} By the universal property of the Albanese variety, there is a commutative diagram
\begin{equation*}\label{eq:K3}
\xymatrix{ 
&X\ar^{f}[r]\ar_{\alpha_X}[d]&Y\ar^{\alpha_Y}[d] \\
&\Alb(X)\ar^{g}[r] & \Alb(Y)\,\,
}
\end{equation*}
Since $\alpha_Y(Y)$ spans $\Alb(Y)$, then $g$ is surjective, hence 
\[q(X)=\dim(\Alb(X))\geq \dim(\Alb(Y))=q(Y).\] 
\end{proof}

\subsection{Double covers}\label{ssec:dc}
Let $X$ be a variety, $B$ be an effective divisor on $X$ and $L$ a line bundle on $X$ with a section $s\in H^ 0(X,L^ {\otimes 2})$ such that ${\rm div}(s)=B$. Let $T(L)$ be the \emph{total space} of $L$ with the projection $\pi: T(L)\to X$. Then $s$ defines in a natural way a section $\bar s$ of $\pi^*(L^ {\otimes 2})$ on $T(L)$. Set $Y={\rm div}(\bar s)\subset T(L)$, with the projection $p=\pi_{|Y}: Y\to X$. We claim that $p: Y\to X$ is a finite cover of degree $2$, i.e. a \emph{double cover}, which is \emph{branched} along $B$ and \emph{ramified} along $R:=p^ {-1}(B)$.

Indeed, let $\{U_i\}_{i\in I}$ be a covering of $X$ with affine open subsets, in which $L$ is defined by the cocycle $h_{ij}\in H^0(U_i\cap U_j, \O^*_X)$. Then $T(L)$ can be covered by open subsets of the form $\{U_i\times \CC\} _{i\in I}$ with transition functions $(u_i,z_i)\to (u_i,h_{ij}z_j)$. The section $s$ is locally given by functions $\{s_i\}_{i\in I}$ such that $s_i=h_{ij}^2s_j$. Then the section $\bar s$ is locally given by functions $\{\bar s_i=z_i^2-s_i\}_{i\in I}$. In fact we have the identity
\[
\bar s_i=z_i^2-s_i=(h_{ij}z_j)^2-h_{ij}^2s_j=h_{ij}^2\bar s_j.
\]
Then $Y$ is locally defined by the equations $\{z_i^2=s_i\}_{i\in I}$. If we have a point $u_i\in U_i$, then the counterimage of $u_i$ via $p: Y\to X$ consists of the points $(u_i,\pm \sqrt {s_i(u_i)})$, thus obtaining two points, unless $s_i(u_i)=0$.  
Thus the branch divisor is defined by the equations $\{s_i=0\}_{i\in I}$, i.e.,  ${\rm div}(s)=B$, and the ramification divisor is defined by the equations $\{z_i=0\}_{i\in I}$, i.e.,  $R:=p^ {-1}(B)$. The above computation also imply that $\O_Y(R)=p^*(L)$. 

Assume now $X$ is smooth. Then $Y$ is singular exactly at the points of $R$ corresponding to singular points of $B$. In particular,
if $B$ is non--reduced, then $Y$ is non--normal, and its normalization is the double cover of $X$ branched along the divisor $B_{\rm red}$. 

Note that if $U_i={\rm Spec}(A_i)$, then 
\[
p^{-1}(U_i)={\rm Spec}(B_i), \quad \text{with} \quad B_i=\frac {A_i[z_i]}{(z_i^2-s_i)},
\]
hence
\begin{equation}\label{eq:lampo}
B_i=A_i\oplus z_iA_i
\end{equation}
as an $A_i$--module. This tells us that 
\[
p_*(\O_Y)=\O_X\oplus L^ \vee.
\]
Indeed $p_*(\O_Y)$ is clearly a rank 2 vector bundle on $X$. There is an obvious injection
$\O_X\hookrightarrow p_*(\O_Y)$, which gives rise to the sequence
\begin{equation}\label{eq:lamp}
0\to \O_X\to p_*(\O_Y)\to F\to 0
\end{equation}
where $F$ is a line bundle. From \eqref {eq:lampo} we see that \eqref {eq:lamp} splits: indeed $\O_X\hookrightarrow p_*(\O_Y)$ corresponds to the inclusion $A_i\hookrightarrow B_i=A_i\oplus z_iA_i$ in \eqref {eq:lampo}, and the projection
$B_i=A_i\oplus z_iA_i\to A_i$ is locally well definied and glue in order to give a surjection
$p_*(\O_Y)\twoheadrightarrow \O_X$ splitting \eqref {eq:lamp}. Moreover, by \eqref {eq:lampo}, $F$ is locally defined by sections $\{\sigma_i\}_{i\in I}$ such that
$\sigma_iz_i=\sigma_jz_j$. Hence
\[
\sigma_iz_i=\sigma_jz_j=\sigma_j(h_{ij}^{-1}z_i),\quad \text{i.e.,} \quad \sigma_i=h_{ij}^{-1}\sigma_j,
\]
which proves that $F=L^ \vee$.

Note that, by the ramification formula, one has 
\[
K_Y=p^ *(K_X)+R=p^*(K_X\otimes L).
\]
Then the projection formula implies that
\[
p_*(K_Y)=K_X\oplus (K_X\otimes L).
\]
In particular we have
\[
p_g(Y)=h^0(K_Y)=h^0(K_X)+ h^0(K_X\otimes L)=p_g(X)+ h^0(K_X\otimes L)
\]
and
\begin{equation}\label{eq:qu}
q(Y)=h^1(\O_Y)=h^1(\O_X)+h^1(L^\vee)=q(X)+h^1(L^\vee).
\end{equation}

\begin{example}\label{ex:et} In the case $B=0$, we have an \emph{\'etale double cover}  $p: Y\to X$, which is never ramified. Such a cover is determined by a non--trivial line bundle $L$ on $X$ such that $L^{\otimes 2}\cong \O_X$, i.e., by a non--trivial element of order 2 in $\Pic(X)$.
\end{example}

\subsection {Riemann's Existence Theorem}\label{ssec:ret}

We will give two version of this theorem, which are not the most general, but are sufficient for our applications.

\begin{thm} \label {thm:ret}
Let $p_1,\dots,p_n$ be distinct points of $\PP^1$ and let  $\sigma_1,\ldots,\sigma_n$ be permutations of the set $\{1,\ldots, m\}$, for  $1\leq i\leq n$,  such that:\\
\begin{inparaenum}
\item [$\bullet$] $\sigma_1,\ldots, \sigma_n$ generate a subgroup of $\mathfrak S_m$ transitive over 
$\{1,\ldots, m\}$;\\
\item [$\bullet$] $\sigma_1\cdots \sigma_n={\rm id}$.
\end{inparaenum} 

Then there is a smooth, irreducible curve $C$, uniquely defined up to isomorphisms,  and a 
morphism $f: C\to \PP^1$ of degree $m$, with branch points $p_1,\dots,p_n$ and local monodromy around $p_i$ given by $\sigma_i$, for $1\leq i\leq n$. 
\end{thm}

\begin{thm}\label{thm:ret2}
Let $p_1,\dots,p_n$ be distinct points of $\PP^1$, let $G$ be a finite group of order $m$ and let $g_1,\ldots,g_n$ elements of $G$ of order $m_1,\ldots,m_n$ respectively, such that:\\
\begin{inparaenum}
\item [$\bullet$] $g_1,\ldots,g_n$ generate $G$;\\
\item [$\bullet$] $g_1\cdots g_n=1$.
\end{inparaenum} 

Then there is a unique irreducible curve $C$ with a $G$ action such that $C/G=\PP^1$ and the quotient morphism $f: C\to \PP^1$ is ramified at $p_1,\dots,p_n$ and over $p_i$ there are $\frac m{m_i}$ points, each with ramification index $m_i$, for $1\leq i\leq n$. \end {thm}

In the set up of the above theorem, $g_1,\ldots,g_n$ are also called the \emph{local monodromies} at $p_1,\ldots, p_n$, respectively. 

For informations on this classical subject, see \cite[Abhyankhar's Appendix 1 to Chapt. VIII]{Zar2}. 

\subsection{Relative duality}\label{ssec:reldu}

\begin{thm}[Relative Duality Theorem]\label {thm:reldu} Let $S$ be a surface and $f: S\to C$ a surjective morphism onto a smooth curve, with connected fibres. Let $\mathcal F$ be a locally free sheaf on $X$. Then there is a natural isomorphism of sheaves
\[
f_*(\mathcal F^\vee\otimes \omega_{S|C})\stackrel {\cong}\longrightarrow (R^1f_*\mathcal F)^ \vee.
\]
\end{thm}

For the proof, see \cite [Theorem (12.3)]{BPHV}.

\section{Characterization of the complex projective plane}\label{sec:p2}

We start with a theorem which characterizes the complex projective plane $\PP^2$. 

\begin{thm}\label{thm:charp2} Let $S$ be an irreducible projective surfaces with $H^1(S,\ZZ)=0$, $H^2(S,\ZZ)\cong \ZZ$ and $P_2=0$. Then $S$ is isomorphic to $\PP^2$.
\end{thm}

\begin{proof} Since $q=0$ and $p_g=0$, since $P_2=0$, we have $\Pic(S)\cong \ZZ$, so we can  choose a generator $L$ of  $\Pic(S)$ such that a multiple $nL$ with $n\gg 0$ is the class of a very ample line bundle. Hence $L^ 2>0$.  

By the hypotheses we have $e=3$ and $\chi=1$ because $q=0$ and $p_g=0$. By Noether's formula \eqref {eq:NNN} we  have $K^2_S=9$. 

On the other hand $K_S\cong aL$, for some integer $a$, hence $9=K^2_S=a^ 2L^ 2$, which implies that either $a=\pm 3, L^2=1$ or  $a=\pm 1, L^2=9$.  

If $a>0$ we have 
\begin{equation}\label{eq:1}
\chi(2K_S)=\chi+\frac {2K_S\cdot K_S}2=10.
\end{equation}
Since $2K_S\cong K_S+aL$, and $aL$ is ample because $a>0$, then $h^i(2K_S)=0$, for $1\leqslant i\leqslant 2$, by Kodaira Vanishing Theorem \ref {thm:kvt}. By \eqref {eq:1} we have $P_2=h^0(2K_S)=10$, a contradiction. 

So $a<0$. Then
\[
\chi(L)=\chi+\frac {L\cdot(L- K_S)}2=1+\frac {1-a}2  L^2.
\]
If $a=-1, L^ 2=9$, we have $\chi(L)=10$. But
\[
h^ 2(L)=h^ 0(K_S-L)=h^ 0(-2L)=0, \quad h^ 1(L)=h^ 1(K_S-L)=h^ 1(-2L)=0,
\]
so $h^ 0(L)=10$
thus $\dim(|L|)\geqslant 9$ and all the curves in $|L|$ are irreducible and reduced, since $L$ is a generator of  $\Pic(S)$.  Moreover $L\cdot (L+K_S)=0$, hence the curves in $|L|$ have arithmetic genus $1$, thus geometric genus 0 or 1. Let $x$ be any point of $S$. We can impose to the curves in $|L|$ a triple point at $x$, which imposes at most 6 conditions,  and we thus have a linear system $\Sigma$ of dimension at least 3. But since a triple point drops the geometric genus by 3, the curves in $\Sigma$ must be reducible, a contradiction.

In conclusion we must have $a=-3, L^ 2=1$. Then $\chi(L)=3$. Moreover, as above, $h^ i(L)=0$, for $1\leqslant i\leqslant 2$, hence $h^0(L)=3$, and the linear system $|L|$ defines a rational map $\phi: S\dasharrow \PP^2$. Recall that any curve in $|L|$ is irreducible and $L^ 2=1$. Moreover $L\cdot (L+K_S)=-2$ hence the curves in $|L|$ have arithmetic genus 0, thus they are smooth and rational, so isomorphic to $\PP^1$. Let $C\in |L|$. By the exact sequence
\[
0\longrightarrow \O_S\longrightarrow L\longrightarrow L_{|C}\cong \O_{\PP^1}(1)\longrightarrow 0
\]
and since $h^ 1(\O_S)=0$, the map $H^ 0(L)\to H^ 0(L_{|C})$ is surjective, so that each curve in $|L|$ is isomorphically mapped to its image in $\PP^ 2$. It is then clear that $\phi$ is a birational morphism mapping the curves in $|L|$ to the lines of $\PP^2$. Moreover there is no curve $E$ contracted by $\phi$ because then we would have $L\cdot E=0$, impossible because $L$ generates $\Pic(S)$. In conclusion $\phi$ is an isomorphism.  \end{proof}

\section{Minimal models}

Recall that, if we blow--up a smooth surface $X$ at a point $x$ we get $\pi: S\to X$, with an \emph{exceptional divisor} $E\cong \PP^ 1$ on $S$ which is contracted to $x$ and $\pi$ induces an isomorphism between $S-E$ and $X-\{x\}$. One has $E^2=-1$ and $K_S\cdot E=-1$ by the adjunction formula. Such a curve is called a \emph{$(-1)$--curve}.
Conversely, by Castelnuovo Contractibility Theorem \ref {thm:cct}, an irreducible curve $E$ on a surface $S$ is the exceptional divisor of a blow--up if $E^2=-1$ and $K_S\cdot E=-1$. 

\begin{example}\label{ex:1} If we blow--up two distinct points of $\PP^2$ we find a surface $S$ on which there are two (-1)--curves which are blown--down to the two points. There is another (-1)--curve though, which is the strict transform $\ell$ on $S$ of the line joining the two points.

Similarly, blow--up  $\PP^2$ at a point $x$, then blow--up the resulting surface at a point $y$ on the exceptional divisor $E$ which is blown--down to $x$. The point $y$ is called  an \emph{infinitely near point} to $x$ (see \S \ref {ssec:inp} below). On the resulting surface $S'$ we have a (1)--curve $E_2$ which is contracted to $y$, and the strict transform $E_1$ of $E$, which is smooth, rational, with $E_1^ 2=-2$ (such a curve is called a \emph{(-2)--curve}). Consider the unique line in $\PP^ 2$ passing through $x$ and having the tangent direction corresponding to the {infinitely near point} $y$. The proper transform $\ell$ of this line  on $S'$ is also a (-1)--curve.

If we blow--up five general point of $\PP^ 2$, we find a surface $S''$ with five (-1)--curves, i.e., the five exceptional divisors of the blow--ups. The proper transform of the unique conic of $\PP^ 2$ passing through the blown--up points is also a $(-1)$--curve.
\end{example}

A smooth, projective surface $S$ is said to be \emph{minimal}, if $S$ does not contain $(-1)$--curves. 

\begin{proposition}\label{prop:min} Let $S$ be a smooth, projective surface. There is a birational morphism $f: S\to X$ with $X$ minimal.
\end{proposition}

\begin{proof} The morphism $f: S\to X$ is a sequence of blow--downs of $(-1)$--curves. This process terminates reaching a minimal surface because if $f: S\to X$ is the blow--up of a smooth point, then 
\[
0\leqslant \dim (H^ 2(X,\RR))\leqslant \dim (H^ 2(S,\RR))-1.
\]
\end{proof} 

The surface $X$ in the statement of Proposition \ref {prop:min} is called \emph{minimal} or a \emph{minimal model} of $S$. 

\begin{remark}\label{rem:1} The minimal model is in general not unique. If we blow down the (-1)--curve $\ell$ on the surface $S$ in  Example \ref {ex:1}, we get a surface which is minimal, but not isomorphic to $\PP^ 2$, which is also a minimal model.

Similarly, if we blow down the curve $\ell$ on the surface $S'$ in  Example \ref {ex:1}.
\end{remark}

\section{Ruled surfaces}\label{sec:rul}

Let $S$ be a surface and $C$ a smooth, irreducible projective curve. Suppose there is a surjective morphism $f: S\to C$. Then $S$ is said to be \emph{ruled} over $C$ if there is a point $x\in C$ such that the (schematic) fibre $F_x=f^*(x)$ of $f$ over $x$ is isomorphic to $\PP^1$. 

\begin{lemma}\label{lem:ban} Suppose that $f:S\to C$ is ruled over $C$. Then there is a non--empty open subset $U$ of $C$ such that for all points $t\in U$ the fibre   $F_t=f^*(t)$ of $f$ over $t$ is isomorphic to $\PP^1$.
\end{lemma}  

\begin{proof} Let $x\in C$ be a point such that the fibre $F_x=f^*(x)$ of $f$ over $x$ is isomorphic to $\PP^1$. Then $f$ is smooth over $x$ and we can find an open neighborhood $U$ of $x$ such that $f$ is smooth over $U$. So for all $t\in U$, the fibre $F_t$ is smooth and irreducible. Moreover, $f$ is also flat over $U$, hence the genus of the fibres over the points of $U$ is constant (see \cite [Chapt. III, Cor. 9.10]{Hart}), i.e.,  for all $t\in U$, the fibre $F_t$ is smooth of genus 0, hence it is isomorphic to $\PP^1$. \end{proof}

\begin{thm}[Noether--Enriques' Theorem]\label{thm:NET} Suppose  $f:S\to C$ is ruled over $C$. Then there is an dense open subset $U$ of $C$ such that the following diagram commutes
\[
\xymatrix{ 
&f^{-1}(U)\ar^{\cong}[r]\ar_{f}[d]&U\times \PP^1\ar^{{\rm pr}_1}[d] \\
&U\ar^{{\rm id}}[r] & U\,\,
}
\]
In particular $S$ is birational to $C\times \PP^1$, and if $C\cong \PP^1$, then $S$ is \emph{rational}, i.e., birational to $\PP^2$. 
\end{thm}
\begin{proof} If $F$ is a general fibre of $f$, one has $F^2=0$, hence by adjunction formula, we have $K_S\cdot F=-2$, thus $p_g=0$. 

\begin{claim}\label{cl:opla}
There is a divisor $H$ on $S$ such that $H\cdot F=1$.
\end{claim}
 
 \begin{proof}[Proof of Claim \ref {cl:opla}] Since $p_g=0$, we have $H^2(\O_S)=0$. From the exponential sequence \eqref {eq:exp}, 
 we have a surjective map $\Pic(S)=H^1(\O_S^*)\to H^2(X,\mathbb Z)$. So it suffices to find a class $h\in H^2(X,\mathbb Z)$ such that $h\cdot f=1$, where $f$ is the class of $F$ in $H^2(X,\mathbb Z)$.
 
 Consider the map
  \[
  a\in H^2(X,\mathbb Z)\to a\cdot f\in \mathbb Z
  \]
  whose image is an ideal $(d)$ in $\mathbb Z$, where $d>0$. Hence the map
  \[
  a\in H^2(X,\mathbb Z)\to \frac {a\cdot f}d \in \mathbb Z
  \]
  is a linear form on $H^2(X,\mathbb Z)$. Poincar\'e duality implies that the cup product 
  \[
  H^2(X,\mathbb Z)\otimes H^2(X,\mathbb Z)\to H^4(X,\mathbb Z)\cong \mathbb Z
  \]  
  induces a map
  \[
  H^2(X,\mathbb Z)\to {\rm Hom}(H^2(X,\mathbb Z),\mathbb Z)
  \]
 that  is surjective and its kernel is the torsion subgroup of $H^2(X,\mathbb Z)$. Therefore there is $f'\in H^2(X,\mathbb Z)$ such that 
  \[
  a\cdot f'= \frac {a\cdot f}d, \quad \text {for all}\quad a\in H^2(X,\mathbb Z).
  \]
  Hence
  \[
  df'=f, \quad \text{modulo torsion}.
  \]
  Let $k$ be the class of $K_S$ in $H^2(X,\mathbb Z)$. We have
  \[
  f^2=0, f\cdot k=-2, \quad \text{hence}\quad f'^2=0, k\cdot f'=\frac {k\cdot f}d=-\frac 2d,
  \]
  and, since $(k+f')\cdot f'$ has to be even, this implies $d=1$, proving the claim. \end{proof}

Let now $U$ be a dense open subset of $C$ over which all fibres of $f$ are isomorphic to $\PP^1$ (see Lemma \ref {lem:ban}). Set $f_{|U}=\phi$ and consider $\E=\phi_*(\O_{\phi^{-1}(U)}(H))$, which is a rank 2 vector bundle over $U$. Up to shrinking $U$ we may assume that $\E$ is trivial over $U$. Now, by applying Proposition \ref {prop:ev} (with $Y=p^{-1}(U)$ and $\L=\O_{\phi^{-1}(U)}(H)$), we have an  epimorphism of sheaves $\phi^*(\E)\to \O_{\phi^{-1}(U)}(H)$, which determines a map $p^{-1}(U)\to \PP(\E)=U\times \PP^1$ which is clearly an isomorphism. \end{proof}

The surface $S$, endowed with a surjective morphism $f: S\to C$ is said to be a \emph{$\PP^1$--bundle}  or a \emph{scroll} over $C$, if every fibre of $f$ is isomorphic to $\PP^ 1$. These fibres are also called \emph{rulings} of $S$. A \emph{section} of the scroll $S$ is an irreducible curve
$\Gamma$ on $S$ which meets in one point the fibres of $f$. 

\begin{thm}\label{thm:plus}
Let $f:S\to C$ be a $\PP^1$--bundle over $C$. Then there is a rank 2 vector bundle $\E$ on $C$ with an isomorphism $\psi: S \to \PP(\mathcal E)$ which makes the following diagram commutative
\begin{equation*}\label{eq:K3}
\xymatrix{ 
&S\ar^{\psi}[r]\ar_{f}[d]&\PP(\mathcal E)\ar^{p}[d] \\
&C\ar^{{\rm id}}[r] & C\,\,
}
\end{equation*}
where $p$ is the obvious projection morphism.

Moreover two vector bundles $\E$ and $\E'$ give rise to the same $\PP^1$--bundle as above if and only if there is a line bundle $L$ such that $\E'=\E\otimes L$.
\end{thm}

\begin{proof} By Noether--Enriques' theorem, $S$ is locally trivial, so it is given by a 1--cocycle with values in ${\rm PGL}(2,\O_C)$. From the exact sequence of groups
\[
1\to \mathbb C^*\to {\rm GL}(2,\mathbb C)\to  {\rm PGL}(2,\mathbb C)\to 1
\]
we deduce the exact sequence of sheaves in groups
\[
1\to \O_C^* \to {\rm GL}(2,\O_C)\to  {\rm PGL}(2,\O_C)\to 1
\]
hence the exact sequence
\begin{equation}\label{eq:slurp}
\Pic(C)\cong H^1(\O_C^*)\to H^1({\rm GL}(2,\O_C))\to H^1({\rm PGL}(2,\O_C))\to H^2(\O_C^*)=1,
\end{equation}
because $H^2(\O_C^*)=1$. In fact we have the exact sequence
\[
1\to  \O_C^*\to K_C^*\to {\rm Div}(C)\to 0
\]
where $K_C^*$ is the  sheaf of non--zero rational functions on $C$ and ${\rm Div}(C)$ is the  sheaf of germs of divisors on $C$: the sections of ${\rm Div}(C)$ over an open subset $U$ of $C$ is the set of all maps $\nu:U\to \mathbb Z$ such that $\nu(p)\neq 0$ for only finitely many $p\in U$. The sheaves $K_C^*$ and $ {\rm Div}(C)$ are fine sheaves, and this implies that $H^2(\O_C^*)=1$. 

Finally the assertion follows right away from \eqref {eq:slurp}.\end{proof}

\begin{proposition}\label{prop:olp} Let $S$ be a minimal surface with a surjective morphism $f: S\to C$, with $C$ a curve, such that the general fibre of $f$ is isomorphic to $\PP^1$. Then $S$ is a $\PP^1$--bundle over $C$.
\end{proposition}

\begin{proof} Let $F$ be a fibre of $f$. We have $F^2=0$, $K_S\cdot F=-2$. If $F$ is irreducible it cannot be multiple. Indeed, if $F=mF'$ with $m>1$, then we have $-2=K_S\cdot F=m(K_S\cdot F')$, hence $m=2$ and $K_S\cdot F'=-1$. On the other hand $F'^2=0$ which is incompatible with $K_S\cdot F'=-1$ because $(K+F')\cdot F'$ has to be even. So for any irreducible fibre $F$ we have that its genus is 0, hence $F\cong \PP^1$. To finish the proof we have to show that no fibre $F$ can be reducible. Assume by contradiction that $F=\sum_{i=1}^nm_iC_i$ with $n>1$. Then, for all $i\in \{1,\ldots, n\}$ we have
\[
n_iC_i^2=C_i\cdot (F-\sum_{j\neq i}m_jC_j)=- \sum_{j\neq i}m_jC_i\cdot C_j<0
\]
because $F$ is connected (cf. \cite [Chapt. III, Ex. 11.4: Principle of Connectedness]{Hart}). Hence $C^2_i<0$ for all $i\in \{1,\ldots, n\}$. Since $K_S\cdot F<0$, there is a $j\in \{1,\ldots, n\}$ such that $K_S\cdot C_j<0$. This implies that  $K_S\cdot C_j=C_j^2=-1$, hence $C_j$ is a $(-1)$--curve, a contradiction. \end{proof}

We recall the following results (cf. \cite [Chapt. III, Prop. 2.6 and Corol. 2.14]{Hart}:

\begin{proposition}\label{prop:section} Given a $\PP^1$--bundle $f: \PP^1(\E)\to C$, there is a 1-to-1 correspondence between sections $\Gamma$ of $f$ and quotients $\E\twoheadrightarrow F$, with $F$ a line bundle. Under this correspondence, if $G=\ker(\E\twoheadrightarrow \F)$, then $G$ is a line bundle on $C$ and 
$\O_\Gamma(\Gamma)\cong F\times G^ \vee$. 
\end{proposition}

\begin{thm}[Grothendieck--Bellatalla's Theorem] \label{thm:gb} Every vector bundle on $\PP^1$ is the direct sum of line bundles.
\end{thm}

As a consequence we have:

\begin{corollary}\label{tmg:bp1} Any scroll on $\PP^1$ is of the form
\[
\FF_n\cong \PP(\O_{\PP^1}\oplus  \O_{\PP^1}(-n)).
\] 
\end{corollary}

The quotient
\[
\O_{\PP^1}\oplus  \O_{\PP^1}(-n)\to  \O_{\PP^1}(-n)\to 0
\]
corresponds to a section $C$ of $\FF_n$ such that $C^ 2=-n$. This is the unique curve with non--positive self intersection of $\FF_n$, except in the case $n=0$ in which $\FF_0\cong \PP^ 1\times \PP^1$: in this case there are no curves on $\FF_0$ with negative self intersection, but the fibres of both projections onto the factors  have self intersection 0.

\begin{example}[Elementary transformations] \label {ex:2} Let $f: S\to C$ be a scroll. 
Let $\alpha: T\to S$ be the blow--up of $S$ at a point $x$. The proper transform of the ruling through $x$ is a (-1)--curve, which can be blown--down to a smooth point $y$, via $\beta: T\to S'$. The surface  
$S'$ a new scroll over $C$, which is birational to $S$, linked to $S$ by the commutative diagram
\begin{equation*}
\xymatrix{ 
&T\ar_{\alpha}[d]\ar^{\beta}[dr] &\\
&S \ar@{-->}^{\gamma}[r]\ar[d]_f &  S'\ar[d]^{f'}&\\
&C\ar^{\rm id}[r] &  C&
}
\end{equation*}
The map $\gamma: S\dasharrow S'$ is called and \emph{elementary transformation}. 

If $S=\FF_n$, then $S'=\FF_i$ where $i=n-1$ is $x$ does not belong to an irreducible section $C$ of $\FF_n$ with $C^ 2=-n$ (this section is unique if $n>0$), whereas $i=n+1$ if $x$ belongs to such a curve. So $S$ has two minimal models $\FF_n$ and $\FF_{n\pm 1}$ unless $n=0$. In this case $\FF_1$ is not minimal, and to get a minimal model we have to blow--down its (-1)--curve, thus obtaining $\PP^ 2$. So if $n=0$, $S$ has the two minimal models $\FF_0$ and $\PP^2$. \end{example}

\section{Surfaces with non--nef canonical bundle}\label{sec:nonef}

Recall that a line bundle $L$ on a surface $S$ is said to be \emph{nef} if $L\cdot C\geq 0$, for all curves $C$ on $S$. 
In this section we will prove the following:

\begin{thm}[Extremal Contraction Theorem]\label{thm:ect} Let $S$ be a smooth, irreducible, projective surface such that $K_S$ is not nef. Then there is a morphism $\varphi: S\to V$, called an \emph{extremal contraction}, such that:\begin{inparaenum}\\
\item [$\bullet$] $\varphi$ is not an isomorphism;\\
\item [$\bullet$] if $C$ is any irreducible curve on $S$ which is contracted to a point by $\varphi$, then $K_S\cdot C<0$;\\
\item [$\bullet$] if $C, D$ are irreducible curves on $S$ both contracted to points by $\varphi$, then the classes of $C$ and $C$ in $NS(S)$ are equal, in particular $C\equiv D$;\\
\item [$\bullet$] conversely, if $C, D$ are irreducible curves on $S$, if $C$ is contracted to a point by $\varphi$ and if $C\equiv D$, then also $D$ is contracted to a point by $\varphi$;\\
\item [$\bullet$] $\varphi$ has connected fibres and $V$ is smooth and projective.
\end{inparaenum}

\end{thm}

The proof  is based on the following two results:

\begin{thm}[Rationality Theorem]\label{thm:rt} Let $S$ be a smooth, irreducible, projective surface such that $K_S$ is not nef and let $A$ be an ample line bundle on $S$. Define
\[
r_A:={\rm sup}\{t\in \RR^+: A+tK_S \quad \mbox{is nef}\,\,\}
\]
(called the \emph{canonical threshold} of $A$, sometimes simply denoted by $r$). Then $r_A\in \QQ^+$.
\end{thm}

\begin{thm}[Base Point Freeness Theorem]\label{thm:bpf}  Let $S$ be a smooth, irreducible, projective surface and let $A$ be an ample divisor on $S$. Consider
\[
L=A+rK_S\in \Div(S)\otimes_\ZZ\QQ
\]
with $r\in \QQ^+$ and $L$ nef. Then  the linear system $\ell L$ is base point free for $\ell\in \NN$, $\ell\gg 0$ and $\ell$ sufficiently divisible so that $\ell r\in \NN$. 
\end{thm}

\subsection{Proof of the rationality theorem}

We keep the notation introduced above.
 
 \begin{lemma}\label{lem:aa} Suppose there is a rational number $r_0\geqslant r_A$ such that $k(A+r_0K_S)$ is \emph{effective} for some $k\in \NN$, i.e., $h^0(k(A+r_0K_S))>0$. Then $r_A\in \QQ$. 
 \end{lemma}
 
 \begin{proof} Write  $k(A+r_0K_S)\sim \sum_{i=1}^ n d_iD_i$, where the $D_i$ are  distinct irreducible curves and $d_i$ positive integers, for $1\leqslant i\leqslant n$. Then we have the following identity in $\Pic(S)\otimes_\ZZ\QQ$
 \[
 K_S=-\frac 1{r_0} A +\frac 1{kr_0} \sum_{i=1}^ nd_iD_i
 \]
 thus
 \[
 A+tK_S= \frac {r_0-t}{r_0} A+ \frac t{kr_0}  \sum_{i=1}^ nd_iD_i, \quad \mbox {for any}\quad t<r_0.
 \]
For all irreducible curves $C$ on $S$ different from the $D_i$s and for all positive $t<r_0$ one has
\[
 (A+tK_S)\cdot C\geqslant  \frac {r_0-t}{r_0} A\cdot C+ \frac t{kr_0}  \sum_{i=1}^ nd_iD_i\cdot C\geqslant 0.
\]
Then for $t\in (0,r_0)$ one has that $A+tK_S$ is nef if and only if $(A+tK_S)\cdot D_j\geqslant 0$, for all $1\leqslant j\leqslant n$. This is the case if and only if
 \[
 \frac {r_0-t}{r_0} A\cdot D_j+ \frac t{kr_0}  \sum_{i=1}^ nd_iD_i\cdot D_j\geqslant 0, \quad \mbox {for all }\quad 1\leqslant j\leqslant n.
 \]
 Thus
 \[
 r_A=\min_{1\leqslant j\leqslant n}\{t_j:  \frac {r_0-t_j}{r_0} A\cdot D_j+ \frac {t_j}{kr_0}  \sum_{i=1}^ nd_iD_i\cdot D_j=0\}
 \]
 and the minimum is taken over a set of rational numbers, so is rational. \end{proof}
 
 \begin{lemma}\label{lem:num} If $r$ is a positive irrational number, then there are infinitely many pairs $(u,v)$ of positive integers such that
 \[
 0<\frac vu-r<\frac 1{3u}
 \]
 \end{lemma}
 
 \begin{proof}  Write $r$ as an infinite  continuous fraction $r=[a_0;a_1, \ldots ]$ so that
 $r=\lim_n c_n$
 where $c_n=\frac  {p_n}{q_n}$ are the convergents of the continuous fraction, for which
 \[
 c_1<c_3< \cdots <c_{2h+1}<\cdots r \dots < \cdots <c_{2h}<\cdots <c_4<c_2,
 \]
see \cite [\S 1.5]{BCP}.  It is well known that 
\[ 
 | c_n-c_{n-1}| = \frac 1{q_nq_{n-1}} \quad \mbox{for all }\quad n\in \NN.
\] 
If $n$ is even we have
\[
0<c_n-r< \frac 1{q_nq_{n-1}}<\frac 1{3q_n}\quad \mbox{if}\quad n\gg 0.
\]
This proves the assertion. \end{proof}
 
 \begin{proof}[Proof of the Rationality Theorem] We  argue by contradiction: we  assume $r_A$ non--rational and we will prove that there is a rational number $r_0$ verifying the hypoteses of Lemma \ref {lem:aa}, leading to a contradiction.
 
 For any pair $(x,y)$ of integers, we set
 \[
 P(x,y):=\chi(\O_S(xA+yK_S))=\chi+\frac 12 (xA+yK_S)\cdot (xA+(y-1)K_S).
 \]
We can interpret $P(x,y)$ as a polynomial of degree 2 in $x,y$. This polynomial is clearly not identically 0.
 
 Now choose infinitely many pairs $(u,v)$ of positive integers as in Lemma \ref {lem:num}, with $r=r_A$. For any such a pair, $P(ku,kv)$ is a quadratic polynomial in $k$ and $P(ku,kv)$ is identically zero if and only if the line $vx-uy=0$ is contained in the curve $P(x,y)=0$. Since we have infinitely many pairs $(u,v)$ as above at our disposal, we can choose $(u_0,v_0)$ in such a way that $P(ku_0,kv_0)$ is not identically 0 in $k$. Then there is a $k_0\in \{1,2,3\}$ such that $P(k_0u_0,k_0v_0)\neq 0$. Set
 \[
 M=k_0(u_0A+v_0K_S)=K_S+k_0u_0\Big (A+ \frac {k_0v_0-1}{k_0u_0} K_S\Big).
 \]
 Note that $A+ \frac {k_0v_0-1}{k_0u_0} K_S$ is ample, because so is $A$ and moreover
 \[
 0< \frac {k_0v_0-1}{k_0u_0}=\frac {v_0}{u_0}-\frac 1 {k_0u_0}\leqslant \frac {v_0}{u_0}-\frac 1 {3u_0}<r_A.
 \]
 Hence, by Kodaira Vanishing Theorem, we have $h^ i(M)=0$, for $1\leqslant i\leqslant 2$. Therefore we find
 \[
 h^ 0(M)=\chi(M)=\chi(\O_S(k_0u_0A+k_0v_0K_S))=P(ku_0,kv_0)\neq 0.
 \]
 This implies that we found
 \[
 r_0:=\frac {v_0}{u_0}>r_A\quad \mbox{such that} \quad h^ 0(k_0u_0(A+r_0K_S)),
 \]
 as needed.  \end{proof}
 
 \subsection{Zariski's Lemma} Before the proof of the Base Point Freeness Theorem, we need an important preliminary. 
 
Suppose $S$ is a surface, $C$ is a smooth, irreducible curve and $f: S\to C$ a morphism with connected fibres. This is a called a \emph{fibration} or a \emph{pencil} on $S$ and the fibres of $f$ are called the \emph {curves of the fibration or pencil}, $C$ is the \emph{base} of the fibration, and the fibration or pencil is said to be \emph{over} $C$. If $C\cong \PP^ 1$, the pencil is \emph{rational} and the curves of the pencil are \emph{linearly equivalent}. Otherwise the pencil is called \emph{irrational}. If the general curve of the fibration is a $\PP^1$ [resp. a curve of genus 1], then the fibration is called \emph{rational} [resp. \emph{elliptic}].  

If $F$ is a curve of a pencil one has $F^2=0$. If $F$ is a general curve of a pencil, then $F$ is smooth and irreducible. There could be however singular and even reducible curves of a pencil. 

\begin{lemma}\label{lem:penc} Let $F=\sum_{i=1}^ hn_iF_i$ be a curve of a pencil, with $F_i$ irreducible and distinct curves and $n_i$ positive integers, for $1\leqslant i\leqslant h$. If $h>1$ then one has $F_i^ 2<0$ for all $1\leqslant i\leqslant h$.
\end{lemma}

\begin{proof} One has
\[
n_i^2F_i^ 2=n_iF_i\cdot ( F-\sum_{j\neq i}^ hn_jF_j)=- n_iF_i\cdot \sum_{j\neq i}^ hn_jF_j
\]
and the intersection product on the right hand side is positive because $F$ is connected. 
\end{proof}

\begin{thm}[Zariski's Lemma] \label{thm:zar} Suppose $D$ is a non--zero  curve in $\Div(S)\otimes_\ZZ \QQ$ such that ${\rm Supp}(D)$ is contained in a curve $F$ of a pencil with connected fibres. Then $D^2\leqslant 0$ and equality holds if and only if $D=kF$ with $k\in \QQ$. 
\end{thm}

For the proof, we need the following lemma of linear algebra: 

\begin{lemma}\label{lem:quad} Let $Q$ be a symmetric bilinear form on $\QQ^n$, determined by the matrix $M_Q=(q_{ij})_{1\leqslant i,j\leqslant n}$, and such that:\\
\begin{inparaenum}
\item the annihilator of $Q$ contains a vector ${\bf z}=(z_1,\ldots, z_n)$, with $z_i>0$ for all $1\leqslant i\leqslant n$;\\
\item  $q_{ij}\geqslant 0$ for all $1\leqslant i<j\leqslant n$.
\end{inparaenum}

Then $Q$ is negative semidefinite and $\Phi:=\{{\bf x}: Q({\bf x},{\bf x})=0\}$ is a subspace with dimension equal to the number of connected components of the graph with vertices $\{1,\ldots,n\}$ and edges $(i,j)$ if $q_{ij}>0$. 
\end{lemma}

\begin{proof} Set ${\bf x}=(x_1,\ldots, x_n)$. One has
\begin{equation}\label {eq:sam}
\begin{split}
Q({\bf x},{\bf x})&=\sum_{i=1}^n q_{ii}x_i^2+2\sum_{1\leqslant i<j\leqslant n} q_{ij}x_ix_j\leqslant \sum_{i=1}^n q_{ii}x_i^2+\sum_{1\leqslant i<j\leqslant n} q_{ij}(x_i^2+x_j^2)=\\
&=\sum_{1\leqslant i,j\leqslant n} q_{ij}x_i^2=\sum_{i=1}^n x_i^2\sum_{j=1}^n q_{ij}.
\end{split}
\end{equation}
Up to a base change we may assume that ${\bf z}=(1,\ldots, 1)$, which means that $\sum_{j=1}^n q_{ij}=0$ for all $1\leqslant i\leqslant n$. Thus we have $Q({\bf x},{\bf x})\leqslant 0$ for any ${\bf x}\in \QQ^n$. Moreover by \eqref {eq:sam} one has
$Q({\bf x},{\bf x})= 0$ if and only if $x_i=x_j$ for all $1\leqslant i,j\leqslant n$ such that $q_{ij}>0$. The assertion follows.\end{proof}

\begin{proof}[Proof of Zariski's Lemma] We can write
$F=\sum_{i=1}^ nh_iF_i$ and $D=\sum_{i=1}^ nk_iF_i$, with $F_1\ldots, F_h$ irreducible and distinct curves and with $h_i\in \QQ_{>0}$ and $k_i\in \QQ$ for $1\leqslant i\leqslant n$. Consider the subvector space $W$ of $\Num(S)\otimes_\ZZ\QQ$ generated by the classes of the curves $F_i$ for $1\leqslant i\leqslant n$. The intersection product induces a symmetric bilinear form $Q$ on $W$. Then $F$ belongs to the annihilator of $W$ and $F_i\cdot F_j\geqslant 0$ for $1\leqslant i<j\leqslant n$. We can apply Lemma \ref {lem:quad} to this situation. We deduce that $\Phi:=\{E: E^2=0\}$ is a subspace of dimension equal to the number of connected components of the graph with vertices $\{1,\ldots,n\}$ and edges $(i,j)$ if $F_i\cdot F_j>0$. Since $F$ is connected, the number of such connected components is 1. Therefore $\Phi$ is generated by $F$. So we deduce that $D^ 2\leqslant 0$ and that if $D^ 2=0$ then $D=kF$ with $k\in \QQ$. \end{proof}

 \subsection{Proof of the Base Point Freeness Theorem}  In this section we give the:
 
 \begin{proof} [Proof of the Base Point Freeness Theorem]
 Since $L$ is nef, we have $L^ 2\geqslant 0$. \medskip
 
\noindent {\bf Case (a):} $L^2>0$.\medskip

If $L$ is ample the assertion is clear. So we may assume $L$ is nef but not ample, hence there is an irreducible curve  $E$ such that
\[
0=L\cdot E= A\cdot E+rK_S\cdot E>rK_S\cdot E
\] 
hence $K_S\cdot E<0$. By the Hodge Index Theorem we also have $E^ 2=-1$, hence $E$ is a $(-1)$--curve, and we can contract it via a birational morphism $\mu: S\to S'$ which maps $E$ to a point. Since $L\cdot E=0$, we have that $L':=\mu_*(L)$ is a line bundle. Moreover $L=\mu^ *(L')$ and $L'$ is  nef: indeed, for every curve $C'$ on $S'$, we have
\[
L'\cdot C'=\mu^*(L')\cdot \mu^ *(C')=L\cdot   \mu^ *(C')\geqslant 0.
\]
By considering $L,L',A, K_S, K_{S'}$ as divisors, we have
\[
L'=\mu_*(L)=\mu_*(A)+rK_{S'}
\] 
because $K_S=\mu^*(K_{S'})+E$. Moreover $A':=\mu_*(A)$ is ample. This can be seen by applying Nakai--Moishezon Theorem. In fact $ A'^2=\mu_*(A)^2= A^ 2>0$. Moreover, for every curve $C'$ on $S'$, we have
\[
A'\cdot C'=\mu^*(A')\cdot \mu^ *(C')=A\cdot   \mu^ *(C')> 0.
\] 
If $L'$ is ample, then there is a positive integer $\ell$ such that $\ell L'$ is base point free, hence so is $\ell L$. Otherwise $L'$ is nef but not ample and $L'^2=L^2>0$. So we can repeat the argument. After finitely many steps we arrive at a birational morphism $f: S\to V$, where $V$ is a surface and $S$ is obtained from $V$ with finitely many blowing--ups of points. Moreover we will have an ample line bundle $M$ on $V$ such that $L=f^*(M)$. Then the same argument as above proves the assertion.\medskip

\noindent {\bf Case (b):} $L\equiv 0$.\medskip

We have 
\[
A+rK_S=L\equiv 0 \quad \mbox{hence}\quad -K_S\equiv \frac 1r A\quad \mbox {is ample}.
\]
Similarly, for all integers $\ell$, we have 
\[
\ell L-K_S\equiv  \frac 1r A\quad \mbox {is ample}.
\]
Hence, by Kodaira Vanishing Theorem we have
\[
h^ i(\O_S)=h^ i(K_S-K_S)=0, \quad \mbox {for}\quad 1\leqslant i\leqslant 2
\]
and, for all integers $\ell$ for which $\ell L$ is a line bundle, we have
\[
h^ i(\ell L)=h^ i(\ell L-K_S+K_S)=0, \quad \mbox {for}\quad 1\leqslant i\leqslant 2.
\]
Then
\[
h^0(\ell L)=\chi(\ell L)=\chi+\frac 12 \ell L\cdot (\ell L-K_S)=\chi=h^ 0(\O_S)=1
\]
so $\ell L$ is effective, but also numerically equivalent to 0. This implies that it is 0, and the assertion follows.\medskip

\noindent {\bf Case (c):} $L^ 2=0$ and $L\not\equiv 0$.\medskip

So there is an irreducible curve $C$ such that $L\cdot C>0$. We claim that $L\cdot K_S<0$. Indeed, take $h\gg 0$ so that $hA-C$ is effective. Then
\[
L\cdot (hA)=L\cdot (hA-C)+ L\cdot C\geqslant L\cdot C>0
\]
so that $L\cdot A>0$. But then
\[
0=L^ 2=L\cdot (A+rK_S)=L\cdot A+rL\cdot K_S>rL\cdot K_S
\]
proving the claim. 

For all integers $\ell$ we have
\[
\ell L-K_S=\ell L - \frac 1r (A-L)= \frac 1r A+\frac {\ell r-1}r L
\]
and, if $\ell\gg 0$ this is ample, because is the sum of an ample and a nef divisor class. Therefore, for each $\ell\gg0$ such that  $\ell L$ is a line bundle, we again have $h^1(\ell L)=h^ 2(\ell L)=0$. Hence
\[
h^ 0(\ell L)=\chi(\ell L)=\chi+\frac 12 \ell L\cdot (\ell L-K_S)=\chi-\frac 12 \ell L \cdot K_S>0.
\]

Finally we write $|\ell L|= |M|+F$, where $ |M|$ is the movable part (which is nef) and $F$ is the fixed part. We have 
\[
0\leqslant M^ 2\leqslant M\cdot (M+F)= \ell L \cdot M\leqslant \ell L\cdot (M+F)=(\ell L)^ 2=0
\]
which implies
\[
M^ 2=M\cdot F=F^ 2=0.
\]
This implies that $|M|$ is \emph{composed with} a pencil, i.e., there is a curve $D$ and a morphism  $f: S\to D$ 
and there is a line bundle $N$ on $C$ such that $M=f^ *(N)$. Moreover $M\cdot F=0$ means that $F$ is contained in a union of fibres of $f$. By Zariski's Lemma, $F^ 2=0$ implies that $F$ is proportional, over $\QQ$, to the sum of full fibres. Then, by taking $\ell$ larger and more divisible, one may achieve the situation in which $\ell L$ is base points free, as wanted. \end{proof}

\begin{remark}\label{rem:pap} In Case (c) of the above proof we eventually have that the system determined by a high multiple of $L$ is composed with a pencil whose general curve we denote by $F$.  One has $L\cdot F=0$ Moreover, since $K_S\cdot L<0$, we have $K_S\cdot F<0$ and $F^ 2=0$, therefore the general curve $F$ of the pencil is isomorphic to $\PP^1$. 
\end{remark}

\subsection{Boundedness of denominators}\label{ssec:bound} In this section we prove the following:

\begin{corollary}\label{cor:bound} Same hypotheses as in the Rationality Theorem and set $r=r_A$. Then $r=\frac pq$, with ${\rm GCD}(p,q)=1$ and $q\in \{1,2,3\}$. 
\end{corollary}

\begin{proof} We consider the three cases as in the proof of the Base Point Freeness Theorem.\medskip

{\bf Case (a):} $L^ 2>0$. \medskip

In this case there is a $(-1)$--curve $E$ such that  
\[
0=L\cdot E=(A+rK_S)\cdot E=A\cdot E-r
\]
hence $q=1$ in this case.\medskip

{\bf Case (c):} $L^ 2=0$ and $L\not\equiv 0$. \medskip

In this case by Remark \ref {rem:pap} there is an irreducible curve $F$ such that $F^ 2=0$, $K_S\cdot F=-2$ and
\[
0=L\cdot F= (A+rK_S)\cdot E=A\cdot E-2r
\]
hence $q=2$ in this case.\medskip

{\bf Case (b):} $L\equiv 0$. \medskip

Assume first that $\rho>1$. Then we can find an ample divisor $A'\not\equiv A$ and we have $L'=A'+r'K_S$ as in the Rationality Theorem. If $L'\equiv 0$, then we have $rA'\equiv -rr'K_S\equiv r'A$, a contradiction. Hence for $L'$ either Case (a) or Case (c) occur. In the former case there is on $S$ a $(-1)$--curve $E$, in the latter a curve $F$ such that $F\cdot K_S=-2$. Since $L\equiv 0$, hence $L\cdot E=0$ or $L\cdot F=0$, we find again $q=1$ in the former case, $q=2$ in the latter. 

Next we assume $\rho=1$.  We can choose an ample generator $H$ of $\Num(S)$, and we have $-K_S\equiv hH$. We claim that $h\in \{1,2,3\}$. Indeed, if $h>3$, for $x=1,2,3$ we have
$h^ 0(K_S+xH)=\chi(K_S+xH)=0$, because $h^ i(K_S+xH)=0$ for $1\leqslant i\leqslant 2$ by Kodaira Vanishing Theorem and $K_S+xH=(x-h)H$ with $x-h<0$. This implies that $P(x):=\chi(K_S+xH)$ is a polynomial of degree 2 which is identically 0, having the three roots $x=1,2,3$. This is clearly a contradiction. Finally, we have $A=kH$ with $k>0$ and for any curve $C$ we have
\[
0=C\cdot L=A\cdot C+rK_S\cdot C=(k-rh)H\cdot C
\]
hence $r=\frac kh$, proving the assertion. \end{proof}

\begin{remark}\label{rem:bap} In Case (b) of the above proof we have $-K_S\equiv \frac 1r A$ is ample. Therefore, by Kodaira Vanishing Theorem we have $h^1(\O_S)=h^2(\O_S)=0$. In particular we have $b_1=0$. If  $\rho=1$ we have also $b_2=1$. Moreover, being $-K_S$ ample, we have also $P_2=0$. Then $S=\PP^ 2$ by the characterization of $\PP^2$ (see Theorem \ref {thm:charp2}). 
\end{remark}

\subsection{Proof of the Extremal Contraction Theorem}\label{ssec:ECT} Now we can give the:

\begin{proof}[Proof of the Extremal Contraction Theorem] By keeping the above notation, we have that the linear system $|\ell L|$, with $\ell\gg 0$ and highly divisible, is base point free. Hence it defines a morphism $\varphi_{|\ell L|}: S\to V$, where $V$ is a projective variety of dimension $\dim(V)\leqslant 2$. \medskip

{\bf Case $\dim(V)=2$.} We have $L^ 2>0$, so we are in Case (a) of the proof of the Base Point Freeness Theorem. As we saw, in this case there is a $(-1)$--curve $E$ such that $L\cdot E=0$. The contraction of $E$ is an extremal contraction, as required.\medskip

{\bf Case $\dim(V)=1$.} We can assume that $V$ is smooth. By Stein Factorization we can also assume that the fibres $F$ of $S\to V$ are connected. We have $F\cdot L=0$ and this implies, as we saw above, that $F\cdot K_S<0$. Hence  $F\cdot K_S=-2$ and the general fibre $F$ is smooth and rational. If all fibres are irreducible, then $S\to V$ is the required extremal contraction. Otherwise there is a fibre $F=\sum_{i=1}^h n_iF_i$. If $h=1$, then $n_1>1$ and we have $F_1^2=0$ and $-2=n_1K_S\cdot F_1$, so that $n_1=2$ and $K_S\cdot F_1=-1$, a contradiction, since by Adjunction Formula $F_1\cdot(K_S+F_1)$ is even. If $h\geqslant 2$, then $F_i^2<0$ for $1\leqslant i\leqslant h$. Moreover, since $K_S\cdot F<0$, there is an $i$ with $1\leqslant i\leqslant h$ such that $K_S\cdot F_i<0$. Then $F_i$ is a $(-1)$--curve and its contraction is an extremal contraction. \medskip

{\bf Case $\dim(V)=0$.} In this case we are in Case (b) of the proof of the Base Point Freeness Theorem. As we saw in the proof of the Boundedness of Denominators, then:\\
\begin{inparaenum}
\item [$\bullet$] either there is a $(-1)$--curve $E$ on $S$, and its contraction is an extremal contraction;\\
\item [$\bullet$] or there is no $(-1)$--curve but there is a morphism $f: S\to C$, with $C$ a smooth curve, with irreducible, smooth and rational fibres, and $f: S\to C$ is an extremal contraction;\\
\item [$\bullet$] or $\rho=1$, in which case $S\to V$ is an extremal contraction as required. 
\end{inparaenum}
\end{proof}

\section{The Cone Theorem}
If $S$ is a surface, for any class $H\in \NS(S)$ we set $\NE(S)_{H\geqslant 0}$ for the cone of classes $Z\in \NS(S)$, such that $K_S\cdot H\geqslant 0$. Similar notations $\NE(S)_{H> 0}$, $\NE(S)_{H\leqslant 0}$, $\NE(S)_{H< 0}$ have similar meaning. 

\begin{thm}[The Cone Theorem] \label{thm:cone} Let $S$ be a surface. Then
\[
\NE(S)=\NE(S)_{K_S\geqslant 0}+\sum_{\ell\in \mathfrak L} R_\ell
\]
where:\\
\begin{inparaenum}
\item [$\bullet$]  $\ell$ varies in a countable set $\mathfrak L$;\\
\item  [$\bullet$] for each $\ell\in \mathfrak L$, $R_\ell$ is a ray in $\NE(S)_{K_S< 0}$, generated by the class of a smooth rational curve $C_\ell$ such that $0<-C\cdot K_S\leqslant 3$;\\
\item  [$\bullet$] for each $\ell\in \mathfrak L$, there is a nef line bundle $L_\ell$ such that $R_\ell=\NE(S)\cap L_\ell^ \perp$, hence $R_\ell$ is an \emph{extremal ray};\\
\item  [$\bullet$] the rays $R_\ell$ are \emph{discrete} in $\NE(S)_{K_S> 0}$, i.e., for any ample divisor $H$ and for all $\varepsilon>0$ there are only finitely many $R_\ell$s, $R_1,\ldots, R_n$ in the cone $\NE(S)_{(K_S+\varepsilon H)<0}$ and 
\[
\NE(S)=\NE(S)_{(K_S+\varepsilon H)\geqslant 0}+\sum_{i=1}^ nR_i
\]
hence $\NE(S)\cap \NE(S)_{(K_S+\varepsilon H)<0}$ is polyhedral;\\
\item  [$\bullet$] for each $\ell\in \mathfrak L$ there is an extremal contraction contracting the curves with class in $R_\ell$ and viceversa, any extremal contraction is of this type. 
 \end{inparaenum}
\end{thm}

\begin{proof}
The proof will be divided in various steps.

\subsection{Step 1 of the proof} In this step we prove the following:

\begin{lemma}\label{lem:upty}  Let $M$ be a nef divisor class which is not ample, so that $\NE(S)\cap M^ \perp \neq \{0\}$. Assume that there is a face  
\[
F:=\NE(S)_{K_S<0}\cap M^ \perp \neq \{0\}
\]  
of $\NE(S)_{K_S<0}$. Then there is some nef divisor class $N$ such that
\[
R:=\NE(S)_{K_S<0}\cap N^ \perp\subseteq \NE(S)_{K_S<0}\cap M^ \perp
\] 
is an extremal ray. 
\end{lemma} 

Before proving this lemma, we need a lemma of linear algebra:

\begin{lemma}\label{lem:LA} Let $W$ be a $\RR$--vector space of dimension $\rho$. Let $v_1,\ldots, v_\rho$ be a basis of $W$, let $v\in W$ and $\lambda_1,\ldots, \lambda_\rho\in \RR$. Then
\[
\dim (\langle v_i+\lambda_iv\rangle_{1\leqslant i\leqslant \rho})\geqslant \rho-1.
\]
\end{lemma} 
\begin{proof} Suppose we have a relation of the form
\begin{equation}\label{eq:cap}
\sum_{i=1}^ \rho a_i(v_i+\lambda_iv)=0
\end{equation}
Write $v=\sum_{i=1}^ \rho \alpha_iv_i$, so that \eqref {eq:cap} becomes
\[
\sum_{i=1}^ \rho \Big (a_i+ (\sum_{j=1}^ \rho a_j\lambda_j) \alpha_i\Big)v_i=0,
\]
hence we have
\[
a_i+ (\sum_{j=1}^ \rho a_j\lambda_j) \alpha_i=0, \quad \mbox{for}\quad 1\leqslant i\leqslant \rho.
\] 
This is a linear system in the $a_i$s with matrix
\[
\begin{pmatrix}
1+\lambda_1\alpha_1&\lambda_2\alpha_2 & \ldots & \lambda_\rho\alpha_\rho\\
\lambda_1\alpha_1&1+\lambda_2\alpha_2 & \ldots & \lambda_\rho\alpha_\rho\\
\dots&\ldots&\dots&\ldots\\
\lambda_1\alpha_1&\lambda_2\alpha_2 & \ldots & 1+\lambda_\rho\alpha_\rho
\end{pmatrix}
\]
By subtracting the last row from the others, we find the matrix
\[
\begin{pmatrix}
1&0 & \ldots &0 & -1\\
0&1 & \ldots &0 & -1\\
\dots&\ldots&\dots&\ldots&\ldots\\
0&0 & \ldots &1 & -1\\
\lambda_1\alpha_1&\lambda_2\alpha_2 & \ldots &\lambda_{\rho-1}\alpha_{\rho-1}& 1+\lambda_\rho\alpha_\rho
\end{pmatrix}
\]
whose rank is at least $\rho-1$. The assertion follows.\end{proof}

\begin{proof}[Proof of Lemma \ref {lem:upty}] The assertion is trivial if $\dim(F)=1$, so we assume $\dim(F)>1$. 

Fix an ample divisor class $B$, a non--negative rational number $\nu$, and consider $r_{\nu M+B}$, so that
$\nu M+B+r_{\nu M+B}K_S$ is nef and not ample. We have:\\
\begin{inparaenum}
\item [$\bullet$] $r_{\nu M+B}$ is non--decreasing in $\nu$. Indeed, if $\nu'>\nu$ we have
\[
(\nu'M+B)+r_{\nu M+B}K_S=(\nu'-\nu)M+\Big (\nu M+B+r_{\nu M+B}K_S\Big)
\]
which is nef, because so are $(\nu'-\nu)M$ and $\nu M+B+r_{\nu M+B}K_S$. This proves that $r_{\nu' M+B}\geqslant r_{\nu M+B}$;\\
\item [$\bullet$] $r_{\nu M+B}$ is bounded. Indeed, if we take $Z\in F-\{0\}$, we have 
\[
0\leqslant Z\cdot (\nu M+B+r_{\nu M+B}K_S)=Z\cdot (B+r_{\nu M+B}K_S)
\]
hence
\[
-(K_S\cdot Z) r_{\nu M+B}\leqslant B\cdot Z, \quad \mbox{thus} \quad r_{\nu M+B}\leqslant -\frac{B\cdot Z}{K_S\cdot Z}
\]
because $-K_S\cdot Z>0$;\\
\item [$\bullet$] the denominators of $r_{\nu M+B}$ are bounded. 
\end{inparaenum} 

This implies that there is a $\nu_B$ such that $r_{\nu M+B}$ stabilizes to a fixed $r_{M,B}$ for $\nu\geqslant \nu_B$. Set 
\[
M_{B,\nu}:=\nu M+B+r_{M,B}K_S \quad \mbox{for}\quad \nu\geqslant \nu_B
\]
which is nef but not ample. Then:\\
\begin{inparaenum}
\item [$\bullet$] we have 
$$\NE(S)_{K_S<0}\cap M_{B,\nu}^ \perp\neq \{0\}\quad \mbox{ for} \quad\nu\geqslant  \nu_ B.$$
Indeed, since $\nu M+B$ is ample and $M_{B,\nu}$ is nef but not ample, there is $Z\in \NE(S)-\{0\}$ such that
\[
0=Z\cdot M_{B,\nu}=Z\cdot (\nu M+B)+r_{M,B}(K_S\cdot Z)>r_{M,B}(K_S\cdot Z)
\]
hence $K_S\cdot Z<0$;\\
\item [$\bullet$] for $\nu\geqslant  \nu_ B$ one has
\[
\{0\}\neq \NE(S)_{K_S<0}\cap M_{B,\nu}^ \perp\subseteq \NE(S)_{K_S<0}\cap M^ \perp=F.
\]
In fact, if $Z\in \NE(S)_{K_S<0}\cap M_{B,\nu}^ \perp$ and since $M_{B,\nu}+B+r_{M,B}K_S$ is nef, one has
\[
0=Z\cdot M_{B,\nu}=Z\cdot (\nu_BM+B+r_{M,B}K_S)+(\nu-\nu_B)Z\cdot M\geqslant (\nu-\nu_B)Z\cdot M\geqslant 0
\]
thus $M\cdot Z=0$;\\
\item [$\bullet$] if $\dim(F)\geqslant 2$, we can find an ample divisor class $B$ such that the face $\NE(S)_{K_S<0}\cap M_{B,\nu_B+1}^ \perp$ has smaller dimension. In fact we can choose $B_1,\ldots, B_\rho$ ample divisor classes which form a basis of $\NS(S)$. Suppose that
\[
\NE(S)_{K_S<0}\cap M_{B_i,\nu_{B_i}+1}^ \perp=F \quad \mbox{for}\quad 1\leqslant i\leqslant \rho.
\]
This implies that all hyperplanes $(B_i+r_{M,B_i}K_S)^ \perp$ contain the face $F$. Then apply Lemma \ref {lem:upty} with $W=\NS(S)^ \vee$, $v_i$ the linear map 
\[
v_i: D\in \NS(S)\to B_i\cdot D\in \RR,
\]
$\lambda _i= r_{M,B_i}$ for $1\leqslant i\leqslant \rho$, and $v$ is the linear map
\[
v: D\in \NS(S)\to K_S\cdot D\in \RR,
\]
to get a contradiction. In conclusion, there is an $i$ such that $1\leqslant i\leqslant \rho$ and 
\[
\{0\}\neq \NE(S)_{K_S<0}\cap M_{B_i,\nu}^ \perp \varsubsetneq F
\]
as wanted.
\end{inparaenum}

By iterating this argument  the assertion follows.\end{proof} 

\subsection{Step 2 of the proof} \label{ssec:step2} In this step we prove the:

\begin{lemma}\label{lem:step2} Consider the family $\{R_\ell\}_{\ell\in \mathfrak L}$ of extremal rays such that for all $\ell\in \mathfrak L$ there is a nef divisor class $L_\ell$ such that
\begin{equation}\label{eq:set}
\NE(S)_{K_S<0}\cap L_\ell^ \perp=\NE(S)\cap L_\ell^ \perp=R_\ell
\end{equation}
Then
\[
\NE(S)= \NE(S)_{K_S\geqslant 0}+ \overline{\sum_{\ell\in \mathfrak L}R_\ell}.
\]
\end{lemma}

\begin{proof} One has $\NE(S)_{K_S\geqslant 0}+ \overline{\sum_{\ell\in \mathfrak L}R_\ell}\subseteq \NE(S)$. Suppose the inclusion is strict, so that we can find a $Z\in \NE(S) - \Big (\NE(S)_{K_S\geqslant 0}+ \overline{\sum_{\ell\in \mathfrak L}R_\ell}\Big)$. Then we can find a divisor class $D$ such that $D\cdot Z<0$ whereas 
$\NE(S)_{K_S\geqslant 0}+ \overline{\sum_{\ell\in \mathfrak L}R_\ell}\subseteq \NS(S)_{D>0}$.

Note that $D^\perp\cap K_S^\perp\cap \NE(S)=\emptyset$. This implies that we can find some negative rational number $a$ such that $A:=D+aK_S$ is an ample divisor class. By the Rationality Theorem we find a nef, not ample, divisor class
\[
M:=A+rK_S=D+(a+r)K_S, \quad \mbox {with}\quad r\in \QQ_{>0}
\]
such that $\NE(S)\cap M^ \perp \neq \{0\}$. Actually, by an argument we already made, we have that 
$\NE(S)_{K_S<0}\cap M^ \perp \neq \{0\}$. Then by Step 1, there is some extremal ray $R\subseteq 
\NE(S)_{K_S<0}\cap M^ \perp$. Now notice that $a+r<0$, because  
\[
0\leqslant M\cdot Z= D\cdot Z+(a+r)K_S\cdot Z<(a+r)K_S\cdot Z
\]
and $K_S\cdot Z<0$. Next we denote by $C$ a generator of the ray $R$, so that
\[
0=M\cdot C=D\cdot C+(a+r)K_S\cdot C
\]
hence, being $a+r<0$ and $K_S\cdot C<0$, we deduce that also $D\cdot C<0$. This is a contradiction, since 
$\overline{\sum_{\ell\in \mathfrak L}R_\ell}\subseteq \NS(S)_{D>0}$. \end{proof}

\subsection{Step 3 of the proof}\label{ssec:step3}

Next we prove the:

\begin{lemma}\label{ref:pap} Consider the family $\{R_\ell\}_{\ell\in \mathfrak L}$ of extremal rays as in the statement of Lemma \ref {lem:step2}. Then $\{R_\ell\}_{\ell\in \mathfrak L}$ is discrete in the half--space
$\NS(S)_{K_S<0}$. Therefore
\[
\NE(S)= \NE(S)_{K_S\geqslant 0}+ {\sum_{\ell\in \mathfrak L}R_\ell}.
\]
\end{lemma}

\begin{proof} For all $\ell\in \mathfrak L$ there is a nef divisor class $L_\ell$  such that \eqref {eq:set} holds. Then, as we saw in the proof of Lemma \ref {lem:upty}, for every ample divisor class $A$ one has
\[
\{0\}\neq \NE(S)_{K_S<0}\cap (L_\ell)_{A,\nu_A+1}^ \perp\subseteq \NE(S)_{K_S<0}\cap (L_\ell)^ \perp
\]
hence
\[
R_\ell=\NE(S)_{K_S<0}\cap (L_\ell)_{A,\nu_A+1}^ \perp.
\]
Let $C_\ell$ a generator of $R_\ell$. The $C_\ell\cdot (L_\ell)_{A,\nu_A+1}=0$ implies that 
\[
(A+r_{L_\ell,A}K_S)\cdot C_\ell=0 \quad \mbox{hence}\quad r_{L_\ell,A}=-\frac {A\cdot C_\ell}{K_S\cdot C_\ell}. 
\]
Now fix an ample divisor class $H$ and an $\varepsilon>0$ and look at the rays $R_\ell$ such that $(K_S+\varepsilon H)\cdot C_\ell<0$. This means that
\[
-\frac {H\cdot C_\ell} {r_{L_\ell,H}}=K_S\cdot C_\ell<-\varepsilon H\cdot C_\ell \quad \mbox{hence}\quad r_{L_\ell,H}<\frac 1\varepsilon.
\]
Since the denominators of $r_{L_\ell,H}$ are bounded, we have only finitely many choices for $r_{L_\ell,H}$, so for $L_\ell$, thus for $R_\ell$.\end{proof}

\subsection{Step 4 of the proof: the Contraction Theorem} \label{ssec:step4}
In this section we prove the:

\begin{thm}[The Contraction Theorem]\label{prop:contract} Consider the family $\{R_\ell\}_{\ell\in \mathfrak L}$ of extremal rays as in the statement of Lemma \ref {lem:step2}. For each $\ell\in \mathfrak L$, there is an extremal contraction (the \emph{contraction of} $R_\ell$)
\[
\varphi_\ell: S\to V_\ell
\]
such that for any curve $C$ contracted by $\phi_\ell$ to a point, one has that $C\in R_\ell$.

Conversely, any extremal contraction is a contraction of an extremal ray as above.
\end{thm}

\begin{proof} For all $\ell\in \mathfrak L$ there is a nef divisor class $L_\ell$  such that \eqref {eq:set} holds. Then for each $\ell\in \mathfrak L$  we have $L_\ell^\perp \cap K_S^\perp\cap \NE(S)=\emptyset$. Hence, if for a given $\ell\in \mathfrak L$ we set $R=R_\ell$ and $L=L_\ell$, we have that $A=L -r K_S$ is ample, if $r\in\QQ_{>0}$ is small enough, and then $r=r_A$. So, the assertion follows by the Extremal Contraction Theorem.

Conversely, let $\varphi: S\to V$ be an extremal contraction. Let $H$ be an ample divisor on $V$. Then $L:=\phi^ *(H)$ is nef, but not ample, and, by the very definition of an extremal contraction, one has
\[
\NE(S)\cap L^ \perp=\NE(S)_{K_S<0}\cap L^ \perp=R
\] 
where $R$ is generated by any curve $C$ contracted to a point by $\varphi$.\end{proof} 

\subsection{Step 5: of the proof}\label{ssec:step5} In this step we prove that:

\begin{lemma}\label{lem:mm} Consider the family $\{R_\ell\}_{\ell\in \mathfrak L}$ of extremal rays as in the statement of Lemma \ref {lem:step2}. For each $\ell\in \mathfrak L$, there is a smooth rational curve $C$ in $R_\ell$
such that 
\[
1\leqslant -C\cdot K_S\leqslant 3.
\] 
\end{lemma}
\begin{proof} Let $\varphi: S\to V$ be the extremal contraction which contracts the curves in $R$. If $\dim(V)=2$ we saw in the proof of the Extremal Contraction Theorem that the only irreducible curve contracted by $\varphi$ is a $(-1)$--curve $E$, for which $E\cdot K_S=-1$. If $\dim(V)=1$ we saw that the only irreducible curves contracted by $\varphi$ are the fibres $F$ of $\varphi$ which are smooth rational with $F^ 2=0$ and $F\cdot K_S=-2$. Finally, if $\dim(V)=0$ then  $\rho=1$. Then by the argument in Case (b) if the proof of the Boundedness of Denominators and by Remark \ref {rem:bap}, we see that $S\cong \PP^2$ and, if $C$ is a line one has $L\cdot K_S=-3$ and $R$ is generated by $C$. \end{proof}

The previous steps prove the Cone Theorem. \end{proof}

\section{The minimal model programme}

The minimal model programme is an algorithm which, given a surface, returns one of the two items: either a surface $S$ birational to the given one with $K_S$ nef which we will call a \emph{strong minimal model} of the given surface, or a \emph{Mori fibre space}, to be shortly defined, again birational to the given surface.

Given a surface $S$, the algorithm works as follows: \medskip

{\bf Step 1}:  if $K_S$ is nef, the algorithm  stops and we say that $S$ is \emph{strongly minimal}. Note that $S$ is in particular minimal, because a $(-1)$--curve $E$ would be such that $K_S\cdot E=-1$, impossible if $K_S$ is nef.\medskip

{\bf Step 2}: if $K_S$ is not nef, then there is an extremal contraction $\varphi: S\to V$. If $\dim(V)<2$, then we say that $\varphi: S\to V$ (or simply $S$ if no confusion arises) is a \emph{Mori fibre space} and the algorithm stops. If $\dim(V)=2$, then we go to Step 1. \medskip

After a finite number of steps the algorithm produces:\\
\begin{inparaenum}
\item [$\bullet$] either a birational morphism $f: S\to S'$, composed of blow--downs of $(-1)$--curves, with $K_{S'}$ nef, so that $S'$ is a strong minimal model of $S$;\\
\item [$\bullet$] or a birational morphism $f: S\to S'$, composed of blow--downs of $(-1)$--curves, and a morphism $\varphi: S'\to V$  which is a Mori fibre space. Precisely, if $\dim(V)=1$, then $S'$ is a scroll over the curve $V$, if $\dim (V)=0$, then $S'\cong \PP^ 2$. 
\end{inparaenum}

Note that the algorithm has a certain amount of freedom in its application, so that its outcome could be not unique. However the outcome is uniquely determined up to birational transformation. We will soon see that different outcomes applied to the same surface cannot be a strong minimal model and a Mori fibre space. However, if the algorithm produces a Mori fibre space, this could be not unique, as the Example \ref {ex:2} shows. 
By contrast, one has the:

\begin{thm} [Uniqueness of the Strong Minimal Model] \label{prop:ssmmmm} Let $S$ be a strong minimal model. Then if $S'$ is a smooth, projective surface and $f: S'\dasharrow S$ a birational map, then it is a morphism.

In particular, if $S,S'$ are both strong minimal models and $f: S'\dasharrow S$ is a birational map, then it is an isomorphism.
\end{thm}

\begin{proof} We resolve the indeterminacies of $f$
\begin{equation*}
\xymatrix{ 
&X\ar_{p}[d]\ar^{q}[dr] &\\
&S'\ar@{-->}^{f}[r] &  S&
}
\end{equation*}
where $p,q$ are sequences of blow--ups. We may assume that in the above diagram the number of blow--ups occurring in $p$  is minimal.

If $p$ is an isomorphism, we are done. If not, by the ramification formula we have
\[K_X=q^*(K_S)+R=p^*(K_{S'})+R'\] with $R$ and $R'$ effective divisors. There is a $(-1)$--curve $E\leq R'$, i.e.,  the exceptional divisor of the last blow--up occurring in $p$, and
\[
-1=E\cdot K_X=E\cdot (q^*(K_S)+R)\geq E\cdot R
\]
because, $K_S$ being nef, so is $q^*(K_S)$. Hence $E\leq R$, hence $E$ is contracted by both $p$ and $q$. This means that, if $\pi: X\to X'$ is the blow--down of $E$, $p$ and $q$ both factor through $\pi$, i.e., there is a diagram
\begin{equation*}
\xymatrix{ 
&X'\ar_{p'}[d]\ar^{q'}[dr] &\\
&S'\ar@{-->}^{f}[r] &  S&
}
\end{equation*}
with $p=p'\circ p$ and $q=q'\circ \pi$. This is a contradiction, since $p'$ is the composition of one blow--up less than $p$.\end{proof}

\begin{remark}\label{rem:kk} If the result of the minimal model programme applied to $S$ is a Mori fibre space $\varphi:S'\to V$, then we have $\kappa(S)=-\infty$. Indeed, $S$ and $S'$ are birationally equivalent, so that $\kappa(S)=\kappa(S')$. Moreover either $S'=\PP^ 2$ or $S'\to V$ is a scroll. In either case there are moving curves $C$ on $S'$ such that $C\cdot K_{S'}<0$, which yields $P_n(S')=0$ for all $n\in \NN$, so $\kappa(S')=-\infty$. 
\end{remark}

\section{Castelnuovo's Rationaly Criterion}

In this section we prove the fundamental:

\begin{thm}[Castelnuovo's Rationaly Criterion] \label{thm:crt} A surface $S$ is rational if and only if $q=P_2=0$.
\end {thm}

\begin{proof} Recall that $q$ and $P_2$ are birational invariants.  They are both zero for $\PP^2$. Indeed, if $L$ is the class of a line, one has $K_{\PP^ 2}=-3L$ and $L$ is very ample. Hence all plurigenera vanish and moreover $q(\PP^ 2)=h^ 1(\O_{\PP^2})= h^ 1(3L+K_{\PP^ 2})=0$ by Kodaira Vanishing Thoerem. Hence $q$ and $P_2$ are zero for all rational surfaces.

To prove the converse, we will prove that
\begin{equation}\label{eq:ccc}
q=P_2=0 \quad \mbox{implies} \quad  K_S \quad \mbox{not nef}.
\end{equation}
Taking this for granted, the theorem follows. Indeed, by applying the minimal model programme to $S$, since $q$ and $P_2$ are birational invariants, we never reach a strong minimal model, so we eventually reach a Mori fibre space $f: S'\to V$, with $q(S')=P_2(S')=0$. If $\dim(V)=0$ we know that $S'\cong \PP^2$ proving the theorem. If  $V$ is a smooth curve, one has $g(V)=0$. Indeed, one has an injection $f^ *: H^ 0(\omega_V^1)\to H^0(\Omega^ 1_S)=\{0\}$, so that $g(V)=h^ 0(\omega_V^1)=0$. Then $S'$ is a scroll over $\PP^1$, hence it is birational to $\PP^2$, so that the assertion follows.

To prove \eqref {eq:ccc}, we note that $P_2=0$ implies $p_g=0$ and, since $q=0$, then $p_a=0$, so that $\chi=1$. Hence, using Serre Duality and Riemann--Roch Theorem, we get
\[
\begin {split}
h^ 0(-K_S)&=h^ 0(-K_S)+h^ 0(2K_S)=h^ 0(-K_S)+h^ 2(-K_S)\geqslant\\
&\geqslant  h^ 0(-K_S)-h^ 1(-K_S)+h^ 2(-K_S)=\chi(-K_S)=\\
&=\chi+\frac {(-K_S)(-2K_S)}2=1+K_S^2.
\end{split}
\]
We now argue by contradiction and assume $K_S$ is nef. Then we have $K_S^ 2\geqslant 0$, hence $h^ 0(-K_S)>0$. If $H$ is an ample divisor on $S$ we have then $-K_S\cdot H>0$, hence $K_S\cdot H<0$, a contradiction.  This proves \eqref {eq:ccc} and the theorem.\end{proof}

\begin{corollary}\label{cor:unir} Any unirational surface is rational.
\end{corollary}

\begin{proof} If $S$ is unirational there is a rational surface $S'$ and a dominant rational map $f: S'\dasharrow S$. By elimination of indeterminacies, we may assume $f$ is a morphism. Then we claim that $q(S)=P_2(S)=0$. Indeed, since we have an injection $f^ *:H^ 0(\Omega^ 1_S)\to H^ 0(\Omega^1_{S'})=\{0\}$, we see that $q(S)=h^ 0(\Omega^ 1_S)=0$. The proof that $P_2(S)=0$ is similar. Hence $S$ is rational. \end{proof}

\section{The Fundamental Theorem of the Classification}\label{sec:ftc}

The main result of this section is the following:

\begin{thm}[Fundamental Theorem of the Classification]\label{sec:ftc}
Let $S$ be a surface. Then:\\
\begin{inparaenum}
\item [(a)] the end result of the minimal model programme applied to $S$ is a strong minimal model if and only if $\kappa(S)\geqslant 0$;\\
\item [(b)]  the end result of the minimal model programme applied to $S$ is a Mori fibre space if and only if $\kappa(S)=-\infty$.
\end{inparaenum}
\end{thm}

Note that (a) and (b) are equivalent. We already observed that if the end result of the minimal model programme applied to $S$ is a Mori fibre space then $\kappa(S)=-\infty$ (see Remark \ref {rem:kk}). This is the same as proving that if $\kappa(S)\geqslant 0$ then the end result of the minimal model programme applied to $S$ is a strong minimal model. So we are left to prove the:

\begin{thm}\label{thm:ftc2} If the end result of the minimal model programme applied to $S$ is a strong minimal model then $\kappa(S)\geqslant 0$.
\end{thm}

To prove this, we need a few preliminary results of independent interest.

\subsection{Castelnuovo--De Franchis' Theorem}\label{ssec:cdf}

\begin{thm}[Castelnuovo--De Franchis' Theorem]\label{thm:cdf} Let $S$ be a surface. Suppose there are two linearly independent holomorphic 1--forms $\omega_1,\omega_2$ on $S$ such that $\omega_1\wedge\omega_2\equiv 0$. Then there is a smooth projective curve $C$, with $g(C)\geqslant 2$, a surjective morphism $f: S\to C$ with connected fibres, and two linearly independent holomorphic 1--forms on $C$ such that $\omega_i=f^*(\alpha_i)$, for $1\leqslant i\leqslant 2$.
\end{thm}

\begin{proof} Since $\omega_1\wedge\omega_2\equiv 0$, there is a non--zero rational function $g$ on $S$ such that $\omega_2=g\omega_1$. Since $\omega_1,\omega_2$ are closed forms, we deduce that $\omega_1\wedge dg\equiv 0$. Hence there is a non--zero rational function $h$ on $S$ such that $\omega_1=hdg$. Hence $dg\wedge dh\equiv 0$. Consider now the rational map
\[
\phi: S\dasharrow \PP^ 2, \quad p\in S\to (g(p),h(p)).
\]
Since $dg\wedge dh\equiv 0$,  the image of $\phi$ is a curve $\Gamma$, with affine equation $P(x,y)=0$. If $\nu: D\to \Gamma$ is the normalization of $\Gamma$, we have a rational map $\psi: S\dasharrow D$. Now, eliminating the indeterminacies and using Stein factorization we can find a commutative diagram
\begin{equation*}\label{eq:K3}
\xymatrix{ 
&S'\ar^{f}[r]\ar_{\pi}[d]&C\ar^{q}[d]\ar^{\bar q}[dr]& \\
&S\ar@{-->}^{\psi}[r] & D\ar^{\nu}[r]& \Gamma \,\,
}
\end{equation*}
where $\pi: S'\to S$ is a birational morphism, $C$ is a smooth curve, $q: C\to D$ is a finite map and $f: S'\to C$  is surjective with connected fibres. Consider the meromorphic 1--forms on $C$ given by $\alpha_1=\bar q^*(xdy)$ and $\alpha_2=\bar q^ *(xydy)$. We have $\pi^ *(\omega_i)=f^ *(\alpha_i)$, for $1\leqslant i\leqslant 2$, hence $\alpha_1$ and $\alpha_2$ are holomorphic as well as  $\pi^ *(\omega_1)$ and $\pi^ *(\omega_2)$. This proves that $g(C)\geqslant 2$. Then the rational map $\pi^ {-1}\circ f: S\dasharrow C$ has no  indeterminacies (see \S \ref {ssec:rm}), so $S=S'$, $\pi={\rm id}$, and the assertion follows. \end{proof}

\begin{lemma}\label{lem:purp} Let $C$ be a reduced curve on a surface $S$. Then 
\begin{equation}\label{eq:piz}
 e(C)\geqslant 2\chi(\O_C)
 \end{equation}
with equality if and only if $C$ is smooth.
\end{lemma}
\begin{proof} First of all if $C$ is a smooth curve then \eqref {eq:piz} holds with equality. In fact, $C$ is the disjoint union of $C_1,\ldots, C_n$ its irreducible components. Then 
$g(C_i)=h^ 1(\O_{C_i})$, and 
\[
e(C_i)=2-2g(C_i)=2\Big (h^0(\O_{C_i})-h^ 1(\O_{C_i})\Big )=2\chi(\O_{C_i}),\quad \text {for}\quad 1\leq i\leq n.
\] Then 
\[
e(C)=\sum_{i=1}^nc(C_i)=2\Big (n-\sum_{i=1}^ng(C_i)\Big)=2\Big (h^0(\O_{C})-h^ 1(\O_{C})\Big )=2\chi(\O_C).
\]

Consider now the general case. Let $\nu: D\to C$ be the normalization of $C$. Consider the diagram
\begin{equation}\label{eq:pup}
\xymatrix{ 
0\ar[r]&\CC_C\ar[r]\ar[d]&\nu_*\CC_D\ar[d]\ar[r]& \epsilon\ar^{f}[d]\ar[r]&0 \\
0\ar[r]&\O_C\ar[r]&\nu_*\O_D\ar[r]& \delta\ar[r]&0\,\,
}
\end{equation}
where $\C_X$ denotes the constant sheaf with stalk $\CC$ on a variety $X$ and the torsion sheaves $\epsilon,\delta$ are defined in such a way that the horizontal sequences are exact. 

The map $f$ is injective. Indeed, this amounts to prove that local sections of $\nu_*\O_D$ which come from both 
$\O_C$ and from $\nu_*\CC_D$, in fact come from $\CC_C$. This is immediate. In fact, the sections in question are local regular functions on $C$ which, when pulled back to $D$, become constant which means they were constant to start with. 

Hence we have $h^ 0(\epsilon)\leqslant  h^ 0(\delta)$. From diagram \eqref {eq:pup} we deduce
\[
e(D)=e(C)+h^ 0(\epsilon), \quad \chi(\O_D)=\chi(\O_C)+h^ 0(\delta)
\]
hence
\[
e(C)=e(D)-h^ 0(\epsilon)=2\chi(\O_D)-h^ 0(\epsilon)=2\chi(\O_C)+2h^ 0(\delta)-h^ 0(\epsilon)
\]
which implies \eqref {eq:piz}. Moreover the equality holds if and only if
\[
0=2h^ 0(\delta)-h^ 0(\epsilon)=h^ 0(\delta)+\Big (h^ 0(\delta)-h^ 0(\epsilon)\Big)
\]
which holds if and only if $h^ 0(\delta)=0$, since $h^ 0(\delta)\geqslant h^ 0(\epsilon)$. On the other hand $h^ 0(\delta)=0$ if and only if $C$ is smooth. \end{proof}

\begin{proposition}\label{prop:cdf} Let $S$ be a surface and $C$ a smooth, projective curve and $f: S\to C$ a surjective morphism, with connected general fibre $F$ (a smooth curve) and $F_1,\ldots,F_h$ the only singular fibres of $f$. Then:\\
\begin{inparaenum}
\item [(i)] one has
\[
e(F_i)\geqslant e(F), \quad \mbox{for}\quad 1\leqslant i\leqslant h;
\]
\item [(ii)] one has
\[
e(S)=e(C)\cdot e(F)+\sum_{i=1}^ h \Big (e(F_i)-e(F)\Big)\geqslant e(C)\cdot e(F).
\]
\end{inparaenum}
\end{proposition}

\begin{proof} (i) Since $f: S\to C$ is flat, then $\chi(\O_{F_i})=\chi(\O_{F})$ for $1\leqslant i\leqslant h$. Since the fibres of $f$ are connected, so that  $h^0(\O_{F_i})=h^0(\O_{F})=1$ for $1\leqslant i\leqslant h$, we have $h^1(\O_{F_i})=h^1(\O_{F})=g(F)$ for $1\leqslant i\leqslant h$. Since $e(F)=2-2g(F)$, we have to prove that
$e(F_i)\geqslant 2-2h^1(\O_{F_i})$ for $1\leqslant i\leqslant h$.

Now notice that for any curve $D$ on a surface, if we denote by $D_{\rm red}$ its (reduced) support  we have $e(D)=e(D_{\rm red})$. If $D$ is connected, by taking into account Lemma \ref {lem:purp}, we have
\[
e(D)=e(D_{\rm red})\geqslant 2\chi(\O_{D_{\rm red}})=2h^0(\O_{D_{\rm red}})-2h^1(\O_{D_{\rm red}})=2-2h^1(\O_{D_{\rm red}}).
\]
Moreover, we have $h^ 1(\O_D)\geqslant h^1(\O_{D_{\rm red}})$, because we have a surjection of sheaves, i.e., the restriction map, $\O_D\to \O_{D|D_{\rm red}}\cong \O_{D_{\rm red}}$. Hence
\[
e(D)\geqslant 2-2h^1(\O_{D_{\rm red}})\geqslant 2-2h^1(\O_{D}).
\]
Applying this to $D=F_i$ for $1\leqslant i\leqslant h$, ends the proof of (i).

(ii) Let $p_1,\ldots,p_h$ be the points such that $F_i$ is the fibre over $p_i$, with $1\leqslant i\leqslant h$. Then $f: S-\cup_{i=1}^ hF_i\to C-\{p_1,\ldots,p_h\}$ is locally topologically a product with fibre $F$. Hence we have
\[
e(S)= e(S-\sum_{i=1}^h F_i)+ e(\sum_{i=1}^h F_i)=(e(C)-h)\cdot e(F)+\sum_{i=1}^ h e(F_i)
\]
whence the assertion follows. \end{proof}

\begin{corollary}\label{cor:cdf} Same hypotheses as in Theorem \ref {thm:cdf}. In addition, assume $K_S$ nef. Then $e(S)\geqslant 0$.
\end{corollary}

\begin{proof} If $K_S$ is nef, then the generic fibre $F$ of $f$ does not have genus 0, hence $e(F)\leqslant 0$. Then, applying Proposition \ref {prop:cdf}, we have
\[
e(S)\geqslant e(F)\cdot e(C)\geqslant 0
\] 
since $e(C)=2-2g(C)\leqslant -2$, being $g(C)\geqslant 2$. 
\end{proof}

\begin{lemma}[Hopf's Lemma]\label{lem:hopf} Let $V, W$ be vector spaces of finite dimension $n,m$ respectively over a field $\mathbb K$. Assume we have a linear map
\[
\phi: \wedge^2 V\to W.
\]
If 
\[
m\leqslant 2n-4
\]
then there are non--zero indecomposable tensors $\omega_1\wedge\omega_2$ in ${\rm Ker}(\phi)$.
\end{lemma}

\begin{proof}
Consider the projective transformation induced by $\phi$
\[
\widetilde \phi: \PP(\wedge^2 V)\dasharrow \PP(W)
\]
whose indeterminacy locus is $\PP({\rm Ker}(\phi))$.  In $ \PP(\wedge^2 V)$ there is the $(2n-4)$--dimensional \emph{Grassmann variety} $\mathbb G$ of lines in $\PP(V)$, which is the locus of equivalence classes of non--zero indecomposable tensors. If $m\leqslant 2n-4$, then
\[
\dim (\PP(W))=m-1 \leqslant 2n-5=\dim(\mathbb G)-1
\]
so that $\widetilde \phi_{|\mathbb G}$ is not finite. This implies that $\mathbb G$ has to intersect the indeterminacy locus $\PP({\rm Ker}(\phi))$ of $\widetilde \phi$. This proves the assertion.
\end{proof}

\begin{corollary}\label{cor:cdf2} Let $S$ be a surface such that
\[
p_g\leqslant 2q-4.
\]
Then there is a morphism $f: S\to C$ as in Castelnuovo--De Franchis' Theorem. 
\end{corollary}

\begin{proof} Consider the obvious linear map
\[ \wedge ^ 2H^ 0(\Omega^ 1_S)\to H^ 0(\Omega^ 2_S)\]
which sends $\sum_{ij}\omega_i\wedge \omega_j$ to the 2--form which is denoted in the same way. An application of Hopf's Lemma implies that there is a non--zero indecomposable tensor $\omega_1\wedge\omega_2$ such that $\omega_1\wedge\omega_2\equiv 0$ as a 2--form. The assertion then follows by Castelnuovo--De Franchis' Theorem. \end{proof}

\begin{thm}[Castelnuovo--De Franchis--Enriques' Theorem]\label{thm:cdfe} If $S$ is a surface with $K_S$ nef, then:\\
\begin{inparaenum}
\item [(i)] $\chi\geqslant 0$ and $e\geqslant 0$;\\
\item [(ii)] if $\kappa=2$, then $\chi >0$ and $K_S^2>0$;\\
\item [(iii)] if $K_S^2>0$ then $\kappa=2$;\\
\item [(iv)] if $\chi>0$ then $\kappa\geq 0$.
\end{inparaenum}
\end{thm}

\begin{proof} (i) First we prove that $e\geqslant 0$. Argue by contradiction, and assume $e<0$. Then
$0>e=2-2b_1+b_2$, hence $2b_1\geqslant 2+b_2\geqslant 3$, which implies $b_1\geqslant 2$. Therefore
$H_1(S,\ZZ)$ is a free abelian group of rank at least 2. Choose a subgroup $G$ of $H_1(S,\ZZ)$ of index $m\geqslant 6$. This lifts to a subgroup $G'$ of the same index of $\pi_1(S)$, so it gives rise to an \'etale cover 
$f: S'\to S$ of order $m$, such that
\[
-6\geqslant me(S)=e(S')=2-2b_1(S')+b_2(S')\geqslant 2-4q(S')+2p_g(S'),
\]
hence $p_g(S')\leqslant 2q(S')-4$. Since $K_{S'}=f^ *(K_S)$, then $K_{S'}$ is nef. Then by Corollary    \ref {cor:cdf2} and Corollary \ref  {cor:cdf} we have $e(S')\geqslant 0$, which implies $e(S)\geqslant 0$, a contradiction.

Then $\chi\geqslant 0$ follows by Noether's Formula, since also $K_S^2\geqslant 0$ because $K_S$ is nef. This proves (i).

(ii) To prove the assertion it suffices to prove that $K_S^ 2>0$, because then $\chi>0$ by Noether's Formula. If $\kappa=2$, then there is an $n\in \NN$ such that the image of the map $\varphi_{|nK_S|}$ is a surface. Set 
\[
|nK_S|= |M|+F
\]
where $F$ is the fixed part and $|M|$ the movable part, which is nef. One has $M^ 2>0$, hence
\[
(nK_S)^ 2=nK_S\cdot F+nK_S\cdot M\geqslant nK_S\cdot M=(F+M)\cdot M\geqslant M^ 2>0
\]
as wanted. 

(iii) Assume  $K_S^ 2>0$. Then for any $n\in \NN$ we have
\begin{equation}\label{eq:olpu}
h^ 0(nK_S)+h^ 2(nK_S)\geqslant \chi(nK_S)=\frac  {m(m-1)}2 K^2_S+\chi. 
\end{equation}
If $n\geqslant 2$, one has $h^ 2(nK_S)=h^ 0((1-n)K_S)=0$. Indeed, if $A$ is ample   and if $n\geqslant 2$, then  $A\cdot  ((1-n)K_S)<0$ because $K_S$ is nef, so $h^ 0((1-n)K_S)=0$. Then \eqref {eq:olpu} implies that $h^ 0(nK_S)$ grows as $n^ 2$ proving that $\kappa=2$ (see Theorem \ref {thm:kod}).  

(iv) The proof is similar to the one of (iii). If $\chi>0$, then from  \eqref {eq:olpu} and $h^ 0((1-n)K_S)=0$ one has $P_n>0$ as soon as $n\geq 2$. \end{proof}

\subsection{The canonical bundle formula for elliptic fibrations}\label{ssec:canb}

Consider the following situation: $S$ is a surface, $C$ is a smooth projective curve, and we have a surjective morphism $f: S\to C$, such that the general fibre $F$ of $f$ is a smooth, irreducible, projective curve of genus $1$, and $S$ is {minimal}. Then $K_S\cdot F=0$ hence $\kappa(S)<2$ and, by Castelnuovo--De Franchis--Enriques' theorem, we have $K_S^2=0$. 
One has:

\begin {thm}\label {thm:canb} In the above setting, one has
\[
K_S=f^ *\Big (K_C\otimes (R^1f_*\O_S)^\vee\Big )\otimes \O_S\Big (\sum_{i=1}^h (m_i-1)F_i\Big),
\]
where $R^1f_*\O_S$ is a line bundle of degree $\chi$ on $C$, and $m_iF_i$, for $1\leqslant i\leqslant h$, are the multiple fibres of $f$. 

As a consequence, there is an $n\in \NN$ and a line bundle $\mathcal K$ on $C$ such that $nK_S=f^*(\mathcal K)$. 

\end{thm}

\begin{proof} Let $x\in C$ be a point. We will denote by $F_x$ the fibre of $f$ over $x$ and, as usual, by $F$ the general fibre of $f$. Let $x_1,\ldots, x_n$ be general points of $C$. Consider  the exact sequence
\[
0\to K_S\to K_S\otimes \O_S(\sum_{i=1}^nF_{x_i})\to \oplus_{i=1}^n \O_{F_{x_i}}\to 0.
\]
The image of $H^0(K_S\otimes\O_S(\sum_{i=1}^nF_{x_i}))$ in $H^0( \oplus_{i=1}^n \O_{F_{x_i}})\cong \CC^n$ has codimension at most $q=h^1(K_S)$, hence
\[
h^0(K_S\otimes\O_S(\sum_{i=1}^nF_{x_i}))\geq	 n+p_g-q
\] 
which is positive when $n\gg 0$. If $D\in  |K_S+\sum_{i=1}^nF_{x_i}|$, we have $D\cdot F=0$. This implies that we can write
\[
K_S=f^*(L)\otimes \O_S(B)
\]
with $L$ a line bundle on $C$ and $B$ a curve contained in a union of fibres of $f$, but not containing any fibre of $f$. Let $B_0$ be any connected component of $B$ and let $F_0$ be the fibre containing it. 

\begin{claim}\label{cl:uff}  $F_0$ is a multiple fibre and $B_0$ is a rational submultiple of $F_0$. i.e., there is an effective divisor $P$ such that
\[
F_0=mP \quad \text{and} \quad B_0=aP
\]
with $0<a <m$. \end{claim}

\begin{proof} [Proof of Claim \ref{cl:uff}] If $B_0,\ldots, B_h$ are the connected components of $B$, we have $B=\sum_{j=1}^h B_j$ and 
\[
0=K_S^2=B^2=\sum_{j=0}^h B^2_j.
\]
On the other hand we have $B_j^2\leq 0$, hence we have 
$B_j^2= 0$ for all $j\in \{0,\ldots, h\}$ and the claim follows by Zariski's Lemma. \end{proof}

So we have
\[
K_S=f^*(L)\otimes \O_S\Big (\sum_{i=1}^h a_iF_i\Big )
\]
with $m_iF_i$ the multiple fibres of $f$ and $0\leq a_i<m_i$.  By the adjunction formula we have
\[
\O_{F_i}=K_{F_i}=\O_S(K_S+F_i)\otimes \O_{F_i} =\O_S((a_i+1)F_i)\otimes \O_{F_i}.
\]
But $\O_S(F_i)\otimes \O_{F_i}$ is torsion, of order $m_i$, therefore we must have $a_i=m_i-1$, for $1\leq i\leq h$. 

 Note that, by the projection formula and relative duality (see Theorem \ref {thm:reldu}), we have
\[
L=L\otimes f_*(\O_S(\sum_{i=1}^h (m_i-1)F_i))=f_*(K_S)=K_C\otimes (R^1f_*\O_S)^\vee.
\]

Finally, by the spectral sequence associated to the map $f$, we have
\[
\begin{split}
&\chi=h^0(\O_S)-h^1(\O_S)+h^2(\O_S)=\\
&=h^0(f_*\O_S)-h^1(f_*\O_S)-h^0(R^1f_*\O_S)+h^1(R^1f_*\O_S)=\\
&=h^0(\O_C)-h^1(\O_C)-h^0(R^1f_*\O_S)+h^1(R^1f_*\O_S)=\\
&=1-g(C)-\chi(R^1f_*\O_S)=1-g(C)-(\deg(R^1f_*\O_S)-g(C)+1)=\\
&=-\deg(R^1f_*\O_S)=\deg (R^1f_*\O_S)^\vee
\end{split}
\]
as desired. \end{proof}

\subsection{Basic lemmas}\label{ssec:basic}

Now we prove a couple of lemmas which will be crucial in what follows. Let $S$ be a surface with $K_S$ nef, $q=1$, $p_g=0$ and $\kappa\leqslant 0$. Consider the Albanese map $\alpha: S\to E$, where $E=\Alb(S)$ is an elliptic curve. Consider the Stein factorization 
\begin{equation*}
\xymatrix{ 
&S\ar_{f}[d]\ar^{\alpha}[dr] &\\
&C\ar^{g}[r] &  E&
}
\end{equation*} 
By the universal property of the Albanese map, $\Alb(C)$ and $E=\Alb(S)$ are isomorphic, hence  $C=E$, hence $\alpha$ has connected fibres. We denote by $g$ the genus of a general fibre of $\alpha$. 

\begin{lemma} \label{cl:one} In the above setting  the Albanese morphism $\alpha: S\to E$ is smooth. \end{lemma}

\begin{proof} First we remark that $\chi=0$ and that, by Castelnuovo--De Franchis--Enriques' Theorem, we have $K_S^2=0$. Then
\begin{equation}\label{eq:park1}
e=12\chi-K_S^2=0=2(b_0-b_1)+b_2=2-4q+b_2=b_2-2, \quad \mbox {hence} \quad b_2=2.
\end{equation}

First we prove that for each fibre $F$ of $\alpha$, $F_{\rm red}$ is irreducible. Indeed, assume the contrary and suppose that a fibre contains two distinct irreducible components $F_1,F_2$. Take $H$ an ample divisor. We claim that the classes of $H, F_1,F_2$ are linearly independent in $H^ 2(S,\ZZ)$, contradicting 
\eqref {eq:park1}. Indeed, assume $hH+f_1F_1+f_2F_2=0$ in $H^ 2(S,\ZZ)$. Then, intersecting with $F$, we find
\[
0=(hH+f_1F_1+f_2F_2)\cdot F=h H\cdot F
\]
whence $h=0$. On the other hand Zariski's Lemma implies that  $f_1F_1+f_2F_2\equiv 0$ if and only if $f_1=f_2=0$. This proves the claim.

Next, for every $x\in E$ we consider the fibre $F_x$, which we write as $F_x=n_xG_x$, with $n_x\geqslant 1$ and $G_x$ irreducible and reduced. 
If $F$ is the general fibre of $\alpha$, by Lemma \ref {lem:purp} we have
\[
\begin{split}
e(G_x)&\geqslant 2\chi(\O_{G_x})=-G_x\cdot (G_x+K_S)=\\
&=-\frac 1{n_x} F\cdot (K_S+\frac 1{n_x}F)=-\frac 1{n_x} F\cdot K_S=\\
&=-\frac 1{n_x} F\cdot (K_S+F)=\frac 2{n_x} \chi(\O_F)=\frac 1{n_x} e(F).
\end{split}
\]
Since $g\geqslant 1$, we have $e(F)\leqslant 0$, hence $e(G_x)\geqslant e(F)$ for all $x\in E$, and the equality holds if and only if 
\[
e(G_x)= 2\chi(\O_{G_x}) \quad \mbox {hence $G_x$ is smooth, and}\quad  \frac 1{n_x} e(F)=e(F).
\]
By part (ii) of Proposition \ref {prop:cdf} we have
\[
0=e(S)=e(E)\cdot e(F)+\sum_{x\in E} (e(F_x)-e(F))=\sum_{x\in E} (e(G_x)-e(F))
\]
since $e(E)=0$. Thus we have $e(G_x)=e(F)$ for all $x\in E$. As we saw, this implies that $G_x$ is smooth for all $x\in E$. Moreover, if $e(F)\neq 0$, i.e., if $g>1$, then we must have $n_x=1$ for all $x\in E$, proving the lemma. If $g=1$ and $n_x>1$ for some $x\in E$, then the canonical bundle formula for $\alpha$ implies that there is an $n\in \NN$ and a line bundle $\mathcal K$ of positive degree on $E$ such that $nK_S=\alpha^ *(\mathcal K)$, which implies that that there is a positive integer $m$ such that $P_{nm}\geqslant 2$, so that $\kappa>0$, a contradiction. So again $n_x=1$ for all $x\in E$, and the assertion holds. \end{proof}

Next we assume that the general fibre of $\alpha$ has genus $g=1$.

\begin{lemma}\label{lem:basic} In the above set up, there is a morphism $\beta: S\to \PP^ 1$ with connected fibres of genus 1.
\end{lemma}
\begin{proof} Recall that $\chi=0$, $K_S^2=0$, $q=1$.

\begin{claim} \label{cl:two} There is a smooth elliptic curve $G$ on $S$ such that $G^ 2=K_S\cdot G=0$ and $\alpha(G)=E$. \end{claim}

\begin{proof}[Proof of the Claim] Consider a very ample divisor $H$, then $H_{|F}$ is an effective divisor of degree $m:=F\cdot H>0$. For each fibre $F$ there are $m^ 2$ points $p_{F,i}$ such that $mp_{F,i}\sim H_{|F}$, for $1\leqslant i\leqslant m^2$. As $F$ moves among the fibres, these points describe a curve $G'$, and the map
$\alpha_{|G'}: G'\to E$ is \'etale. Hence any irreducible component $G$ of $G'$ is a smooth elliptic curve. The canonical bundle formula for $\alpha$ implies that $K_S\equiv 0$, hence $K_S\cdot G=0$, thus also $G^2=0$, as desired. \end{proof}

\begin{claim} \label{cl:two} There is another smooth elliptic curve $G'\neq G$ on $S$ such that $G\cdot G'=G'^ 2=K_S\cdot G'=0$ and $\alpha(G')=E$. \end{claim}

\begin{proof}[Proof of the Claim] For every integer $n\geqslant 2$, consider the exact sequence
\[
0\to \O_S(nK_S+(n-1)G)\to \O_S(nK_S+nG)\to \O_G(nK_S+nG)=nK_G\cong \O_G \to 0
\]
and the related exact cohomology sequence
\[
H^1(\O_S(nK_S+nG))\to H^ 1(\O_G)\cong \CC\to H^2( \O_S(nK_S+(n-1)G)).
\]
By Serre Duality we have
\begin{equation}\label{eq:cak}
h^2( \O_S(nK_S+(n-1)G))=h^0( \O_S(-(n-1)K_S-(n-1)G))=0
\end{equation}
because $-(n-1)K_S-(n-1)G\equiv -(n-1)G$ and $G$ is effective. Hence we have $h^1(\O_S(nK_S+nG))\geqslant 1$. Similarly to \eqref {eq:cak} we see that $h^2(\O_S(nK_S+nG))=0$.
Therefore
\[
h^0(\O_S(nK_S+nG))=\chi +\frac 12 (nK_S+nG)\cdot ((n-1)K_S+nG)+h^1(\O_S(nK_S+nG))\geqslant 1. 
\]
Hence, for every $n\geqslant 2$ we find a curve $G_n\in |nK_S+nG|$. 

Note that $G_n$ cannot be a multiple of $G$ for all $n\geqslant 2$. Otherwise we would have
\[
G_n=mG\sim n(K_S+G), \quad G_{n+1}=m'G\sim (n+1)(K_S+G).
\]
If $A$ is any ample divisor, since $K_S\equiv 0$, we clearly have $A\cdot G_{n+1}>A\cdot G_n$, hence $m'>m$. But then
\[
K_S\sim G_{n+1}-G_n-G\sim (m'-m-1)G
\]
hence we would have $h^ 0(K_S)>0$, a contradiction. 

So we can take an integer $n\geqslant 2$ such that 
\[
G_n= mG+\sum_{i=1}^ hd_iD_i
\]
with $h>0$,  $D_1,\ldots, D_h$ irreducible distinct curves distinct from $G$, $d_1,\ldots, d_h$ positive integers. Since 
\[
0=G\cdot (nK_S+nG)=G\cdot G_n=G\cdot (mG+\sum_{i=1}^ hd_iD_i)=G\cdot \sum_{i=1}^ hd_iD_i
\] 
then we must have $G\cdot D_i=0$, for $1\leqslant i\leqslant h$. By Hodge Index Theorem this yields $D_i^2\leqslant 0$, for $1\leqslant i\leqslant h$. On the other hand
\[
0=K_S\cdot (nK_S+nG)=K_S\cdot G_n=K_S\cdot (mG+\sum_{i=1}^ hd_iD_i)= K_S\cdot \sum_{i=1}^ hd_iD_i
\]
and since $K_S$ is nef we have $K_S\cdot D_i=0$, for $1\leqslant i\leqslant h$. This implies that each $D_i$ is either rational or it is smooth of genus 1. However, there is no rational curve on $S$, because a rational curve cannot dominate $E$, so it should be contained in a fibre of $\alpha$, which is impossible, since all fibres of $\alpha$ are smooth elliptic curves. So $D_i$ is smooth, elliptic, for $1\leqslant i\leqslant h$. So we can take $G'$ equal to one of the $D_i$s.  Furthermore, $\alpha(G')=E$, otherwise $G'$ would be a fibre of $\alpha$, and this is impossible, because the fibre of $\alpha$ intersect $G$ positively. \end{proof}

Now we can finish the proof of the Lemma. Consider the exact sequence
\[
0\to \O_S(2K_S+G+G')\to  \O_S(2K_S+2G+2G')\to  \O_{G\cup G'}(2K_S+2G+2G')\cong \O_G\oplus \O_{G'}\to 0
\]
which gives rise to
\[
H^1( \O_S(2K_S+2G+2G'))\to H^1(\O_{G\cup C'}(2K_S+2G+2G'))\cong \CC^2\to H^2(\O_S(2K_S+G+G')).
\]
We have
\[
h^2(\O_S(2K_S+G+G'))\cong h^0(-K_S-G-G')=0
\]
and similarly $h^2(\O_S(2K_S+2G+2G'))=0$. Then 
\[
h^1( \O_S(2K_S+2G+2G'))\geqslant h^1(\O_{G\cup C'}(2K_S+2G+2G'))=2.
\] 
Hence
\[
\begin{split}
h^0(\O_S(2K_S+2G+2G'))&=\chi+\frac 12 (2K_S+2G+2G')\cdot (K_S+2G+2G')+\\
&+h^1(\O_S(2K_S+2G+2G'))\geqslant 2.
\end{split}
\]
Write $|2K_S+2G+2G'|=|M|+\Phi$, where $|M|$ is the movable part and $\Phi$ the fixed part and note that $\dim (|M|)\geqslant 1$. We have that $2K_S+2G+2G'$ is nef, because so if $K_S$ and moreover $(2K_S+2G+2G')\cdot G=(2K_S+2G+2G')\cdot G'=0$. So, since $M$ is also nef, we have
\[
\begin{split}
&0=(2K_S+2G+2G')^2=(2K_S+2G+2G')\cdot (M+\Phi)\geqslant\\
&\geqslant (2K_S+2G+2G')\cdot M=M\cdot (M+\Phi)\geqslant M^2\geqslant 0
\end{split}
\]
hence $M^2=0$ and $|M|$ is base point free. Moreover
\[
\begin{split}
&0=(2K_S+2G+2G')^2\geqslant (2K_S+2G+2G')\cdot (M+\Phi)\geq \\
&\geqslant (2K_S+2G+2G')\cdot M\geqslant 2K_S\cdot M\geqslant 0, 
\end{split}
\]
thus $K_S\cdot M=0$, hence $|M|$ is composed with  a pencil of elliptic curves. Hence there is a smooth curve $B$, a morphism $\beta: S\to B$ with connected fibres of arithmetic genus 1, and a linear series $\mathcal M$ on $B$ such that $|M|=\beta^*(\mathcal M)$. Note that $\beta: S\to B$ cannot coincide with $\alpha: S\to E$, since otherwise we have
\[
0<G\cdot F\leqslant G\cdot (M+\Phi)=G\cdot (2K_S+2G+2G')=0,
\]
a contradiction. Finally we observe that the genus of $B$ cannot be positive, otherwise by the universal property of the Albanese variety, we would have a commutative diagram
\begin{equation*}\label{eq:K3}
\xymatrix{ 
&S\ar^{\beta}[r]\ar_{\alpha_S}[d]&B\ar^{\alpha_B}[d] \\
&E\ar^{f}[r] & J(B)\,\,
}
\end{equation*}
and $\alpha_S$ would coincide with $\beta$, a contradiction. Hence $B\cong \PP^1$.  \end{proof}

\subsection{The proof of Theorem \ref {thm:ftc2} }\label{ssec:s1}

Now we can turn to the:
\begin{proof}[Proof of Theorem \ref {thm:ftc2}]  We argue by contradiction, assuming $\kappa=-\infty$. In  particular $P_2=0$. Then we must have $q>0$ otherwise, by Castelnuovo's Criterion, $S$ would be rational, i.e.,
 we would have a birational map $f: S\dasharrow \PP^2$, which leads to a contradiction.
 Indeed,  if $f$ is an isomorphism, then $K_S$ is not nef, a contradiction. If $f$ is a morphism which is not an isomorphism, then $S$ has some $(-1)$--curve, and $K_S$ is not nef, a contradiction. Then assume $f$ is not a morphism. Then we have a resolution of the indeterminacies
 \begin{equation*}
\xymatrix{ 
&X\ar_{p}[d]\ar^{\pi}[dr] &\\
&S\ar@{-->}^{f}[r] &  \PP^2&
}
\end{equation*}
and assume that $\pi$ is composed with the minimal number of blow--ups. Then, by an argument we already made, $X$ contains a $(-1)$--curve $C$, the exceptional curve of the last blow--up composing $\pi$, which is not contracted by $p$. We have
\[
K_X\sim p^ *(K_S)+E\sim \pi^*(K_{\PP^2})+E',
\]
where $E$ and $E'$ are divisors contracted to points by $p$ and $\pi$ respectively. As we said, $C$ is not contained in $E$. Hence
\[
-1=K_X\cdot C=p^ *(K_S)\cdot C+E\cdot C\geq p^ *(K_S)\cdot C= K_S\cdot p_*(C),
\]
thus $K_S$ is not nef, a contradiction. 

Now we claim that $q=1$. In fact 
\[
0\leqslant \chi=1-q+p_g=1-q, \quad \mbox {hence} \quad 0<q\leqslant 1.
\]
Hence the Albanese variety of $S$ is an elliptic curve $E$ and the Albanese map is a morphism $\alpha: S\to E$ which, as we saw at the beginning of \S \ref {ssec:basic}, has
 connected fibres. We denote by $g$ the genus of the general fibre of $\alpha$. Since we may assume $K_S$ is nef, we have $g>0$. \medskip

\noindent {\bf Step 1.}  First we assume $g=1$.  Then we can apply Theorem \ref {thm:canb}. Note that $K_E$ is trivial, $\chi=0$, so that $R^1\alpha_*\O_S$ has degree 0. Moreover, by Lemma \ref {cl:one}, the Albanese morphism is smooth. So we conclude that $K_S\equiv 0$. 

By Lemma \ref {lem:basic}, we have a morphism $\beta: S\to \PP^ 1$ with connected fibres of genus 1, and, again by Theorem \ref {thm:canb}, we have an $n\in \NN$ and  a line bundle $\mathcal K$ on $\PP^1$ such that $nK_S=\beta^*(\mathcal K)$. Since $K_S\equiv 0$, then $\mathcal K\cong \O_{\PP^1}$. This implies that $nK_S\cong \O_S$, so $P_n=1$, proving that $\kappa\geqslant 0$, thus contradicting the hypothesis $\kappa=-\infty$. \medskip

\noindent {\bf Step 2.}
Next we assume that the fibres of the Albanese map $\alpha: S\to E$ have genus $g>1$. Let us apply Theorem \ref {thm:rcs} to $\alpha: S\to E$, noting that $\omega_{S|E}\cong K_S$ in this case, since $K_E$ is trivial. If $\deg(\alpha_*\omega_{S|E})>0$, by Riemann--Roch theorem for vector bundles on curves, we have 
\[
\chi(\alpha_*\omega_{S|E})=h^ 0(\alpha_*\omega_{S|E})-h^1(\alpha_*\omega_{S|E})=\deg(\alpha_*\omega_{S|E})>0,
\]
hence $h^ 0(\alpha_*\omega_{S|E})>0$. But $\alpha_*\omega_{S|E}=\alpha_*(K_S)$, hence we get $p_g>0$, a contradiction. So it does not happen that $\deg(\alpha_*\omega_{S|E})>0$, then by Theorem \ref{thm:rcs}  all fibres of $\alpha: S\to E$ are isomorphic.  Let $F$ be such a fibre. Then $\Aut(F)$ is finite and $\pi_1(E)$ acts on $F$, i.e., we have a homomorphism
\[
\rho: \pi_1(E)\to \Aut(F).
\]
Hence $\ker(\rho)$ has finite index, so it corresponds to an \'etale cover $f: C\to E$, with $C$ an elliptic curve. Consider the cartesian square
\begin{equation*}\label{eq:K3}
\xymatrix{ 
&S'\ar^{g}[r]\ar_{\alpha'}[d]&S\ar^{\alpha}[d] \\
&C\ar^{f}[r] & E\,\,
}
\end{equation*}
with $\alpha': S'\to C$ isotrivial  as well, with fibres isomorphic to $F$. By definition now $\pi_1(C)$ acts trivially on the fibres of $\alpha'$, hence $S'=C\times F$. Then the canonical bundle of $S'$ is the pull--back of the canonical bundle $K_F$ on $F$ via the obvious projection, hence $\kappa(S')=1$. 

Finally we prove that $\kappa(S)\geqslant 0$. Since $g$ is \'etale,   we have $nK_{S'}=g^* (nK_S)$, for all $n\in \NN$.  The surface $S$ is the quotient of $S'$ via the action of a finite group $G$ of order $m$ which acts freely on $S'$. Hence, for all $n\in \NN$, we have
\[
H^ 0(nK_S)=H^0(nK_{S'})^ G.
\]
Take any non--zero section $s\in H^0(nK_{S'})$ and consider the section $s^G\in H^ 0(mnK_{S'})$ defined as
\[
s^G(x)=\prod_{g\in G}s(g(x)), \quad \mbox{for all}\quad x\in S'. 
\] 
The section $s^G$ is clearly non--zero and $s^G\in H^ 0(mnK_{S'})^G\cong H^ 0(mnK_{S})$. Hence $P_{mn}(S)>0$, proving that $\kappa(S)\geq 0$.

This ends the proof of Theorem \ref {thm:ftc2}. \end{proof}

\section{The classification and the Abundance Theorem}

By the Fundamental Theorem of the Classification and by Remark \ref {rem:kk}, we have that:\\
\begin{inparaenum}
\item [$\bullet$] the result of the minimal model programme applied to a surface $S$ is a Mori fibre space if and only if $\kappa(S)=-\infty$;\\
\item [$\bullet$] the result of the minimal model programme applied to a surface $S$ is a strong minimal model if and only if $0\leqslant \kappa(S)\leq 2$.
\end{inparaenum}

\subsection{Surfaces with $\kappa=-\infty$}\label{ssec:k-inf}

If $\kappa(S)=-\infty$, then either $S$ is rational, or it is birational to an irrational scroll. 
In the former case the birational equivalence class of $S$ is represented by $\PP^2$, thus $q=0$, in the latter case, since $S$ is birational to $\PP(\mathcal E)$ with $\mathcal E$ a rank 2 vector bundle on a curve $C$, with $g(C)>0$, the birational equivalence class of $S$ is represented by $C\times \PP^1$, hence $q=g(C)$, because the only holomorphic 1--forms on $C\times \PP^1$ are pull--back of the 1--forms on $C$ via the projection to the first factor. More precisely, one has $\Alb(S)=J(C)$ in this case. 

It is useful to record the following:

\begin{thm}[Vaccaro's Theorem]\label{thm:vac} If $S$ is a minimal surface with $\kappa(S)=-\infty$, then:\\
\begin{inparaenum}
\item [(i)] if $S$ is rational then either $S=\PP^2$ or $S=\FF_n$, with $n\in \NN-\{1\}$;\\
\item [(ii)] if $S$ is not rational, then $S=\PP(\mathcal E)$ with $\mathcal E$ a rank 2 vector bundle on a curve $C$ of genus $g=q>0$.
\end{inparaenum}
\end{thm}
\begin{proof} (i) First it is clear that $\PP^2$ is minimal. Also $\FF_n$ with $n\in \NN-\{1\}$, is minimal. Indeed on 
$\FF_0$ there is no curve with negative self--intersection and on $\FF_n$ with $n\geq 2$, there is only one irreducible curve with negative self--intersection equal to $-n$. Moreover all $\FF_n$ are rational.

Conversely, let $S$ be rational and minimal. Then, by the structure of the minimal model programme, $S$ is a Mori fibre space, hence it is a scroll over $\PP^1$, and the assertion  follows.

(ii) Let $S$ be not rational, minimal and $\kappa(S)=-\infty$. Then the  minimal model programme applied to $S$  is a Mori fibre space, hence $S$ is a scroll over a curve $C$ of genus $g=q>0$.   Hence $S=\PP(\mathcal E)$ with $\mathcal E$ a rank 2 vector bundle on $C$ (see Theorem \ref {thm:plus}). \end{proof}

\subsection{The Abundance Theorem: statement}\label{ssec:ab}

In this section we will  prove the following result:

\begin{thm} [The Abundance Theorem]\label{thm:ab} 
Let $S$ be a strong minimal model. Then there is an $m\gg 0$ and highly divisible such that $|mK_S|$ is base point free. 
\end{thm}

Actually we will prove a finer result. The proof will consist in separately considering the cases $\kappa=0,1,2$.  

\subsection{Surfaces with $\kappa=2$}\label{ssec:gt}  Surfaces with $\kappa=2$ are called \emph{surfaces of general type}. Recall that for a minimal such surface $S$ one has $K_S^2>0$. For any positive integer $n\in \NN$ we have the \emph{n--canonical map} $\varphi_{|nK_S|}$, often denoted by $\varphi_n$, whose image we denote by $S_n\subseteq \PP^{P_n-1}$, and call the  \emph{n--canonical image} of $S$. If $S$ is of general type, $S_n$ is a surface for infinitely many $n\in \NN$.

For surfaces of general type  we have the following result which is much stronger than the Abundance Theorem:

\begin{thm}[Bombieri's Theorem]\label{thm:bomb} Let $S$ be a minimal surface of general type. Then:\\
\begin{inparaenum}
\item [$\bullet$]  $\varphi_n$ is a birational morphism onto its image $S_n$, for all $n\geqslant 3$, except for 
$n=3,4$, $p_g=2$, $K_S^2=1$ and $n=3$, $p_g=3$, $K_S^2=2$;\\
\item [$\bullet$] $\varphi_2$ is  a birational morphism onto $S_2$ if $K_S^2\geqslant 10$, unless $S$ has a pencil of genus 2 curves;\\
\item [$\bullet$] $\varphi_n$ is an isomorphism onto $S_n$, for all $n\geqslant 5$, off finitely many rational curves $C$ with $C^2=-2$. 
\end{inparaenum}

\end{thm}

More precise results are known for canonical and bicanonical maps, but we do not dwell on this here. We will not prove Theorem \ref {thm:bomb} in its full generality here. However, in \S \ref {sec:sgt},  using some of the main ideas of Bombieri's, we will prove a weaker result which implies the Abundance Theorem in this case.  

\subsection{Surfaces with $\kappa=1$}\label{ssec:pe}  Surfaces with $\kappa=1$ are called \emph{properly elliptic surfaces}. By Castelnuovo--De Franchis--Enriques' Theorem, we have $K_S^2=0$ for a minimal such surface. 

\begin{thm}\label{thm:ell}Let $S$ be a minimal properly elliptic surface. Then there is a unique elliptic pencil $f: S\to C$ and there are infinitely many $n\in \NN$ such that there is a base point free complete linear series $\mathcal K_n$ on $C$ such that $|nK_S|=f^ *(\mathcal K_n)$.
\end{thm}

\begin{proof} Remember that there is a sequence $\{h_n\}_{n\in \NN}$ of positive integers such that $P_{h_n}$ grows like $n$ as $n\to \infty$ (see Theorem \ref {thm:kod}). For such an $h_n\in \NN$ we write
\[
|h_nK_S|=F+|M|
\]
with $F$ the fixed part and $|M|$ the movable part. We have
\[
0=(h_nK_S)^ 2=h_nK_S\cdot (M+F)\geq h_nK_S\cdot M\geqslant (M+F)\cdot M\geq M^ 2+M\cdot F\geq M^2\geq 0,
\]
hence $M^2=M\cdot F=F^ 2=0$. This implies that $|M|$ is a composed with a pencil, i.e., there 
is a pencil $f: S\to C$ and a complete, base point free linear series $\mathcal M$ on $C$ such that $|M|=f^*(\mathcal M)$. Since $K_S\cdot M=\frac 1{h_n}M\cdot (F+M)=0$, we have that the arithmetic genus of the curves of the pencil is 1. Moreover, since $F^2=M\cdot F=0$, by Zariski's lemma we have that $F=F_1+\cdots+F_k$, where each $F_i$ is a submultiple of a mupltiple of a curve of the pencil, for $1\leq i\leq k$. It is then clear that by taking a sufficiently large and highly divisible $m\in \NN$, the linear system $|mh_nK_S|$ has the required property.  

The uniqueness of the pencil $f: S\to C$ is immediate.  \end{proof}

\subsection{Surfaces with $\kappa=0$}\label{ssec:k1} If $S$ is a minimal surface with $\kappa=0$, by Castelnuovo--De Franchis--Enriques' Theorem we have $K_S^2=0$, $e\geq 0$ and $\chi\geq 0$. Moreover, for all $n\in \NN$ we have $P_n\in \{0,1\}$ and there is an $n\in \NN$ such that $P_n=1$. Furthermore
\[
0\leq \chi=1-q+p_g, \quad \mbox{hence} \quad 1\geq p_g\geq q-1, \quad \mbox{thus}\quad q\leq 2, 
\]
so we have only the following possibilities:\\
\begin{inparaenum}
\item [(i)] $q=0, p_g=1$, these surfaces are called \emph{K3 surfaces};\\
\item [(ii)] $q=0, p_g=0$, these surface are called \emph{Enriques surfaces};\\
\item [(iii)] $q=2, p_g=1$, these are \emph{abelian surfaces};\\
\item [(iv)] $q=1, p_g=0$, these surfaces are called \emph{bielliptic surfaces};\\
\item [(v)]  $q=1, p_g=1$, as we will see there are no surfaces with $\kappa=0$ with  these invariants.
\end{inparaenum}
\medskip

\begin{case}\label{case:K3} $q=0, p_g=1$ \end{case}

\begin{thm}\label{thm:k3} If $S$ is a minimal surface with $\kappa=0$ and $q=0, p_g=1$, then $K_S$ is trivial, hence all pluricanonical systems are trivial, and therefore are base point free.
\end{thm}

\begin{proof} One has
\[
h^ 0(-K_S)+h^0(2K_S)\geq h^ 0(-K_S)-h^ 1(-K_S)+h^ 2(-K_S)=\chi(-K_S)=\chi=2
\]
but $h^0(2K_S)=P_2\leq 1$, hence $h^ 0(-K_S)\geq 1$. On the other hand also $h^ 0(K_S)=p_g=1$, hence $K_S$ is trivial.
\end{proof}

\begin{example}\label{ex:k3} K3 surfaces do exist. For instance, by \S \ref {ssec:ci} and \S \ref {ssec:dc}, a smooth surface of degree
4 in $\PP^3$, a smooth complete intersection of type $(2,3)$ in $\PP^4$, a smooth complete intersection of type $(2,2,2)$ in $\PP^5$, the double cover of $\PP^2$ branched along a smooth sextic curve, are K3 surfaces. More generally, for any integer $g\geq 3$, there are smooth K3 surfaces of degree $2g-2$ in $\PP^g$, and they depend on 19 parameters (see \cite [Chapt. VIII]{BPHV}). 
\end{example}

\begin{case}\label{case:enr} $q=0, p_g=0$ \end{case}

\begin{thm}\label{thm:k3} If $S$ is a minimal surface with $\kappa=0$ and $q=0, p_g=0$, then $2K_S$ is trivial, hence all even pluricanonical systems are trivial, and therefore are base point free, whereas $P_n=0$ for all odd  $n\in \NN$.
\end{thm}

\begin{proof} First we observe that $P_2=1$ otherwise by Castelnuovo's criterion the surface is rational.  One has
\begin{equation}\label{eq:pt}
h^0(-2K_S)+h^0(3K_S)\geqslant h^0(-2K_S)-h^1(-2K_S)+h^2(-2K_S)=\chi(-2K_S)=\chi=1.
\end{equation}
Next we claim that $P_3=0$. Otherwise one has $P_3=1$ and there is a unique curve $D\in |3K_S|$. On the other hand there is a unique curve $D'\in |2K_S|$. We can write $D=\sum_{i=1}^ha_iD_i$ and $D'=\sum_{i=1}^hb_iD_i$, with $D_1,\ldots,D_h$ distinct irreducible curves and $a_i,b_i\geq 0$, for $1\leqslant i\leq h$. Then we have $P_6=1$ and the unique curve in $|6K_S|$ is $2D=3D'$. Hence we have
\[
2\sum_{i=1}^ha_iD_i=3\sum_{i=1}^hb_iD_i
\]
which implies that
\[
2a_i=3b_i, \quad \mbox{for all}\quad 1\leqslant i\leq h.
\]
Then there are non--negative integers $\lambda_1,\ldots, \lambda_h$ such that
\[
a_i=2\lambda_i, \quad b_i=3\lambda_i,  \quad \mbox{for all}\quad 1\leqslant i\leq h.
\]
Hence
\[
D-D'=\sum_{i=1}^h(a_i-b_i)D_i =\sum_{i=1}^h\lambda_iD_i\in |K_S|
\]
which implies $p_g=1$, a contradiction. 

Since $P_3=0$, then  \eqref {eq:pt} implies that $h^0(-2K_S)>0$. Since also $h^0(2K_S)=P_2=1$, we have that $2K_S$ is trivial, and the assertion follows. 
\end{proof}

\begin{remark}\label {rem:2:1} For Enriques surfaces $K_S$ is a non--trivial element of order 2 in $\Pic(S)$. This determines an \'etale double cover $\pi: S'\to S$ (see Example \ref {ex:et}), such that $K_{S'}=\pi^ *(K_S)=\O_{S'}$, so that $\kappa(S')=0$ . Moreover, by \eqref {eq:qu}, we have 
\[
q(S')=h^ 1(\O_{S'})=h^1(\pi_*\O_S')=h^ 1(\O_S)+h^1(K_S)=2h^ 1(\O_S)=0.
\]
Thus $S'$ is a K3 surface. Enriques surfaces can also be defined as the quotient of a K3 surface via a fixed point free involution. 
\end{remark}

Enriques surfaces do exist. Here we give three examples.

\begin{example}\label {ex:enr} (i) {\bf Classical Enriques sextic surfaces.} Let $T$ be a tetrahedron in $\PP^3$, which, up to change of coordinates, we may assume to be the \emph{coordinate tetrahedron}, whose four vertices are the \emph{coordinate points} with all coordinates but one equal to zero. Consider the linear system $\mathcal S$ of surfaces of degree 6 which are singular along the 6 edges of $T$. The linear system $\mathcal S$ has equation
 \begin{equation*} \label{eq:sigma}
c_3(x_0x_1x_2)^2+c_2(x_0x_1x_3)^2+c_1(x_0x_2x_3)^2+c_0(x_1x_2x_3)^2+Qx_0x_1x_2x_3=0,
\end{equation*} 
where $c_0,\ldots, c_3\in \CC$ and $Q=\sum_{i \leqslant j}q_{ij}x_ix_j$ is a quadratic form. 
This shows that $\dim(\mathcal S)=13$ and we may identify $\mathcal S$ with the $\PP^ {13}$ with  homogeneous coordinates
\[ q=(c_0:c_1:c_2:c_3:q_{00}:q_{01}:q_{02}:q_{03}:q_{11}:q_{12}:q_{13}:q_{22}:q_{23}:q_{33}). \]  
An open dense set of $\mathcal S$ correspond to irreducible surfaces which have ordinary singularities along the edges of $T$ and no further singularity. For any such surfaces $\Sigma$ its normalization $\nu:S\to \Sigma$ is smooth and
it is an Enriques surface.  Indeed, by \S \ref {ssec:ci}, $K_S$ is the pull--back via $\nu$ of the linear system of quadrics passing through the edges of $T$, off the pull--back of these edges. Since there is no such a quadric, then $|K_S|=\emptyset$ and $p_g=0$. The bicanonical system $2K_S$ is the pull--back via $\nu$ (off the pull--back of the edges of $T$) of the linear system of surfaces of degree 4 which are singular along the edges of $T$. There is only one such surface, namely $T$ itself, and the mentioned pull--back is trivial. Thus $2K_S\sim \O_S$ is trivial, $P_2=1$, $\kappa=0$ and $S$ is an Enriques surface, because $q=0$. To see this one can argue as follows. First blow--up the vertices of $T$, then the edges of $T$, so to obtain a smooth threefold $X$. The  take the strict transform $S'$ of an  element $S\in \mathcal S$ with ordinary singularities to $X$. Then $S'$ moves in a linear system $\mathcal S'$, the proper transform of $\mathcal S$, which is base point free, hence $S'$ is nef. Moreover it is not difficult to see that $S'^3>0$. So $\O_X(S')$ is big and nef.   Then the same argument we made in \S \ref {ssec:ci}, using the Kawamata--Viehweg Vanishing Theorem, tells us that $h^1(\O_{S'})=h^1(\O_X)$. But $S'$ is birational to $S$ and $X$ to $\PP^3$, hence $h^1(\O_S)=h^1(\O_{S'})=h^1(\O_X)=h^1(\O_{\PP^3})=0$.

It is known that \emph{all} Enriques surfaces appear as the desingularization of some surface $\Sigma\in \mathcal S$. Since $\dim(\mathcal S)=13$ but there is the 3--dimensional group of projective transformations fixing $T$ which acts on $\mathcal S$, one has that Enriques surfaces depend on $10$ parameters. 

(ii) Consider the complete intersection surface $S'$ of type $(2,2,2)$ in $\PP^5$ defined by equations of the form
\[
Q_i(x_0,x_1,x_2)+Q'_i(x_3,x_4,x_5)=0 \quad \mbox{with} \quad 1\leqslant i\leqslant 3, 
\]
where $Q_i(x_0,x_1,x_2)$ and $Q'_i(x_3,x_4,x_5)$ are quadratic forms, for $ 1\leqslant i\leqslant 3$.  It is not difficult to see that for a sufficiently general choice of $Q_i(x_0,x_1,x_2)$ and $Q'_i(x_3,x_4,x_5)$, for $1\leqslant i\leqslant 3$, the surface $S'$ is smooth, hence it is a K3 surface (see Example \ref {ex:k3}). Consider the projective transformation
\[
\sigma: [x_0,x_1,x_2,x_3,x_4,x_5]\in \PP^5 \to   [x_0,x_1,x_2,-x_3,-x_4,-x_5]\in \PP^5
\]
which clearly fixes $S'$ and acts on it freely for a sufficiently general choice of $Q_i(x_0,x_1,x_2)$ and $Q'_i(x_3,x_4,x_5)$, for $1\leqslant i\leqslant 3$. The quotient $S$ of $S'$ by the action of $\sigma$ is an Enriques surface. 

(iii) {\bf Reye congruences.} Let $\mathcal P$ be a linear system of dimension 3 of quadrics in $\PP^3$ such that :\\
\begin{inparaenum}
\item [$\bullet$] $\mathcal P$ is base point free;\\
\item [$\bullet$] if $\ell$ is a double line for $Q\in \mathcal P$, then $Q$ is the unique quadric in $\mathcal P$ containing $\ell$.
 \end{inparaenum}
 
 These conditions are verified if $\mathcal P$ is sufficiently general. 
 
 Next, denote by $\mathbb G$ the grassmannian variety of lines in $\PP^3$ and consider 
 \[
 S=\{\ell\in \mathbb G:  \mbox{$\ell$ is contained in a pencil contained in $\mathcal P$}\}.
 \]
 We claim that $S$ is an Enriques surface. Surfaces obtained in this way are called \emph{Reye congruences}.
 
 To prove our claim, consider $S'\subset \PP^3\times \PP^3$ thus defined
 \[
 S'=\{(x,y)\in  \PP^3\times \PP^3:  \mbox{$x\neq y$ and $x$ and $y$ are conjugated with respect to all  $Q\in \mathcal P$}\}.
 \]
 One remarks that $\ell=\langle x,y\rangle\in S$ if and only if $(x,y)\in S'$ and moreover the  involution 
 \[
 \sigma: (x,y)\in S'\to (y,x)\in S'
 \]
 has no fixed points. Then $S$ is the quotient of $S'$ via the action of $\sigma$ and, to prove that $S$ is an Enriques surface, one has to prove that $S'$ is a K3 surface. To see this, suppose $\mathcal P$ is spanned by the four independent quadrics $Q_i(x_0,x_1,x_2,x_3)=0$, with $1\leqslant i\leqslant 4$. Consider the bilinear forms $\phi_i(x_0,x_1,x_2,x_3; y_0,y_1,y_2,y_3)$ associated to $Q_i$, for $1\leqslant i\leqslant 4$. Then $S'$ is clearly defined in $\PP^3\times \PP^3$ by the four equations $\phi_i=0$, for $1\leqslant i\leqslant 4$.  The canonical bundle of $\PP^3\times \PP^3$ is the sum of the pull--backs of $K_{\PP^3}=\O_{\PP^3}(-4)$ via the projections to the two factors. Then  the adjunction formula (see \S \ref {ssec:ci}) tells us that $K_{S'}$ is trivial.  Moreover one can prove that $q(S')=0$ (see  \S \ref {ssec:ci} for a similar argument), thus $S'$ is a K3 surface.  
\end{example} 

For more information on Enriques surfaces, see \cite {CD} and \cite [Chapt. VII] {BPHV}. 

\begin{case} \label{case:ab} $q=2$, $p_g=1$
\end{case}

We need the following special case of Poincar\'e Complete Reducibility Theorem \ref {thm:pcrt}:

\begin{lemma}\label{lem:pcrt} Let $A$ be an abelian surface and let $C$ be a smooth elliptic curve contained  in $A$. Then there is a smooth elliptic curve $E$ and a morphism $f: A\to E$ with connected fibres such that $C$ is a fibre of $f$.
\end{lemma}

\begin{proof} Since $K_A$ is trivial, we have $C^2=C\cdot (C+K_A)=0$. We also assume that, up to translations,  $C$ contains the point $0$ of $A$.
 
For each $x\in A$ consider the curve $C_x=x+C$, which is isomorphic to $C$. Note that $x\in C_x$.
For each $x\in A$, the curve $C_x$ is homologically equivalent to $C$, hence $C_x\cdot C=0$ and $C_x^2=0$.
Then given $x,y\in A$, either $C_x=C_y$ or $C_x\cap C_y=\emptyset$. Moreover, for every point $z\in A$,  $C_z$ is the unique curve in the family $\mathcal C:=\{C_x\}_{x\in A}$ passing through $z$. Furthermore, if $y\in C_x$, then $C_x=C_y$, because both $C_x$ and $C_y$ contain $y$. 

Consider now the map 
\[
f: x\in A\to [C_x-C]\in \Pic^0(A).
\]
From the above remarks it follows that all curves in the family $\mathcal C$ are contained in fibres of $f$.
Note that $f$ cannot be constant. Otherwise the curves  in  $\mathcal C$ would be linearly equivalent and we would have a morphism $\phi: A\to \PP^1$ such that the curves  in  $\mathcal C$ are the fibres of $\phi$. By applying the canonical bundle  formula for elliptic fibrations, we find a contradiction to $K_A$ being trivial. 

So the image of $f$ is a curve $E$ which we can assume to be smooth and, up to Stein factorization, we can assume that  $\mathcal C$  is the family of all fibres of $f$. By applying again  the canonical bundle  formula for elliptic fibrations we se that $E$ has to be an elliptic curve, as desired. 
\end{proof}

\begin{thm}[Enriques' Theorem] \label{thm:ab} If $S$ is a minimal surface wirth $\kappa=0$ and $q=2$, $p_g=1$, then $S$ is an abelian surface. In particular $K_S$ is trivial, so that the Abundance Theorem holds. 
\end{thm}
\begin{proof} Consider the Albanese map
\[
\alpha: S\to A:=\Alb(S).
\]
There are two cases to be discussed:\\
\begin{inparaenum}
\item [(i)] $\dim(\alpha(S))=1$;\\
\item [(ii)] $\alpha(S)=A$.
\end{inparaenum}\medskip

\noindent {\bf Case (i).}\label{casei}  We prove that this case does not occur. Indeed, consider the Stein factorization
\begin{equation*}
\xymatrix{ 
&S\ar_{\beta}[d]\ar^{\alpha}[dr] &\\
&C\ar^{j}[r] &  A&
}
\end{equation*}
Since $\alpha(S)$ spans $A$, then $g(C)\geq 2$. On the other hand the map
\[
\beta^*: H^0(\Omega^1_C)\to H^0(\Omega^1_S)
\]
is an injection, hence $2=q(S)\geq g(C)$, thus $g(C)=2$. Now take a non--trivial point of order 2 in $\Pic^0(C)$, and consider the corresponding \'etale double cover $f: \widetilde C\to C$, with $g(\widetilde C)=3$. Consider the cartesian diagram
\begin{equation*}\label{eq:K3}
\xymatrix{ 
&\widetilde S\ar^{g}[r]\ar_{\tilde\alpha}[d]&S\ar^{\beta}[d] \\
&\widetilde C\ar^{f}[r] & C\,\,
}
\end{equation*}
We claim that $\kappa(\widetilde S)=\kappa(S)=0$. Since $K_{\widetilde S}=g^ *(K_S)$, it is clear that  
$\kappa(\widetilde S)\geq \kappa(S)$. To see the converse, we make an argument similar to the one at the end of the proof of  Theorem \ref {thm:ftc2}. Take an $n\gg 0$ such that 
\[
\dim({\rm im}(\varphi_{|nK_{\widetilde S}|}))=\kappa(\widetilde S). 
\]
The surface $S$ is the quotient of $\widetilde S$ via the action of a group $G\cong \ZZ_2$ which acts freely on $\widetilde S$. Hence we have
\[
H^ 0(nK_S)=H^0(nK_{\widetilde S})^ G.
\]
Take any non--zero section $s\in H^0(nK_{\widetilde S})$ and consider the section $s^G\in H^ 0(2nK_{\widetilde S})$ defined as
\[
s^G(x)=\prod_{g\in G}s(g(x)), \quad \mbox{for all}\quad x\in \widetilde S. 
\] 
The section $s^G$ is  non--zero and $s^G\in H^ 0(2nK_{\widetilde S})^G\cong H^ 0(2nK_{S})$. Consider now the map
\[
\rho: [(s)]\in \PP(H^ 0(nK_{\widetilde S}))\to [(s^G)]\in \PP(H^ 0(2nK_{\widetilde S})^G)=\PP(H^ 0(2nK_{S})).
\]
The map $\rho$ is clearly finite, since it maps the general $D\in |nK_{\widetilde S}|$ to 
\[
D^G:= \sum_{\gamma\in G}\gamma(D)\in |2nK_{\widetilde S}|.
\]
This proves that $\dim({\rm im}( \varphi_{2nK_{S}}))=\dim ({\rm im}(\varphi_{nK_{\widetilde S}}))=\kappa(\widetilde S)$, which proves the claim.

Hence we proved that $\kappa(\widetilde S)=0$. Moreover $q(\widetilde S)\geq 3$, because we have the map $\tilde\alpha: \widetilde S\to C$ and $g(\widetilde C)=3$. This is a contradiction, since, as we saw, there are no surfaces with $\kappa=0$ and $q>2$. This proves that case (i) does not occur.\medskip

\noindent {\bf Case (ii).} We prove that $K_S$ is trivial in this case. Assume the contrary holds and let $D=\sum_{i=1}^ha_iD_i$ be the unique element in $|K_S|$, where $D_1,\ldots, D_h$ are irreducible and distinct and $a_1,\ldots, a_h$ are positive integers. For every $i\in \{1,\dots, h\}$ one has
\begin{equation}\label{eq:kk}
0=K_S^2=K_S\cdot D\geqslant K_S\cdot D_i=D\cdot D_i=\sum_{j\neq i}^ha_jD_j\cdot D_i+a_iD_i^2\geq a_iD_i^2
\end{equation}
and 
\[
p_a(D_i)=\frac 12 (D_i+K_S)\cdot D_i+1 \leq 1,
\]
hence:\\
\begin{inparaenum}
\item [(a)] either $D_i$ is a smooth elliptic curve;\\
\item [(b)] or $D_i$ is rational which could either be smooth or singular with a node or a cusp.
\end{inparaenum}

Assume that case (b) occurs for all $i\in \{1,\dots, h\}$ . Then each divisor $D_i$ is contracted to a point by $\alpha$, because there are no rational curves on an abelian variety. Thus $D$ is contracted to a union of points, hence, by the Hodge Index Theorem, one has $K_S^2=D^2<0$, a contradiction. Hence there is an $i\in \{1,\dots, h\}$  such that case (a) occurs. From \eqref {eq:kk} we have $D_i^2\leq 0$ and $K_S\cdot D_i\leq 0$. On the other hand one has $D_i\cdot (K_S +D_i)=0$, so that $D_i^2= 0$ and $K_S\cdot D_i=0$. Therefore $D_i$ is not contracted to a point by $\alpha$. Its image is a smooth elliptic curve $E$ in $A$. By Lemma \ref {lem:pcrt}, there is  an elliptic curve $B$ and a morphism $f: A\to B$ with connected fibres such that $E$ is a fibre of $f$. Consider $\beta=f\circ \alpha: S\to B$. Then $D_i$ is contained in a fibre of $\beta$ and, being $D_i^2=0$, one has that $D_i$ is the support of a fibre of $\beta$. If $nD_i$ is the full fibre, then for all  $h\in \NN$ we have
\[
h^ 0(hnK_S)\geq h^ 0(hnD_i)=h
\]
which is not possible. This proves that $K_S$ is trivial.

Let now $R$ be the ramification divisor of the Albanese map $\alpha: S\to A$. By the ramification formula we have 
\[
\O_S=K_S=\alpha^ *(K_A)+R=R
\]
which proves that $R$ is trivial, hence $\alpha$ is unramified. This implies that $S$ is an abelian surface and $\alpha$ is an isomorphism (see Proposition \ref{prop:pp}).\end{proof}

\begin{case}\label{case:bdf}
$q=1$, $p_g=0$
\end{case} 

\begin{thm}\label{thm:biel} If $S$ is a surface with $\kappa=0$, $q=1$ and $p_g=0$, for $n\in \NN$ with $P_n=1$, one has that $nK_S$ is trivial.
\end{thm}
\begin{proof} By an argument we already made, the Albanese map $\alpha: S\to A:=\Alb(S)$, where $A$ is an elliptic curve, has connected fibres. Let $F$ be the general such fibre. 
Of course $g(F)\geq 1$. 

We claim that $g(F)=1$.  Indeed, assume by contradiction that $g(F)\ge 2$. Then we can argue as in Step 2 of the proof of Theorem \ref {thm:ftc2}. Namely, if $\deg(\alpha_*\omega_{S|A})>0$, we have $p_g>0$, a contradiction. Otherwise we have a cartesian square
\begin{equation*}\label{eq:K3}
\xymatrix{ 
&S'\ar^{g}[r]\ar_{\alpha'}[d]&S\ar^{\alpha}[d] \\
&C\ar^{}[r] & A\,\,
}
\end{equation*}
with $C$ an elliptic curve and $S'=C\times F$, hence  $\kappa(S')=1$. On the other hand, the same argument we made in the proof of case (i) of Enriques' Theorem \ref {thm:ab}, on p. \pageref {casei}, proves that $\kappa(S)=\kappa(S')=1$, a contradiction. This proves that $g(F)=1$. 

Let now $n\in \NN$ be such that $P_n=1$ and consider the unique divisor $D\in |nK_S|$. 
We claim that $D=0$. We argue by contradiction and assume $D\neq 0$.
Since $K_S\cdot F=0$, we have $D\cdot F=0$, hence $D$ consists of parts of fibres of $\alpha$. Moreover $D^2=(nK_S)^2=0$ and Zariski's Lemma, imply that the support of $D$ consists of the support of fibres of $\alpha$. But in this case for $m\gg0$ and highly divisible we would have 
\[
P_{nm}=h^ 0(nmK_S)=h^0(mD)>1
\]
a contradiction. Hence  $nK_S$ is trivial, proving the theorem.\end{proof}

The surfaces in this case, called bielliptic surfaces, have been classified by Bagnera--De Franchis, and we will explain their classification in \S \ref{sec:bdf} below.

\begin{case}\label{case:pq1}
$q=p_g=1$
\end{case}

This case does not occur. Indeed, let $\eta\in \Pic(S)$ be a non--trivial order two element. Then
\[
h^0(\eta)+h^0(K_S-\eta)\geq \chi(\eta)=\chi=1.
\]
Since $h^0(\eta)=0$, we have $h^0(K_S-\eta)\geq 1$. Take $D\in |K_S-\eta|$ and $D'\in |K_S|$. Then $2D, 2D'\in |2K_S|$, hence $2D=2D'$,  thus $D=D'$ and $\eta=0$, a contradiction. 

\section{Surfaces of general type}\label{sec:sgt}

In this section we give a sketch of the proof of the following weaker version of Bombieri's Theorem \ref {thm:bomb}:

\begin{thm}\label{thm:bbbb} Let $S$ be a minimal surface of general type. Then $|nK_S|$ is base point free as soon as $n\geq 5$. Moreover $\varphi_{nK_S}$ is birational onto its image $S_n$ as soon as $n\geq 6$.
\end{thm}

The proof of Bombieri's Theorem \ref {thm:bomb}, though more complicated, relies on the same basic ideas. For the proof we need a few essential preliminaries. 

\subsection{Some vanishing theorems}\label{ssec:FR} Let $S$ be a surface and $C$ a curve on $S$. One says that $C$ is \emph{numerically $m$--connected} (or simply  \emph{$m$--connected}), if for every decomposition $C=C_1+C_2$ with $C_1,C_2$ effective, non--zero, one has $C_1\cdot C_2\geq m$. 
%A basic fact is the following:

%\begin{lemma}[Franchetta--Ramanujam Lemma] \label{lem:FR} If $C$ is connected one has $h^0(\O_C)=1$. \end{lemma}

The following are fundamental results which widely extend Kodaira vanishing theorem:

\begin{thm} [Franchetta--Ramanujam's Theorem]\label {thm:FR} Let $S$ be a surface and $C$ an effective divisor on $S$. If $C$ is 1--connected and $C^2>0$, then $h^1(K_S+C)=0$. 
\end{thm}

For the proof, see \cite [p. 179--180]{Bom}. 

The following is the surface version of Kawamata--Viehweg Vanishing Theorem:

\begin{thm}[Ramanujam--Mumford's Vanishing Theorem]\label{thm:mum} Let $S$ be a surface and $C$ an effective divisor on $S$. If $C$ is nef and $C^2>0$, then $h^1(K_S+C)=0$. 
\end{thm}

For the proof, see \cite [Theorem (12.1)]{BPHV}. 

\begin{corollary}\label{cor:purj} Let $S$ be a minimal  surface of general type. Then
\begin{equation}\label{eq:plur}
P_n=\chi+\frac {n(n-1)}2 K_S^2\quad \text{for all}\quad n\geq 2.
\end{equation}
\end{corollary}

\begin{proof} By Serre duality we have $h^2(nK_S)=h^0((1-n)K_S)=0$, for all $n\geq 2$. Moreover all pluricanonical divisors are nef, with positive self--intersection. Hence, by Ramanujam--Mumford's  Vanishing Theorem, we have $h^1(nK_S)=0$, for all $n\geq 2$. The assertion follows by Riemann--Roch theorem.\end{proof}

\subsection{Connectedness of pluricanonical divisors}\label{ssec:plur}

We consider here a minimal surface $S$ of general type, so that it is also a strong minimal model and therefore $K_S$ is nef. 

\begin{lemma}[Franchetta--Bombieri's Lemma] \label{lem:cp} Let $S$ be a minimal  surface of general type and let $C\equiv nK_S$ with $n\geq 1$ a curve. Then $C$ is 1--connected.
\end{lemma}

For the proof, see \cite [p. 181]{Bom}.

\begin{lemma}\label{lem:ytic}
Let $S$ be a minimal  surface of general type, let $x$ be a point on $S$, let $p: S'\to S$ the blow--up of $S$ at $x$, with exceptional divisor $E$. Let $C$ be a curve on $S'$ with $C\in |p^*(nK_S)-2E|$, with $n\geq 3$. Then $C$ is  1--connected. 
\end{lemma} 

For the proof, see \cite [p. 183]{Bom}.

\begin{lemma}\label{lem:ytic2} Let $S$ be a minimal  surface of general type, let $x,y$ be distinct points on $S$, let $p: S'\to S$ be the blow--up of $S$ at $x,y$, with exceptional divisor $E_x, E_y$. Let $C$ be a curve on $S'$ with $C\in |p^*(nK_S)-2E_x-2E_y|$, with $n\geq 4$. Then $C$ is  1--connected. 
\end{lemma}

For the proof, see again \cite [p. 183]{Bom}.

\subsection{Base point freeness}\label{ssec:bpf}

In this section we prove the:

\begin{thm}\label{thm:bbbb1} Let $S$ be a minimal surface of general type. Then $|nK_S|$ is base point free as soon as $n\geq 5$. 
\end{thm}

\begin{proof} Let $x\in S$ and let $p:S'\to S$ be the blow--up of $S$ at $x$, with exceptional divisor $E$. Consider the exact sequence
\[
0\to p^*(nK_S)(-E)\to  p^*(nK_S)\to  p^*(nK_S)_{|E}=\O_E\to 0.
\]
From this it follows that $x$ cannot be a base point for $|nK_S|$ if $h^1(p^*(nK_S)(-E))=0$. Suppose there exists a divisor  $C\in |p^*((n-1)K_S)(-2E))|$. Then $C$ is 1--connected as soon as $n\geq 4$ by Lemma \ref {lem:ytic}. Moreover $C^2=(n-1)^2K_S^2-4\geq 5>0$ if $n\geq 4$. Hence, by Theorem
\ref {thm:FR}, we have that $h^1(K_{S'}+C)=0$. Since $K_{S'}=p^*(K_S)+E$, then
$K_{S'}+C\sim p^*(K_S)+E+p^*((n-1)K_S)(-2E))=p^*(nK_S)(-E)$, hence 
$h^1(p^*(nK_S)(-E))=0$, which proves that $x$ is not a base point for $|nK_S|$.

Finally the existence of $C\in |p^*((n-1)K_S)(-2E))|$ is ensured by the fact that, if $n\geq 5$ one has $\dim(|(n-1)K_S|)\geq \chi+\frac {(n-1)(n-2)}2-1\geq 6$, hence certainly there is
a curve in $|(n-1)K_S|$ which is singular at $x$, because this imposes at most 3 conditions to $|(n-1)K_S|$. \end{proof}

\subsection{Birationality}\label{ssec:bir}

In this section we prove the:

\begin{thm}\label{thm:bbbb1} Let $S$ be a minimal surface of general type. Then the map $\phi_{nK_S}$ is birational onto its image as soon as $n\geq 6$. 
\end{thm}

\begin{proof} Let $x\in S$ be a general point. Suppose $\phi_{nK_S}$ non--birational. Then there is another point $y\in S$ different from $x$, such that $\phi_{nK_S}(x)=\phi_{nK_S}(y)$, i.e., $x$ and $y$ are not separated by  $\phi_{nK_S}$.

Let $p:S'\to S$ be the blow--up of $S$ at $x$ and $y$ with exceptional divisors $E_x$ and $E_y$. Consider the exact sequence
\[
0\to p^*(nK_S)(-E_x-E_y)\to  p^*(nK_S)\to  p^*(nK_S)_{|E_x\cup E_y}=\O_{E_x}\oplus \O_{E_y}\to 0.
\]
Since $x$ and $y$ are not separated by $\phi_{nK_S}$, then the map
\[
H^0( p^*(nK_S))\to H^0(\O_{E_x}\oplus \O_{E_y})\cong \CC^2
\]
is not surjective, and this implies that $h^1(p^*(nK_S)(-E_x-E_y))>0$. We will see this is a contradiction.

Indeed, suppose there exists a divisor  $C\in |p^*((n-1)K_S)(-2E_x-2E_y)|$. Then $C$ is 1--connected as soon as $n\geq 5$ by Lemma \ref {lem:ytic2}. Moreover $C^2=(n-1)^2K_S^2-8\geq 1>0$ if $n\geq 4$. Hence, by Theorem
\ref {thm:FR}, we have that $h^1(K_{S'}+C)=0$. Since $K_{S'}=p^*(K_S)+E_x+E_y$, then
$K_{S'}+C\sim p^*(K_S)+E_x+E_y+p^*((n-1)K_S)(-2E_x-2E_y)=p^*(nK_S)(-E_x-E_y)$, hence 
$h^1(p^*(nK_S)(-E_x-E_y))=0$, a contradiction.

Finally the existence of $C\in |p^*((n-1)K_S)(-2E_x-2E_y)|$ is ensured by the fact that, if $n\geq 6$, one has $\dim(|(n-1)K_S|)\geq \chi+\frac {(n-1)(n-2)}2-1\geq 10$, hence certainly there is
a curve in $|(n-1)K_S|$ which is singular at $x$ and $y$, because this imposes at most 6 conditions to $|(n-1)K_S|$. \end{proof}

\section{Bagnera--De Franchis' classification of bielliptic surfaces}\label {sec:bdf}

In this section we explain Bagnera--De Franchis' classification of bielliptic surfaces, with $\kappa=0$, $q=1$, $p_g=0$. The main result is as follows:

\begin{thm}[Bagnera--De Franchis' Theorem]\label {thm:bdff} The bielliptic surfaces $S$ are as follows: $E,F$ are smooth elliptic curves, $G$ is a group of translations of $E$ acting on $F$ and $S=E\times F/G$, with:\\
\begin{inparaenum}
\item [(i)] $G=\ZZ_2$ acting on $F$ as $x\in F\to -x\in F$;\\
\item [(ii)] $G=\ZZ^2_2$ acting on $F$ as $x\in F\to -x\in F$ and $x\in F\to x+\epsilon\in F$, where $\epsilon$ is a non--trivial order 2 point on $F$;\\
\item [(iii)] $F=F_i=\CC/\ZZ\oplus i\ZZ$ and $G=\ZZ_4$, acting on $F$ as $x\in F\to ix\in F$;\\
\item [(iv)] $F=F_i$ and $G=\ZZ_4\times \ZZ_2$ acting on $F$ as $x\in F\to ix\in F$ and $x\in F\to x+\frac {1+i}2\in F$;\\
\item [(v)] $F=F_\rho=\CC/\ZZ+\rho\ZZ$, with $\rho$ a non--trivial cubic root of 1, $G=\ZZ_3$ acting on $F$ as
$x\in F\to \rho x\in F$;\\
\item [(vi)]  $F=F_\rho$, $G=\ZZ_3^2$, acting on $F$ as $x\in F\to \rho x\in F$ and $x\in F\to x+\frac {1-\rho}3\in F$;\\
\item [(vii)] $F=F_\rho$, $G=\ZZ_6$, acting on $F$ as $x\in F\to -\rho x\in F$.
\end{inparaenum}

The first trivial pluricanonical bundle for $S$ is:\\
\begin{inparaenum}
\item [(a)] $2K_S$  in cases (i) and (ii);\\
\item  [(b)] $4K_S$ in cases (iii) and (iv);\\
 \item  [(c)] $3K_S$  in cases (v) and (vi);\\
  \item  [(d)] $6K_S$  in case (vii).
\end{inparaenum}
\end{thm}

We start with the:

\begin{lemma}\label{lem:bb} A minimal bielliptic surface $S$ is the quotient of the product of two elliptic curves via a free group action.  More precisely there is a commutative diagram
\begin{equation}\label{eq:oop}
\xymatrix{ 
&S'\ar^{g}[r]\ar_{\alpha'}[d]&S\ar^{\alpha}[d] \\
&E\ar^{f}[r] & A\,\,
}
\end{equation}
where $\alpha: S\to A=\Alb(S)$ is the Albanese elliptic isotrivial fibration with elliptic base $A$, $S'=E\times F$, is the product of two smooth elliptic curves, $\alpha'$ is the projection  onto the first factor, and $f,g$ are \'etale, induced by a free group $G$ action.
\end{lemma}
\begin{proof}
We saw in the proof of Theorem \ref {thm:biel} that a bielliptic surface  $S$ has an elliptic fibration $\alpha: S\to A$, with $A$ elliptic base. This clearly coincides with the Albanese map. Moreover $p_g=0$ whereas there is a multiple of $K_S$ which is trivial. Hence $K_S$ is a non--trivial torsion  point of $\Num(S)$. By the canonical bundle formula for elliptic fibrations there are no multiple fibres of $\alpha$ and $(R^1\alpha_*\O_S)^ \vee$ is a torsion line bundle on $A$. Since $e(S)=0$, by Proposition \ref {prop:cdf}, (ii), there are no singular fibres of $\alpha$. Then we have a morphism $\mu:A\to \mathcal M_1$, where $\mathcal M_1\cong \mathbb A^1$ is the moduli space of curves of genus 1, which assignes to each point $x$ the modulus of the fibre of $\alpha$ over $x$.  Since $A$ is projective, $\mu$ is constant, hence the elliptic fibration $\alpha: S\to A$ is isotrivial with fibre $F$. Then, 
by an argument we already made in Step 2 of the proof of Theorem \ref {thm:ftc2}, 
there is a diagram as \eqref {eq:oop}
where $f,g$ are \'etale, induced by the action of a finite group $G$ acting freely on $S'$ and $E$, so that $E$ is an elliptic curve and $S'=E\times F$. \end{proof}

Note that  the group $G$ in the statement of Lemma \ref {lem:bb} is a finite group of translations of $E$, because  $E/G$ is the elliptic curve $A$. In fact, if $E=\CC/\Lambda$ and $A=\CC/\Lambda'$, then $\Lambda\subseteq \Lambda'$ and $G=\Lambda'/\Lambda$. Then $G$ is abelian.

\begin{lemma}\label{lem:patr} In the above set up, there is a smooth elliptic curve $E'$, an \'etale morphism $h: E'\to E$, and a finite group $G'$ acting on $E'$ and on $F$ (hence acting diagonally on the product $E'\times F$), such that $E'/G'\cong E/G=A$ and $E'\times F/G'\cong E\times F/G=S$.
\end{lemma}

\begin{proof} For every $g\in G$, every $x\in E$ and every $y\in F$ we set $g\cdot (x,y)=(g\cdot x, \phi_{g,x}(y))$, where  $\phi_{g,x}$ is an automorphism of $F$ depending on $x$ and $g$. Let us fix an origin on the elliptic curve $F$ so that $F$ has the structure of an abelian variety. Then $\phi_{g,x}$ has the form
\[
\phi_{g,x}: y\in F\to a_{g,x}\cdot y\oplus t_{g,x}\in F
\]
where $a_{g,x}$ is an automorphism of $F$ preserving the group structure  and $t_{g,x}\in F$ (we denote by $\oplus$ the addition in $F$ and by $\ominus$ the difference). The group $\Aut_0(F)$ of automorphisms of $F$ preserving the group structure is finite (see Theorem \ref {thm:ell} below), hence $a_{g,x}=a_g$ does not depend on $x$. Moreover
\[
t_g: x\in E\to t_{g,x}\in F
\]
is a morphism.  Of course $a_{\rm id}={\rm id}$ and $t_{\rm id}$ is the zero map.

By imposing that $fg(x,y)=f(g(x,y))$ we get $a_{gf}=a_g\circ a_f$ and also
\begin{equation*}\label{eq:klop}
a_f\cdot t_{g,x}+t_{f, g\cdot x}=t_{fg,x},
\end{equation*}
in particular
\begin{equation}\label{eq:klop}
a_{g^{-1}}\cdot t_{g,x}+t_{g^-1, g\cdot x}=0
\end{equation}

Since $a_{gf}=a_g\circ a_f$ the map 
\[
a: g\in G\to a_g\in \Aut_0(F)
\]
is a homomorphism. 

\begin{claim}\label{cl:kop} There is a morphism $r: E\to F$ and a positive integer $n$ such that, for all $x\in E$ and $g\in G$, one has
\begin{equation}\label{eq:hji}
r(g\cdot x)\ominus a_g\cdot r(x)=nt_{g,x}.
\end{equation}
\end{claim}
\begin{proof}[Proof of Claim \ref {cl:kop}] Note that we have a canonical isomorphism
\[
\psi: y\in F\to [y-0]\in \Pic^0(F)
\]
(we denote by $[D]$ the class of a divisor $D$ in $\Pic(F)$) and we denote the sum in $\Pic^0(F)$ still by $\oplus$.
If we have $y_1,\ldots, y_n\in F$, then 
\[
\psi(y_1\oplus\ldots\oplus y_n)=[y_1\oplus\ldots\oplus y_n-0]
\]
and also
\[
\psi(y_1\oplus\ldots\oplus y_n)=\psi(y_1)\oplus\cdots \oplus \psi(y_n)=[y_1-0]\oplus\cdots \oplus [y_n-0]=[y_1+\cdots +y_n-n0]
\]
so that
\[
y_1+\cdots +y_n-n0\sim y_1\oplus\ldots\oplus y_n-0
\]
hence
\begin{equation}\label{eq:lol}
y_1+\cdots +y_n\sim y_1\oplus\ldots\oplus y_n+(n-1)0.
\end{equation}

Let $v$ be an automorphism of $F$ given by $v(y)=a\cdot y\oplus t$, with $a\in \Aut_0(F)$ and $t\in F$. If $D=(n-1)0+y$, we have
\[
v^*(D)=(n-1)v^{-1}(0)+v^{-1}(y)=(n-1) (\ominus a^{-1}\cdot t)+(a^{-1}\cdot y\ominus a^{-1}\cdot t).
\]
By applying \eqref {eq:lol}, we deduce that
\begin{equation}\label{eq:qrt}
\begin{split}
v^*(D)&\sim (n-1) (\ominus a^{-1}\cdot t)+(a^{-1}\cdot y\ominus a^{-1}\cdot t)\\
&\sim (n-2)0+(\ominus (n-1)a^{-1}\cdot t)+(a^{-1}\cdot y\ominus a^{-1}\cdot t)\\
&\sim(n-2)0+0+(a^{-1}y\ominus na^{-1}\cdot t)\\
&\sim (n-1)0+(a^{-1}y\ominus na^{-1}\cdot t).
\end{split} 
\end{equation}

Let $H$ be a very ample line bundle on $E\times F$. The line bundle 
\[
L=\O_{E\times F}(\sum_{g\in G}g^*(H))
\]
is invariant by the $G$ action. For a given $x\in E$, let us set
\[
L_x=L_{|\{x\}\times F}
\]
which is an invertible sheaf of degree $n>0$ on $F$. The invariance of $L$ by the $G$ action implies that for any $x\in E$ and $g\in G$ one has
\[
L_x=\phi_{g,x}^*L_{gx}.
\]

We define the morphism $r: E\to F$ in such a way that 
\[
L_x=\O_F((n-1)0+r(x)).
\]
Then, by applying \eqref {eq:qrt}, we see that \eqref {eq:hji} holds. \end{proof}

Assume first that $n=1$ in Claim \ref {cl:kop}. Let $u$ be the automorphism of $E\times F$ defined by $u(x,y)=(x,y\ominus r(x))$. Then we have
\[
\begin{split}
ugu^{-1} (x,y)&= ug(x,y\oplus r(x))=u(gx, \phi_{g,x}(y\oplus r(x))=u(g\cdot x, a_g\cdot y \oplus  a_g\cdot r(x)\oplus t_{g,x})\\
&=(gx, a_g\cdot y \oplus  a_g\cdot r(g\cdot x)\oplus t_{g,x}\ominus r(x))=(g\cdot x, a_g\cdot y)
\end{split}
\]
the last equality following from \eqref {eq:hji} with $n=1$. So if we set $G'=uGu^{-1}$, then $G'$ acts on $E$ and $F$, and on $E\times F$ with the diagonal action, and $E\times F/G=E\times F/G'$, and the proof is completed in this case.

Consider now the case $n>1$. Consider the cartesian diagram 
\begin{equation}\label{eq:oopl}
\xymatrix{ 
&\widetilde E \ar^{\tilde r}[r]\ar_{\pi}[d]&F\ar^{\cdot n}[d] \\
&E\ar^{r}[r] & F\,\,
}
\end{equation}
where $\pi: \widetilde E\to E$ is a covering, not necessarily connected. We have that $\widetilde E$ is stable for the $G$ action. Indeed, $\widetilde E\subset E\cdot F$ consists of all pairs $(x,y)$ such that $r(x)=ny$. If we have such a pair, then
\[
g\cdot (x,y)=(g\cdot x, a_g\cdot y\oplus t_{g,x})
\]
and, by \eqref {eq:hji}, we have
\[
r(g\cdot x)=a_g\cdot r(x)\oplus n t_{g,x}=a_g\cdot (ny)\oplus nt_{g,x}=n(a_g\cdot y\oplus t_{g,x})
\]
as wanted.

Moreover the group $F[n]$ of points of $F$ of order $n$ also acts on $\widetilde E$ via the action
\[
\varepsilon \cdot (x,y)=(x, y\oplus \varepsilon).
\]
Let $G'$ be the group of automorphism of $\widetilde E$ generated by $G$ and $F[n]$. 
For any $g\in G$ and $\varepsilon\in F[n]$, one has  
\[
\begin{split}
g\varepsilon g^{-1}(x,y)&=g\varepsilon (g^{-1}\cdot x, a_{g}^ {-1}\cdot y\oplus t_{g^{-1},x})\\
&=g(g^{-1}\cdot x, a_{g}^ {-1}\cdot y\oplus t_{g^{-1},x}\oplus \varepsilon)\\
&=(x, y\oplus a_g\cdot \varepsilon \oplus a_g\cdot t_{g^{-1},x}\oplus t_{g,g^{-1}\cdot x})=(x,y\oplus a_g\cdot \varepsilon ),
\end{split}
\]
the last equality holding because of \eqref {eq:klop}. Hence we have $g\varepsilon g^{-1}= a_g\cdot \varepsilon$. This means that $G'$ is the semidirect product of $F[n]$ and $G$, and one has an exact sequence
\[
0\to F[n]\to G'\stackrel{\mathfrak u}\to G\to 1.
\]

We define an operation of $G'$ on $\widetilde E\times F$ as follows
\[
\gamma (\tilde x,y)=(\gamma\cdot \tilde x,\phi_{\mathfrak u(\gamma),\pi(\tilde x)}(y))
\]
with $\gamma\in G'$, $\tilde x\in \widetilde E$ and $y\in F$. Since $\widetilde E/F[n]\cong E$, one has $E/G\cong (\widetilde E/F[n])/G\cong \widetilde E/G'$ and therefore $\widetilde E\times F/G'\cong E\times F/G$. 

Consider now the automorphism $w$ of $\widetilde E\times F$ defined as
\[
w(\tilde x, y)=(\tilde x, y\ominus \tilde r(\tilde x))
\]
with $\tilde x\in \widetilde E, y\in F$. One has
\[
\begin{split}
w\gamma w^{-1} (\tilde x, y)&=w\gamma (\tilde x, y\oplus \tilde r(\tilde x))\\
&=w(\gamma\cdot \tilde x, \phi_{\mathfrak u(\gamma),\pi(\tilde x)}(y\oplus \tilde r(\tilde x)))\\
&=w(\gamma\cdot \tilde x, a_{\mathfrak u(\gamma)}(y\oplus \tilde r(\tilde x))\oplus t_{\mathfrak u(\gamma),\pi(\tilde x)})\\
&=(\gamma\cdot \tilde x, a_{\mathfrak u(\gamma)}(y\oplus \tilde r(\tilde x))\oplus t_{\mathfrak u(\gamma),\pi(\tilde x)}\ominus \tilde r(\gamma(\tilde x)))\\
&=(\gamma\cdot \tilde x, a_{\mathfrak u(\gamma)}\cdot y \oplus \eta_{\gamma,\tilde x})
\end{split}
\]
where
\[
\eta_{\gamma,\tilde x}= a_{\mathfrak u(\gamma)}\cdot \tilde r(\tilde x)\oplus t_{\mathfrak u(\gamma),\pi(\tilde x)}\ominus \tilde r(\gamma(\tilde x)).
\]

By looking at diagram \eqref {eq:oopl} and by \eqref {eq:hji}, we have
\[
\begin{split}
n\eta_{\gamma,\tilde x}&=a_{\mathfrak u(\gamma)}\cdot n\tilde r(\tilde x)\oplus nt_{\mathfrak u(\gamma),\pi(\tilde x)}\ominus n\tilde r(\gamma(\tilde x))\\
&=a_{\mathfrak u(\gamma)}\cdot r(\pi(\tilde x))\ominus r(\pi(\gamma\cdot \tilde x))\oplus nt_{\mathfrak u(\gamma),\pi(\tilde x)}=0,
\end{split}
\]
hence $\eta_{\gamma,\tilde x}\in F[n]$. Since $F[n]$ is finite, $\eta_{\gamma,\tilde x}$ does not depend on $\tilde x$, so we can denote it by $\eta_\gamma$. Thus we have 
\[w\gamma w^{-1} (\tilde x, y)=(\gamma\cdot \tilde x, a_{\mathfrak u(\gamma)}\cdot y \oplus \eta_{\gamma})\]
which implies that the action of $G'$ on $\widetilde E\times F$, after conjugation by $u$, is of the required form. 

If $\widetilde E$ is not connected, one may replace $\widetilde E$ by one of its connected components $E'$ and $G'$ with its subgroup fixing $E'$. \end{proof}

According to the previous lemma, up to replacing $E$ with $E'$, we may and will assume that $G$ is a finite group of translations of $E$ acting also on $F$ and diagonally on  $S'=E\times F$.

Set $C=F/G$  and note that $g(C)\leq 1$. We claim that $g(C)=0$. In fact, we have an obvious morphism $\phi: S\to C$. If $g(C)=1$, $\phi$ would factor through the Albanese map, thus if would factor through $\alpha$, which is  impossible.

Though Theorem \ref {thm:bdff}   tells in detail which pluricanonical systems are trivial for bielliptic surfaces, we show a partial result in this direction in the following lemma: 

\begin{lemma}\label{lem:triv} In the above setting one has that
either $4K_S$ or $6K_S$ is trivial,
hence $12K_S$ is trivial and $P_{12}=1$.
\end{lemma}

\begin{proof} By Theorem \ref {thm:biel} it suffices to prove that either $P_4=1$ or $P_6=1$. Note that $H^ 0((\Omega^1_{E})^ {\otimes n})$ has dimension 1 and is invariant by $G$, since the non--zero  1--form on $E$ is invariant by translations. Hence
\[
\begin{split}
H^0(nK_S)&=H^ 0((\Omega^2_S)^ {\otimes n})=H^ 0((\Omega^2_{S'})^ {\otimes n})^ G\cong \\
&\cong \Big ( H^ 0((\Omega^1_{E})^ {\otimes n}) \otimes H^ 0((\Omega^1_{F})^ {\otimes n}) \Big) ^G\cong \CC\otimes H^ 0((\Omega^1_{F})^ {\otimes n}) ^G
\end{split}
\]
and 
\[
P_n=H^ 0((\Omega^1_{F})^ {\otimes n}) ^G\cong H^ 0(\PP^1, L_n)=\ell_n+1
\]
where $L_n$ is a line bundle of degree $\ell_n$ on $\PP^1$. 

Let $\nu$ be the order of $G$ and consider the $\nu:1$ morphism $\pi: F\to \PP^1=F/G$. Let $p\in \PP^1$ be a branch point. Then $G$ acts transitively on $\pi^ {-1}(p)=\{q_1,\ldots, q_s\}$ and the stabilizers of the points $q_i$ are conjugated in $G$ for $1\leq i\leq s$. Let $r_p$, or simply $r$, be their common order. This is the \emph{ramification index} of $\pi$ at each of the points $q_1,\ldots, q_s$, also called the \emph{branching index} of $\pi$ at $p$, i.e., for $1\leq i\leq s$, there are local coordinates  $y$ at $p$ and $x_i$ at $q_i$ such that $\pi^*y=x_i^r$. Note that $\nu=sr$. 

\begin{claim}\label{cl:ram}
The degree of $L_n$ is
\[
\ell_n=-2n+\sum_{p\in \PP^1} \Big [n \Big (1-\frac 1{r_p}\Big)  \Big] 
\]
where $r_p$ is the ramification index of $\pi: F\to \PP^1$ at $p\in \PP^1$. 
\end{claim}

\begin{proof}[Proof of Claim \ref {cl:ram}] We have to look at rational $n$--tensor forms $\sigma$ on $\PP^1=F/G$ such that $\pi^*\sigma\in H^ 0(\omega_{F}^ {\otimes n})$. With the notation introduced above, we locally write around any branch point $p\in  \PP^1$
\[
\sigma=fy^ {-m}dy^{\otimes n}, \quad\mbox{with}\quad f  \quad\mbox{non--zero and}\quad m\in \ZZ.
\]
Then around any point $q_i\in \pi^ {-1}(p)$, for $1\leqslant i\leqslant h$, we have
\[
\pi^ *\sigma=(\pi^*f) x_i^{-rm}(rx_i^{r-1}dx_i)^ {\otimes n}=f_i x_i^{-rm+(r-1)n}dx_i^ {\otimes n}
\]
where we set $r\pi^*f:=f_i$, which is a non--zero function. So $\pi^ *\sigma$ is holomorphic if and only if $-rm+(r-1)n\geq 0$, i.e., if and only if
\[
m\leq n\Big (1- \frac 1r\Big).
\]
Hence $\sigma$ is a $n$--tensor form with poles of order at most $n\Big (1- \frac 1{r_p}\Big)$ at $p\in \PP^1$. \end{proof}

Now we can finish the proof of the lemma. We have to prove that either $\ell_4\geq 0$ or $\ell_6\geq 0$. Let $h$ be the number of branch points of $\pi$ and let $r_1, \ldots, r_h$ be the corresponding branching indices, with $r_1\leq \ldots\leq r_h$. By the Riemann--Hurwitz formula we have
\[
0=2g(F)-2=-2\nu+\sum_{i=1}^h \nu\Big(1-\frac 1 {r_i}\Big),
\]
hence
\begin{equation}\label{eq:olp}
\sum_{i=1}^h \Big(1- \frac 1 {r_i}\Big)=2.
\end{equation}
Moreover, since $r_p\geq 2$ for any branch point $p\in \PP^1$, we have
\[
2\Big(1- \frac 1 {r_i}\Big)\geq 1.
\]

First we assume $h\geq 4$. Then
\[
P_2=\ell_2+1=-4+\sum_{i=1}^h\Big [ 2\big( 1-\frac 1 {r_i}  \big) \Big] +1 \geq 1, \quad \mbox{hence}\quad P_4=P_6=1,
\]
proving the lemma in this case.

Next notice that $h\leq 2$ is not possible because of \eqref {eq:olp}, so we only have to discuss the case $h=3$. Then \eqref {eq:olp} reads
\begin{equation}\label{eq:sum}
\frac 1{r_1}+\frac 1{r_2}+\frac 1 {r_3}=1.
\end{equation}

If $r_1\geq 3$, we have
\[
3\Big (1- \frac 1{r_i} \Big)\geq 2, \quad \mbox{for}\quad i=1,2,3,
\]
hence 
\[
P_3=\ell_3+1=-6+\sum_{i=1}^3\Big [3 \big( 1-\frac 1 {r_i}  \big) \Big]+1 \geq 1, \quad \mbox{hence}\quad P_6=1,
\]
again proving the lemma in this case. 

If $r_2\geq 4$, by \eqref {eq:sum} we have $r_1=2, r_2=r_4=4$. This implies $\ell_4=0$ hence $P_4=1$. Finally we can only have $r_2=3$, and by \eqref {eq:sum} only the following cases are possible
\[
\begin{split}
&(r_1,r_2,r_3)=(3,3,3) \quad \mbox{hence} \quad \ell_3=0  \quad \mbox{thus} \quad P_3=1\\
&(r_1,r_2,r_3)=(2,3,6) \quad \mbox{hence} \quad \ell_6=0  \quad \mbox{thus} \quad P_6=1
\end{split}
\]
and in both cases $P_6=1$. This ends the proof of the lemma. \end{proof}

For the proof of Theorem \ref {thm:bdff} we need a well known preliminary result (see, e.g., \cite [Chapt. IV, Cor. 4.7]{Hart}):

\begin{thm}\label{thm:ell} Let $E$ be a smooth elliptic curve, which we look at as an abelian variety of dimenion 1. Let $j$ be the $j$--invariant of $E$. Denote by $\Aut(E)$ the group of all automorphisms of $E$ and by $\Aut_0(E)$ the group of automorphism of $E$ as an abelian variety. Then $\Aut_0(E)$ is a finite group and precisely:\\
\begin{inparaenum}
\item [(i)] $\ZZ_2$, generated by the symmetry $x\to -x$ if $j\neq 0, 1728$;\\
\item [(ii)] $\ZZ_4$, generated by the automorphism $x\to ix$, if $j=1728$, i.e., if $E=\CC/\ZZ+i\ZZ$;\\
\item [(iii)] $\ZZ_6$, generated by the automorphism $x\to - \rho x$, with $\rho$ a non--trivial cubic root, if $j=0$, i.e., if $E=\CC/\ZZ+\rho\ZZ$.
\end{inparaenum}

Moreover $\Aut(E)=E \ltimes \Aut_0(E)$, where $E$ is identified with the group of translations of $E$.  
\end{thm}

We can finally give the:

\begin{proof}[Proof of Theorem \ref {thm:bdff}] By Lemma \ref {lem:bb} and the subsequent discussion,  a bielliptic  surface $S$ is a quotient $E\times F/G$, where $E,F$ are elliptic curves, $G\subset E$ is a finite (abelian) group of translations (so that $E/G$ is an elliptic curve), and $G$ is also a subgroup of $\Aut(F)$, such that $F/G\cong \PP^1$. By the last assertion of Theorem \ref {thm:ell}, $G$ is the direct product $T\times H$ where $T$ is a finite subgroup of translations of $F$ and $H$ is a subgroup of $\Aut_0(F)$. Since 
 $F/G\cong \PP^1$, then $H\neq 0$, and therefore, by Theorem \ref {thm:ell}, one has $H=\ZZ_p$, with $p=2,3,4,6$. Since $G$ is abelian, the elements of $T$ commute with those of $H$, i.e., they are translations by the fixed points of the action by $H$. It is immediate to find these fixed points:\\
\begin{inparaenum}
\item [(a)] if $H=\ZZ_2$ generated by the symmetry $x\to -x$, the fixed points are the points of order 2 of $F$, whose set we denote by $F[2]$;\\
\item [(b)] if $H=\ZZ_4$, generated by the automorphism $x\to ix$, and $F=\CC/\ZZ+i\ZZ$,
the fixed points are (the classes of) $0$ and $\frac {1+i}2$;\\
\item [(c)] if $H=\ZZ_3$, generated by the automorphism $x\to \rho x$, with $\rho$ a non--trivial cubic root, and  $F=\CC/\ZZ+\rho\ZZ$, the fixed points are (the classes of) $0$ and $\pm \frac {1-\rho}3$;\\
\item [(d)] if $H=\ZZ_6$, generated by the automorphism $x\to  -\rho x$, with $\rho$ a non--trivial cubic root, and $E=\CC/\ZZ+\rho\ZZ$, then the only fixed point is (the class of) $0$.\\
 \end{inparaenum}
 Finally $G=T\times H$ is a subgroup of $E$, which excludes the possibility that $G=F[2]\times \ZZ_2$. The cases (i)--(viii) in the statement are now an easy consequence of the above considerations. 
 
 As for the final assertion of the theorem, let $\omega$ be a non--zero holomorphic $2$--form on $E\times F$. Then the minimum $n$ such that $nK_S$ is trivial, is the minimum $n$ such that $G$ acts trivially on $\omega^n$. It is then clear that the assertion follows.  
\end{proof}

\section{The $P_{12}$--Theorem}\label{sec:p12}

In this section we prove the:

\begin{thm}[The $P_{12}$--Theorem]\label{thm:p12} Let $S$ be a surface. Then:\\
\begin{inparaenum}
\item [(i)] $\kappa=-\infty$ if and only if $P_{12}=0$;\\
\item [(ii)] $\kappa=0$ if and only if $P_{12}=1$;\\
\item [(iii)]  $\kappa=1$ if and only if $P_{12}\geq 2$ and $K_S^2=0$ for $S$ minimal;\\
\item [(iv)]  $\kappa=2$ if and only if $P_{12}\geq 2$ and $K_S^2>0$ for $S$ minimal.\\
\end{inparaenum}
More precisely:\\
\begin{inparaenum}
\item [(a)] $P_{12}=0$ if and only if $S$ is {ruled}, i.e., birationally equivalent to $C\times \PP^1$, where $C$ is a smooth curve of genus $g(C)=q(C)$; if $S$ is minimal, then either $S=\PP^2$ or $S$ is a scroll over a curve;\\
\item [(b)] $P_{12}=1$ if and only if $12K_S$ is trivial, with the following subcases for $S$ minimal:\\
\begin{inparaenum}
\item [(I)] $p_g=1$, $q=2$ if and only if $S$ is an abelian surfaces, in which case $K_S$ is trivial;\\
\item [(II)] $p_g=1$, $q=0$ if and only if $S$ is a K3 surface, in which case $K_S$ is trivial;\\
\item [(III)] $p_g=0$, $q=0$ if and only if $S$ is an Enriques surface, in which case $K_S$ is non--trivial, but $2K_S$ is trivial;\\
\item [(IV)] $p_g=0$, $q=1$ if and only if $S$ is  a bielliptic surface, in which case $K_S$ is non--trivial, but $nK_S$ is trivial, for $n\in \{2,3,4,6\}$.\\
\end{inparaenum}
\item [(c)] $P_{12}\ge 2$ and $K_S^2=0$ for $S$ minimal if and only if $S$ is properly elliptic and $\varphi_{|12K_S|}$ provides an elliptic fibration over a curve (up to eliminating fixed components and up to Stein factorization);\\
\item [(d)] $P_{12}\ge 2$ and $K_S^2>0$ for $S$ minimal if and only if $S$ of general type, in which case $\varphi_{|nK_S|}$ is a morphism which is birational onto its image, as soon as $n\geq 6$.
\end{inparaenum}
\end{thm}

\begin {proof}

The second part of the theorem follows from the classification we already made. So it suffices to prove (i)--(iv).
The proof follows the one by Catanese--Li in \cite{CL} and by Francia \cite {F} and it is a case by case analysis which follows the classification. 

If $\kappa=-\infty$ then $P_n=0$ for all $n\in \NN$, in particular $P_{12}=0$. Then  (i) follows as soon as we prove (ii)--(iv). If $\kappa=0$, following the classification, we have the cases (I)--(IV) listed in the statement of the Theorem. For all of them we have $P_{12}=1$. Then (ii) follows once we prove (iii) and (iv). On the other hand, if $\kappa=2$ and if $S$ is minimal we have \eqref {eq:plur} and $K^2_S>0$, hence $P_{12}\geq 2$. Conversely, if $S$ is minimal, $P_{12}\geq 2$ and $K^2_S>0$, then $\kappa=2$. Hence we only have to consider case (iii), i.e., the $\kappa=1$ case.

If $P_{12}\geq 2$, then $\kappa\geq 1$. Moreover, if $K_S^2=0$ on a minimal model $S$, then $\kappa<2$, hence we have $\kappa=1$. Thus  the theorem follows from the:

\begin{proposition}\label{prop:k1} If $\kappa=1$ one has:\\
\begin{inparaenum}
\item [(1)] $P_{12}\geq 2$;\\
\item [(2)] there a positive integer  $n\leq 4$ such that $P_n\neq 0$;\\
\item [(3)] there a positive integer  $n\leq 8$ such that $P_n\geq 2$;\\
\item [(4)] for all integers $n\geq 14$ one has $P_n\geq 2$.
\end{inparaenum}
\end{proposition}

\begin{example}\label{rem:pu} The bounds in Proposition \ref {prop:k1}  are sharp, as the following examples show. 

\noindent  {\bf (i) Cases (2) and (4).}  Consider the group $G=\ZZ_2\oplus \ZZ_6$, which is a subgroup of translations of any elliptic curve $C_1$, in such a way that $C_1/G$ is still elliptic. 
 
 By Riemann's Existence Theorem \ref {thm:ret2}  there is a curve $C_2$ with a $G$ action, such that $C_2/G=\PP^1$ and the quotient morphism $f: C_2\to \PP^1$ of degree $m=12$ is branched at three points $p_1,p_2,p_3$ with local monodromies
 \[
 g_1=(1,0), \quad g_2=(0,1), \quad g_3=(1,5)
 \]
 with respective orders $m_1=2$, $m_2=6$, $m_3=6$. Thus over $p_1$ there are 6 points with simple ramification, and over $p_2$ and $p_3$ there are 2 points with ramification index 6. By Riemann--Hurwitz formula one has 
 \[
 22+2g(C_2)=2m+2g(C_2)-2=26\quad \mbox{hence}\quad g(C_2)=2.
 \]
 
 Now $G$ acts diagonally on $C_1\times C_2$ and let $S=C_1\times C_2/G$. Then we have the quotient map
 $S\to  C_2/G=\PP^1$ with elliptic fibres and with exactly 3 multiple fibres with multiplicities $2, 6, 6$. 
 Moreover $K^2_S=e(S)=0$. Hence, according to the canonical bundle formula for elliptic fibrations, we have
 \[
 P_n(S)=\max \{0, -2n+1+\Big [ \frac n2\Big] +2 \Big [ \frac {5n}6 \Big]\}
 \]
 thus 
 \[
 P_1=P_2=P_3=0, P_4=P_5=1, P_6=2, P_{13}=1.
 \]
 
\noindent  {\bf (ii) Cases (2) and (3).}  Here we set $G=\ZZ_{10}$ which acts on any elliptic curve $C_1$ with elliptic quotient $C_1/G$. By Riemann's Existence Thoerem, we find a curve $C_2$ with a $G$ action such that $C_2/G=\PP^1$ and the degree 10 quotient map $ C_2\to \PP^1$ is ramified at 3 points $p_1,p_2,p_3$ with local monodromies 
\[
g_1=5, \quad g_2=4, \quad g_3= 1
\]
 with respective orders $m_1=2$, $m_2=5$, $m_3=10$. Thus over $p_1$ there are 5 points with simple ramification, over $p_2$ there are 2 points with ramification index 5 and over $p_3$ there is one point with total ramification index. Riemann--Hurwitz formula
 says 
 \[
 18+2g(C_2)=2m+2g(C_2)-2=22\quad \mbox{hence}\quad g(C_2)=2.
 \]
 As above,  $G$ acts diagonally on $C_1\times C_2$ and we set $S=C_1\times C_2/G$. Then we have an obvious map
 $S\to  C_2/G=\PP^1$ with elliptic fibres and with exactly 3 multiple fibres with multiplicities $2, 5, 10$. 
 Again $K^2_S=e(S)=0$ and therefore
 \[
 P_n(S)=\max \{0, -2n+1+\Big [ \frac n2\Big] + \Big [ \frac {4n}5 \Big]+ \Big [ \frac {9n}{10} \Big]\}
 \]
 thus 
 \[
 P_1=P_2=P_3=0, P_4=P_5=P_6=P_7=1, P_8=P_9=2, P_{10}=3, P_{11}=1, P_{12}=P_{13}=2.
  \]
\end{example} 

\begin{proof} [Proof of Proposition \ref {prop:k1}] By Castelnuovo--De Franchis--Enriques' Theorem, we have $\chi\geq 0$ and $K_S^2=0$ if $S$ is a minimal model. Suppose we have an elliptic fibration $f: S\to C$, with $g=g(C)$. 

First suppose $\chi\geq 1$ and $g\geq 1$. Then  the canonical bundle formula for elliptic fibrations implies that
\[
P_n\geq n(2g-2+\chi)-g+1=n\chi+(g-1)(2n-1)\geq n
\]
proving the assertion. Next assume $\chi=0$ and $g\geq 2$. By the same computation we get $P_n\geq 2n-1$, again proving the assertion. If $\chi=0, g=1$, then there are multiple fibres of the elliptic fibration, otherwise  the canonical bundle formula for elliptic fibrations would imply that $K_S\equiv 0$, which is not possible. Then we have
\begin {equation}\label{eq:lof}
nK_S  \sim  \eta+ n\sum_{i=1}^h(m_i-1)F_i\equiv \eta+\sum_{i=1}^h \Big [ n\frac {m_i-1}{m_i}\Big ] \cdot F+D
\end{equation}
where $F$ is a general fibre of $f$,  $h>0$, $\eta$ is the pull--back to $S$ via $f$ of a degree 0 line bundle on $C$ and $D$ is an effective, integral divisor. Hence
\begin{equation}\label{eq:pirt}
P_n\geq \sum_{i=1}^h \Big [ n\frac {m_i-1}{m_i}\Big ] \geq \Big [\frac n2\Big]
\end{equation}
whence the assertion follows. 

So we are left with the case $g=0$.  If $\chi\geq 3$, then again 
\[
P_n\geq n(\chi-2)+1\geq n+1
\] 
which proves the assertion. If $\chi=2$ the same argument as above tells us that \eqref {eq:lof} and \eqref {eq:pirt} hold, proving the assertion again. So we are left with $g=0$ and $0\leq \chi\leq 1$. In both cases the 
canonical bundle formula for elliptic fibrations implies that $p_g=0$. 

\begin{case}\label{case:uno} $g=0$, $\chi=1$
\end{case}
We have
\[
K_S\sim -F+\sum_{i=1}^h (m_i-1)F_i.
\]
Since $\kappa=1$, then some multiple of  $K_S$  is equivalent to an effective divisor. Hence for any ample divisor $A$ on $S$, we have
\begin{equation}\label{eq:ret}
0<A\cdot K_S=\Big (-1+\sum _{i=1}^h \frac {m_i-1}{m_i}\Big)F\cdot A \quad \mbox{hence}\quad h-1 >\sum _{i=1}^h\frac 1{m_i}.
\end{equation}
This implies $h\geq 2$. If $h\geq 3$ we have
\[
nK_S\sim -nF+\sum_{i=1}^h\Big [ n\frac {m_i-1}{m_i} \Big] F+D
\]
where $D$ is an effective, integral divisor. On the right hand side the coefficient of $F$ is
\[
\begin{split}
&-n+\Big [ n \Big (1-\frac 1{m_1}\Big)\Big ] +\cdots +\Big [ n \Big (1-\frac 1{m_h}\Big)\Big ]\geq \\
&\geq -n+h \Big [ \frac n2 \Big ]\geq \Big [ \frac n2 \Big ]
\end{split}
\]
proving the assertion. If $h=2$, we must have $m_i\geq 3$ for at least an $i\in \{1,2\}$. Hence
\[
P_n\geq  \Big [ \frac n2 \Big ]+  \Big [ \frac {2n}3 \Big ]-n+1
\]
hence $P_n\geq 1$ for $n\geq 3$ and $P_n\geq 2$ for $n\geq 8$, proving again the assertion.

\begin{case}\label{case:due} $g=0$, $\chi=0$
\end{case}

Now we have
\begin{equation}\label{eq:hot}
K_S\sim -2F+\sum_{i=1}^ h(m_i-1)F_i
\end{equation}
where we assume $m_1\leq\ldots\leq m_h$. Recall that
\begin{equation}\label{eq:orc}
P_n\geq \max\Big \{ 0, 1-2n+\sum_{i=1}^h \Big [  n\frac {m_i-1}{m_i}\Big]\Big \}.
\end{equation}

If $h\geq 5$, we have
\[
P_n\geq 1-2n+5\Big [  \frac n2\Big].
\]
If $n=2m+\epsilon$, with $0\leq \epsilon\leq 1$, we get
\[
P_n\geq 1-4m+5m-2\epsilon=1+m-2\epsilon\geq \begin{cases} 1 &\quad\mbox{for}\quad n\geq 4\\
2& \quad \mbox {for}\quad n\geq 6 \end{cases}
\]
whence the assertion follows.

Next we assume $h\leq 4$. Similarly to \eqref {eq:ret}, we get
\begin{equation}\label{eq:tru}
 h-2 >\sum _{i=1}^h\frac 1{m_i},
\end{equation}
which implies $h\geq 3$. Furthermore we make the following:

\begin{claim}\label{cl:piff}
For each $i\in\{1,\ldots, h\}$, $m_i$ divides the least common multiple of $m_1,\ldots$, $m_{i-1},m_{i+1},\dots, m_h$.
\end{claim}

Let us take this claim for granted for the time being. Assume first $h=4$. Observe that the right hand side of \eqref {eq:orc} is increasing in the $m_i$s. The worst case for the sequence of the $m_i$s is  $(2,2,3,3)$, because $(2,2,2,2)$ is not a solution of \eqref {eq:tru} and $(2,2,2,3)$ does not verify Claim \ref {cl:piff}. For $(2,2,3,3)$ we obtain
\[
P_n\geq 1-2n+2\Big [  \frac n2\Big]+2 \Big [  \frac {2n}3\Big]=1+\Big(2\Big [  \frac n2\Big]-n  \Big)+\Big(2 \Big [  \frac {2n}3\Big]-n  \Big).
\]
If $n=2k$ we get 
\[
P_n=1+2\Big [  \frac k3 \Big]\geq \begin{cases} 1\quad \mbox{if} \quad k\geq 1\\
3\quad \mbox{if}\quad k\geq 3
\end{cases}
\]
If $n=2k+1$ we get
\[
P_n=2\Big [  \frac {4k+2} 3 \Big]-2k-1 = 2\Big [  \frac {k+2} 3 \Big]-1
\geq \begin{cases} 1\quad \mbox{if} \quad k\geq 1\\
3\quad \mbox{if}\quad k\geq 4
\end{cases}
\]
proving the assertion. 

Finally we assume $h=3$. Suppose that $m_1\geq 4$. Then, as above, the worst case for the sequence of the $m_i$s is $(4,4,4)$. Then set $n=4k+\epsilon$, with $0\leq \epsilon\leq 3$. Hence
\[
\begin{split}
&P_n\geq 3\Big [  \frac {3n} 4 \Big]-2n+1=3\Big ( 3k+  \Big [  \frac {3\epsilon} 4 \Big]\Big)-8k-2\epsilon+1=\\
&=1+k+3\Big [  \frac {3\epsilon} 4 \Big]-2\epsilon=\begin{cases} k+1 \quad \mbox {for}\quad \epsilon = 0,3\\ 
k  \quad\quad\,\,\,\, \mbox {for}\quad \epsilon = 2\\ k-1 \quad \mbox {for}\quad \epsilon = 1. \end{cases}
\end{split}
\]
Thus $P_3=1, P_4=2, P_n\geq 2$ for $n\geq 10$, proving the assertion. 

Assume next $m_1=3$. Since $3$ divides the least common multiple of $m_2$ and $m_3$, then either $3$ divides both $m_2$ and $m_3$ or it divides only one of the two. So we have the following possibilities for $m_2,m_3$:\\
\begin{inparaenum}
\item [$\bullet$]  $(m_2,m_3)=(3a,3b)$: then $a$ divides $b$ and $b$ divides $a$, thus $a=b$ and the worst case is $(3,6,6)$;\\
\item [$\bullet$] $\{m_2,m_3\}=\{c,3b\}$, with $c$ not divisible by $3$: then $3b$ divides $3c$, hence $b$ divides $c$ and moreover $c$ divides $b$, thus $b=c$ and the worst case is $(3,4,12)$.
\end{inparaenum}
Then we have:
\[
\mbox{in the case}\quad (3,6,6): P_n=\max\{0, F(n)\}, \quad \mbox{with}\quad F(n):= 1-2n +\Big [  \frac {2n} 3 \Big]+2\Big [  \frac {5n} 6 \Big]
\]
and 
\[
\mbox{in the case}\quad (3,4,12): P_n=\max\{0, G(n)\}, \quad \mbox{with}\quad G(n):= 1-2n +\Big [  \frac {2n} 3 \Big]+\Big [  \frac {3n} 4 \Big]+\Big [  \frac {11n} {12} \Big].
\]

In the former case, write $n=6k+\epsilon$, with $0\leq \epsilon\leq 5$. We get
\[
F(n)=2k+F(\epsilon), \quad \mbox{and}\quad F(\epsilon)=1,-1, 0, 1,1,2 \quad \mbox{for}\quad  \epsilon=0,\ldots, 5,
\]
thus $P_3\geq 1$, $P_5\geq 2$ and $P_n\geq 2$ for $n\geq 8$, proving the assertion.

In the latter case,  write $n=12k+\epsilon$, with $0\leq \epsilon\leq 11$. We get
\[
G(n)=4k+G(\epsilon),  \quad \mbox{and}\quad G(\epsilon)=1,-1, 0, 1,1,1, 2,2 , 3,3,3, 4 \quad \mbox{for}\quad  \epsilon=0,\ldots, 11,
\]
thus $P_n\geq 1$ for $n\geq 3$ and $P_n\geq 2$ for $n\geq 6$, proving the assertion.

Finally, we assume $m_1=2$. Then one of $m_2,m_3$ is even.  In case $m_2=2a,m_3=2b$, we have $a=b$ and the worst case is $(2,6,6)$  which was considered in Example \ref {rem:pu}, (i). If $\{m_2,m_3\}=\{c, 2b\}$, with $c$ odd, then $c$ has to divide $b$ and $2b$ divides $2c$, hence $b$ divides $c$, thus $b=c$. The worst case is then $(2,5,10)$ which was considered in Example \ref {rem:pu}, (ii).

\medskip
In conclusion we have to give the:

\begin{proof}[Proof of Claim \ref {cl:piff}] First, note that by \eqref {eq:hot}, we have $p_g=0$. Since $\chi=0$, we have $q=1$, hence the Albanese map is a morphism $\alpha: S\to E$, with $E$ an elliptic curve, with fibres of genus $g>1$, because, $S$ having $\kappa=1$, it cannot have two different elliptic fibrations. Note also that, since all fibres of the elliptic pencil $f: S\to \PP^1$ map to $E$ via $\alpha$, no fibre of $f$ can be rational, hence all fibres of $f$ are smooth, except the multiple fibres. 
We denote by $D$ the general fibre of the Albanese map and, as usual, by $F$ the general fibre of $f$. More specifically, if $y\in E$ [resp., $x\in \PP^1$] , we denote by $D_y$ [resp., by $F_x$] the fibre of $\alpha$ [resp., of $f$] over $y$ [resp., over $x$]. 

Note that the general Albanese fibre $D$ maps to $\PP^1$ via $f_{|D}$. 
For each point $p\in D$ we have that $p$ is a ramification point of order $\nu$ for $f_{|D}: D\to \PP^1$ if and only if the intersection multiplicity of $D$ and $F_x$ is $\nu$, with $x=f(p)$. Let now $F_x$ be reduced.
The morphism $\alpha_{|F_x}:F_x\to E$ cannot be ramified at $p$, because it is the map between two elliptic curves. Thus the only ramification points of $f_{|D}: D\to \PP^1$ are the intersections of $D$ with the multiple fibres $m_iF_i$ of $f$, and their ramification order is exactly $m_i$, with $1\leq i\leq h$. Let $x_1,\ldots, x_h\in \PP^1$ the points such that $F_{x_i}=m_iF_i$, for $1\leq i\leq h$. 

\begin{claim}\label{cl:cl:subcl}   The cover $f_{|D}: D\to \PP^1$ is Galois, with an abelian monodromy group.\end{claim}

\begin{proof} [Proof of Claim \ref {cl:cl:subcl}] For $x\neq x_i$, with $1\leq i\leq h$, consider the smooth elliptic curve $F_x$ and the \'etale cover $\alpha_{|F_x}: F_x\to E$, which has degree $n=D\cdot F$ and is Galois with abelian monodromy group, because $\pi_1(E)$ is abelian. We abuse notation and denote $\alpha_{|F_x}$ simply by $\alpha$. 

First we show that the subgroup $\alpha_*(\pi_1(F_x))$ of $\pi_1(E)$ does not depend on $x$. Indeed, if $x'\neq x$, choose a path $\gamma: [0,1]\to \PP^1-\{x_1,\ldots, x_h\}$ such that $\gamma(0)=x$ and $\gamma(1)=x'$. Consider the map
\[
q\in f^{-1}(\gamma)\to (\alpha(q), f(q))\in E\times \gamma. 
\]
If $c$ is a loop in  $\alpha_*(\pi_1(F_x))$, let $\sigma(c\times \gamma)$ be a  lifting on $f^{-1}(\gamma)$ of the cycle $c\times \gamma$. Since $\sigma(c\times \gamma(0))$ is a closed path, the same happens for its homotopic path $\sigma(c\times \gamma(1))$, and this gives us a map
$\alpha_*(\pi_1(F_x))\to \alpha_*(\pi_1(F_{x'}))$ which is clearly an identification. 

Consider then the group
\[
G=\frac {\pi_1(E)}{\alpha_*(\pi_1(F_x))}
\]
which does not depend on $x\neq x_i$, for $1\leq i\leq h$. This group acts on the fibres $F_x$ in the following way
\[
(g,q)\in G\times F_x\mapsto \sigma_x(g(1))\in F_x.
\]
where $g:[0,1]\to E$ is a loop on $E$ and  $\sigma_x(g)$ is the lifting on $F_x$ of the loop $g$ such that $\sigma_x(g(0))=q$. In particular we obtain an action of $G$ on the fibres of the \'etale cover 
\[f: D-\{f^{-1}(x_1),\ldots, f^{-1}(x_h)\}\to \PP^1-\{x_1,\ldots, x_h\}.\]
 Since $G$ acts freely, because of the way we defined the action, we get an injective homomorphism
\begin{equation}\label{eq:sprof}
G\hookrightarrow \Aut\Big ( D-\{f^{-1}(x_1),\ldots, f^{-1}(x_h)\} \Big).
\end{equation}
Since the curves $F_x$ are irreducible, $G$ acts transitively on the fibres, hence also
$\Aut\Big ( D-\{f^{-1}(x_1),\ldots, f^{-1}(x_h)\} \Big)$ acts transitively, i.e., 
\[f: D-\{f^{-1}(x_1),\ldots, f^{-1}(x_h)\}\to \PP^1-\{x_1,\ldots, x_h\}\]
 is Galois with group $G$.\end{proof}

As a consequence of the claim, the homorphism in \eqref {eq:sprof} is an isomorphism. Moreover we have
\[
 \Aut\Big ( D-\{f^{-1}(x_1),\ldots, f^{-1}(x_h)\} \Big)\cong \frac {\pi_1( \PP^1-\{x_1,\ldots, x_h\})}{f_*\pi_1\Big(  D-\{f^{-1}(x_1),\ldots, f^{-1}(x_h)\}\Big)}.
\]
Hence finally
\[
G\cong \frac {\pi_1( \PP^1-\{x_1,\ldots, x_h\})}{f_*\pi_1\Big(  D-\{f^{-1}(x_1),\ldots, f^{-1}(x_h)\}\Big)}
\]
and since $G$ is abelian, the same happens for the group on the right hand side. 

Let us look at this group more closely. First of all, $\pi_1( \PP^1-\{x_1,\ldots, x_h\})$ is the free group generated by $h$ loops $\gamma_1,\ldots, \gamma_h$ around the points $x_1,\ldots, x_h$ respectively, with the only relation $\gamma_1\cdots\gamma_h={\rm id}$. 
Moreover, since the ramification order at $x_i$ is $m_i$, we have that $\gamma_i^r\in f_*\pi_1\Big( D- \{ f^{-1}(x_1),\ldots, f^{-1}(x_h) \} \Big)$  if and only if $r\equiv 0$ mod. $m_i$. Hence in  the group
\[
\frac {\pi_1( \PP^1-\{x_1,\ldots, x_h\})}{f_*\pi_1\Big(  D-\{f^{-1}(x_1),\ldots, f^{-1}(x_h)\}\Big)}
\]  
we have the relations
\[
\gamma_1\dots\gamma_h={\rm id}, \quad \gamma_i^{m_i}={\rm id}, \quad  \text {for}\quad 1\leq i\leq h. 
\]
Since the group in question is abelian, we have
\[
\gamma_i^{-1}=\gamma_1\cdots \gamma_{i-1}\gamma_{i+1}\cdots\gamma_h,  \quad  \text {for}\quad 1\leq i\leq h.
\]
If $M_i$ is the least common multiple of $m_1,\ldots m_{i-1}, m_{i+1}\cdots m_h$ we have
\[
(\gamma_i^{-1})^{M_i}=(\gamma_1\cdots \gamma_{i-1}\gamma_{i+1}\cdots\gamma_h)^ {M_i}=\text {id},
\]
hence $m_i$ has to divide $M_i$, for all $1\leq i\leq h$.\end{proof}

This ends the proof of Proposition \ref {prop:k1}.  \end{proof} 

Hence also the proof of Theorem  \ref {thm:p12} is finished.\end{proof}

\section{The Sarkisov's Programme}\label{sec:sark}
This section is devoted to describe the birational maps between Mori fibre spaces in the rational case. This goes under the name of Sarkisov's Programme. It is an essential complement of the minimal model programme, since it clarifies the relation between the various outcomes of the programme when they are Mori fibre spaces.

\subsection{Sarkisov's links}\label{ssec:sarklink} The Sarkisov's links are diagrams of maps, which are the basic bricks for the construction of birational maps between Mori fibre spaces. They are as follows:\medskip

{\bf Type I}: 
\[
\xymatrix{ 
&\PP^2 \ar[d]&\mathbb F_1\ar[l]\ar[d] \\
&p&\ar[l]  \PP^1
}
\]
where $p$ is a point, $\mathbb F_1\to \PP^2$ is the blow--up of a point and $\mathbb F_1\to \PP^1$ is the obvious map which makes $\mathbb F_1$ a scroll over $\PP^1$. \medskip

{\bf Type III} (inverse of type I):
\[
\xymatrix{ 
&\mathbb F_1 \ar[d]\ar[r]&\PP^2\ar[d] \\
& \PP^1\ar[r]&  p
}
\]
with the same meaning of the maps as above.\medskip

{\bf Type II}:
\[
\xymatrix{ 
&S \ar[d]&Z\ar[r]\ar[l]&S'\ar[d] \\
& C&=&  C
}
\]
where $S\to C$ and $S'\to C$ are $\PP^1$--bundles, and $S\leftarrow Z\rightarrow S'$ is an elementary transformation (see Example \ref {ex:2}). \medskip

{\bf Type IV}:
\[
\xymatrix{ 
&\mathbb F_0=\PP^1\times \PP^1\ar[d]_{p_1}&=&\PP^1\times \PP^1=\mathbb F_0\ar[d]^{p_2} \\
& \PP^1\ar[r]& p&\ar[l]  \PP^1
}
\]
where $p$ is a point, and $p_1$ and $p_2$ are the projections onto the first and second factor of the product.

\subsection{Noether--Castelnuovo's Theorem: statement}
The main objective of this section is to prove the following:

\begin{thm}[Noether--Castelnuovo's Theorem]\label{thm:NC} Suppose we have
\begin{equation}\label{eq:star}
\xymatrix{ 
&S\ar[d]_{p} \ar@{-->}[r]^{\phi}&S'\ar[d]^{p'} \\
& C& C'
}
\end{equation}
where $S\to C$ and $S'\to C'$ are Mori fibre spaces and $\phi: S\dasharrow S'$ is a birational map. Then $\phi$ can be factored through a finite sequence of Sarkisov's links. 
\end{thm}

The rest of the section is devoted to prove this theorem.

\subsection{The Sarkisov's degree}\label{ssec:sdeg} Consider a situation as in \eqref {eq:star}. 

\begin{lemma}\label{lem:pac}
There is an ample divisor $A'$ on $C'$ such that $-K_{S'}+p'^*(A')$ is ample on $S'$.
\end{lemma}
\begin{proof}
If $C'$ is a point, then $S'=\PP^2$ and the assertion is clear. Suppose that $C'$ is a curve, so that $S'$ is a $\PP^1$--bundle over $C'$. Then $\Num(S')\cong \mathbb Z^2$, generated by the classes of a section $\Gamma$ of $S'\to C'$ and of a fibre $F$ of   $S'\to C'$. The class of $F$ generates an extremal ray of $\overline {\rm NE}(S')$, and the other extremal ray of $\overline {\rm NE}(S')$ is generated by the class of $a\Gamma-hF$, with $a,h$ suitable positive integers.  

We have $K_{S'}\equiv -2\Gamma+xF$, with $x\in \mathbb Z$, and therefore the assertion is proved if we prove that there is a $y\gg 0$ such that $yF+2\Gamma$ is ample. Now
$F\cdot (yF+2\Gamma)=2$ and $(yF+2\Gamma)\cdot (a\Gamma-hF)=ay-2a\Gamma^2-2h$, which is positive for $y\gg 0$. Hence $yF+2\Gamma$ is positive on  $\overline {\rm NE}(S')$ for $y\gg 0$, which proves the assertion.
\end{proof}

We fix once and for all a very ample line bundle $H'$ of the form 
\[H'=-\mu'K_{S'}+p'^*(A'),\quad \text{ with} \quad \mu'\gg 0\]
and $A'$ very ample on $S'$ (see Lemma \ref {lem:pac}), and we set $\L'=|H'|$. Next, consider a resolution of indeterminacies 
\begin{equation}\label{eq:per}
\xymatrix{ 
&X\ar_{\sigma}[d]\ar^{\sigma'}[dr] &\\
&S\ar@{-->}^{\phi}[r] &  S'&
}
\end{equation}
We may assume that $\sigma'$ is composed with the minimal number of blow--ups. In particular, if $\sigma'$ is not the identity there is a curve on $X$ contracted by $\sigma'$ but not by $\sigma$.

For each divisor $D'\in \L'$ we define its \emph{transformed} via $\phi$ to be $D:=\sigma_*\sigma'^*(D')$. As $D'$ varies in $\L'$, $D$ varies in a linear system $\L$, called the \emph{transformed} linear system of $\L'$. Note that:\\
\begin{inparaenum}
\item[$\bullet$] $\L$ can be incomplete;\\
\item[$\bullet$] $\L$ depends only on $\phi$ and not on the resolution of the indeterminacies \eqref {eq:per};\\ 
\item[$\bullet$] $\L$ has no  base curve, hence it is nef.
\end{inparaenum} 

Next we define the \emph{quasi--effective threshold} of $\phi$ to be the real number $\mu$ such that for every curve $F$ contracted by $p$ one has 
\[
(\mu K_S+D)\cdot F=0.
\]
Note that:\\
\begin{inparaenum}
\item[$\bullet$] $\mu$ is rational and it is positive because $K_S\cdot F<0$ and $D$ is nef;\\
\item[$\bullet$] $\mu$ depends only on $\L$ and not on the  curve $D\in \L$;\\
\item[$\bullet$] since $F$ generates an extremal ray which is also generated by the class of a smooth rational curve $\ell$ such that $-3\leq \ell\cdot K_S\leq -2$, $\mu$ is defined by the relation
\[
(\mu K_S+D)\cdot \ell=0 \quad \text{hence}\quad \mu=-\frac {D\cdot \ell}{K_S\cdot \ell}\in \frac 16 \mathbb Z. 
\]
\end{inparaenum}

Then we define the \emph{maximal multiplicity} $\lambda$ of $\L$ as follows. We have
\begin{equation}\label{eq:pur1}
K_X=\sigma^*(K_S)+\sum_{k=1}^n a_kE_k,
\end{equation}
where $E_1,\ldots,E_n$ are the distinct curves contracted by $\sigma$ and $a_1,\dots, a_n$ are positive integers. Moreover
\begin{equation}\label{eq:pur2}
\L_X:=\sigma'^*(\L')=\sigma^*(\L)-\sum_{k=1}^n b_kE_k,
\end{equation}
where $b_1,\dots, b_n$ are non--negative integers. We define
\[
\lambda:=\max \Big \{\frac {b_k}{a_k}\Big \}_{k=1,\ldots,n}.
\]
One has:\\
\begin{inparaenum}
\item[$\bullet$] $\lambda$ is attained at \emph{proper points},
which implies that if $\lambda=\frac {b_i}{a_i}$, for $1\leq i\leq n$,  we may assume that $a_i=1$ and therefore $\lambda$ can be also defined as the maximum multiplicity of a base point of $\L$. 

In fact, suppose we have a proper point $p_1$ and a sequence $p_2,\ldots, p_j$ of infinitely near points to $p_1$, each infinitely near to the other (see \S \ref {ssec:inp} below). Let $E_1,\ldots, E_j$ be the proper transform on $X$ of the exceptional divisors of the blow--ups at $p_1,\ldots, p_j$. Then the total transforms of the exceptional divisors are of the form
\[
E_1+\ldots+r_1E_j, \quad E_2+\ldots+r_2E_j, \quad \ldots, \quad E_j, \quad \text{with}\quad r_1,r_2, \ldots, \quad \text {positive integers}.
\]
Hence $a_1=1$ and 
\[
a_j=r_1+r_2+\cdots+r_{j-1}+1.
\]
If $\L$ has multiplicities $m_1,\ldots, m_j$ at $p_1,\ldots, p_j$, one has $m_1\geq m_2\geq \cdots\geq m_j$,  $b_1=m_1$ and
\[
b_j=r_1m_1+r_2m_2+\cdots+ r_{j-1}m_{j-1} +m_j.
\]
Hence $\frac {b_1}{a_1}=m_1$, whereas
\[
\frac {b_j}{a_j}=\frac {r_1m_1+r_2m_2+\cdots+ r_{j-1}m_{j-1} +m_j}{r_1+r_2+\cdots+r_{j-1}+1}\leq \frac {r_1m_1+r_2m_1+\cdots+ r_{j-1}m_1 +m_1}{r_1+r_2+\cdots+r_{j-1}+1}=m_1.
\]
\item[$\bullet$] in particular $\lambda=0$ if and only if $\L$ is base point free;\\
\item[$\bullet$] if $\lambda >0$, one has
\begin{equation*}\label{eq:purr}
K_X+\frac 1\lambda \L_X=\sigma^*(K_S+\frac 1\lambda \L)+\sum_{k=1}^n \big ( a_k -\frac 1\lambda b_k\big ) E_k
\end{equation*}
and 
\[
a_k -\frac 1\lambda b_k\geq 0, \quad \text {and the equality holds for some}\quad k\in \{1,\ldots, n\}. 
\]
\end{inparaenum}

Finally we introduce the number $\ell$ of \emph{crepant exceptional divisors}. When $\lambda=0$ then $\ell$ is undetermined. If $\lambda>0$ then we set
\[
\ell:= \quad \text{the order of the set}\quad \Big\{ k: \lambda =\frac {b_k}{a_k}\Big \}.
\]

The triple $(\mu,\lambda, \ell)$ is, by definition, the \emph{Sarkisov's degree} of $\phi$.

\subsection{The Noether--Fano--Iskovkikh Theorem}\label{ssec:NFI}

We keep the above notation and prove the:

\begin{thm}[Noether--Fano--Iskovkikh's Theorem]\label{thm:NFI} Suppose that $\lambda\leq \mu$ and that $K_S+\frac 1\mu\L$ is nef. Then $\phi$ is an isomorphism. 
\end{thm}

\begin{proof} The proof will be divided in various steps. 

\subsection{Step 1: one has  $\mu=\mu'$}

Recall that, as in Lemma \ref {lem:pac}, we have $H'+\mu'K_{S'}=p'^ *(A')$, with $H'$ and $A'$ very ample. Hence, for all $\epsilon >0$
\[
(1+\epsilon)H'+\mu' K_{S'}=\epsilon H'+H'+\mu' K_{S'}=\epsilon H'+p'^ *(A')
\]
is  ample. Therefore
\[
\frac 1t H'+K_{S'} \quad \text{is  ample for}\quad 0<t<\mu'.
\]
Hence its transformed system via $\phi$ intersects positively any movable curve, in particular the curves $F$ contracted to points by $p: S\to C$. 

One has
\[
\begin{split}
&\sigma'^*(\frac 1t H'+K_{S'})=\frac 1t \sigma'^*(H')+(K_X-E)=\\
& =\frac 1t \sigma'^*(H')+(\sigma^*(K_S)+\sum_{k=1}^na_kE_k-E)
\end{split}
\]
where $E$ is an effective divisor. Hence
\[
\sigma_*(\sigma'^*(\frac 1t H'+K_{S'}))=\frac 1t H+ K_S- \sigma_*(E)\quad \text{where}\quad H:=\sigma_*(\sigma'^*(H')).
\]
So, if $F$ is a (movable, hence nef)  curve contracted to a point by $p: S\to C$, and if $ 0<t<\mu'$, we have
\[
(\frac 1t H+K_S)\cdot F=\sigma_*(\sigma'^*(\frac 1t H'+K_{S'}))\cdot F+ \sigma_*(E)\cdot F\geq \sigma_*(\sigma'^*(\frac 1t H'+K_{S'}))\cdot F>0
\]
hence $\mu\geq \mu'$. 

It is worth noticing that so far we did not use that $K_S+\frac 1\mu H$ is nef. This will enter soon into play. In fact we have
\[
(\frac 1t H'+K_{S'})\cdot F'<0 \quad \text {for}\quad t>\mu'
\]
where $F'$ is a curve contracted to a point by $p': S'\to C'$. Indeed, for $t>\mu'$ we have
\[
F'\cdot (H'+tK_{S'})=F'\cdot (p'^*(A')+(t-\mu')K_{S'})=(t-\mu') F'\cdot K_{S'}<0.
\]
Then transforming by $\phi$ we get that $\frac 1t H+K_S-\sigma_*(E)$ is not nef for $t>\mu'$. Actually, if $Z$ is the transform of $F'$, we have $F'\cdot \sigma'_*(E)=0$, hence $Z\cdot \sigma_*(E)=0$, thus 
\[
(\frac 1t H+K_S)\cdot Z=(\frac 1t H+K_S-\sigma_*(E))\cdot Z<0 \quad \text {for}\quad t>\mu'. 
\]
Since $K_S+\frac 1\mu H$ is nef, we must have $\mu\leq \mu'$, thus $\mu=\mu'$ as wanted.

\subsection{Step 2: invariance of the adjoints} By \eqref {eq:pur1} and \eqref {eq:pur2}, and since $\lambda\leq \mu$, we have
\[
K_X+\frac 1\mu \L_X=\sigma^*(K_S+\frac 1\mu \L)+\sum_{k=1}^n r_kE_k, \quad \text{with} \quad r_k=a_k-\frac 1\mu b_k\geq a_k-\frac 1\lambda b_k\geq 0
\]
and 
\[
\sum_{k=1}^n r_kE_k\leq R:=\sum_{k=1}^n a_kE_k
\]
where $R$ is the ramification divisor of $\sigma: X\to S$. 

On the other hand
\[
K_X+\frac 1\mu \L_X=\sigma'^*(K_{S'}+\frac 1\mu \L')+R', \quad \text {with} \quad R'=\sum_{h=1}^m r'_hE'_h
\]
where $R'$ is the ramification divisor of $\sigma'$, with $E'_1,\ldots, E'_m$ the irreducible curves contracted to points by $\sigma'$ and $r'_1,\ldots, r'_m$ are positive integers. 

The aim of this step is to prove that
\begin{equation}\label{eq:invar}
\sum_{k=1}^n r_kE_k=\sum_{h=1}^m r'_hE'_h
\end{equation} 
which implies that
\begin{equation}\label{eq:invad}
\sigma^*(K_S+\frac 1\mu \L)=\sigma'^*(K_{S'}+\frac 1\mu\L')
\end{equation}
an equality which goes under the name of \emph{invariance of the adjoints}.  

To prove \eqref {eq:invar}, we introduce the notation $F_\ell$ to denote either the curves $E_1,\ldots, E_n$ or the curves $E'_1,\ldots, E'_m$, with the convention that the index $\ell$ varies in the following sets:\\
\begin{inparaenum}
\item [$\bullet$] $\ell\in S_\sigma$ if $F_\ell$ is contracted  to a point by $\sigma$ but not by $\sigma'$;\\
\item [$\bullet$] $\ell\in S_{\sigma'}$ if $F_\ell$ is contracted  to a point by $\sigma'$ but not by $\sigma$;\\
\item [$\bullet$] $\ell\in S_{\sigma,\sigma'}$ if $F_\ell$ is contracted  to a point by both $\sigma$ and  $\sigma'$.
\end{inparaenum}

Then
\[
\begin{split}
\sum_{k=1}^nr_kE_k=\sum_{\ell\in S_\sigma}r_\ell F_\ell+ \sum_{\ell\in S_{\sigma,\sigma'}}r_\ell F_\ell,\\
\sum_{h=1}^m r'_hE'_h=\sum_{\ell\in S_{\sigma'}}r'_\ell F_\ell+ \sum_{\ell\in S_{\sigma,\sigma'}}r'_\ell F_\ell.
\end{split}
\]

We have
\[
\sum_{k=1}^nr_kE_k+\sigma^*(K_S+\frac 1\mu\L)=\sigma'^*(K_{S'}+\frac 1\mu\L')+\sum_{h=1}^m r'_hE'_h,
\]
hence
\begin{equation}\label{eq:cc1}
\sum_{\ell\in S_\sigma}r_\ell F_\ell+\sum_{\ell\in S_{\sigma,\sigma'}}(r_\ell-r'_\ell) F_\ell=\sigma'^*(K_{S'}+\frac 1\mu\L')-\sigma^*(K_S+\frac 1\mu\L)+\sum_{\ell\in S_{\sigma'}}r'_\ell F_\ell
\end{equation}
and similarly
\begin{equation}\label{eq:cc2}
\sum_{\ell\in S_{\sigma'}}r'_\ell F_\ell+\sum_{\ell\in S_{\sigma,\sigma'}}(r'_\ell-r_\ell) F_\ell=
\sigma^*(K_S+\frac 1\mu\L)-\sigma'^*(K_{S'}+\frac 1\mu\L')+\sum_{\ell\in S_{\sigma}}r_\ell F_\ell.
\end{equation}

At this point we need the following:

\begin{lemma}[Negativity Lemma]\label{lem:neg} Let $f: V\to T$ be a birational morphism between two surfaces $V$ and $T$. Suppose $E=\sum_{i=1}^na_iE_i$ is an effective divisor contracted to points by $f$. Assume that for all $i\in \{1,\ldots, n\}$, one has $E\cdot E_i\geq 0$. Then $a_i\leq 0$,  for all $i\in \{1,\ldots, n\}$.
\end{lemma}

\begin{proof} Suppose that there is a $j\in \{1,\ldots, n\}$ such that $a_j>0$. We have
\[
\begin{split}
&0\leq E\cdot (\sum_{a_i>0}a_iE_i)=(\sum_{a_i\leq 0}a_iE_i+\sum_{a_i>0}a_iE_i)\cdot (\sum_{a_i>0}a_iE_i)=\\
&=(\sum_{a_i\leq 0}a_iE_i)\cdot (\sum_{a_i>0}a_iE_i))+(\sum_{a_i>0}a_iE_i)^2<0
\end{split}
\]
because $\sum_{a_i>0}a_iE_i$ is contracted to points, so $(\sum_{a_i>0}a_iE_i)^2<0$ and $(\sum_{a_i\leq 0}a_iE_i)\cdot (\sum_{a_i>0}a_iE_i))\leq 0$. Hence we have a contradiction, which proves the assertion.
\end{proof}

To apply this lemma, intersect both sides of \eqref {eq:cc1} with $F_s$, with $s\in S_\sigma\cup S_{\sigma, \sigma'}$. We get
\[
(\sum_{\ell\in S_\sigma}r_\ell F_\ell+\sum_{\ell\in S_{\sigma,\sigma'}}(r_\ell-r'_\ell) F_\ell)\cdot F_s=\sigma'^*(K_{S'}+\frac 1\mu \L')\cdot F_s+\sigma^*(K_{S}+\frac 1\mu\L)\cdot F_s+(\sum_{\ell\in S_{\sigma'}}r'_\ell F_\ell)\cdot F_s.
\]
The first summand on the right hand side is non--negative, because $K_S+\frac 1\mu\L'=\frac 1\mu p'^*(A')$ is nef. The second summand is clearly 0. The third summand is  non--negative. Hence
\[
(\sum_{\ell\in S_\sigma}r_\ell F_\ell+\sum_{\ell\in S_{\sigma,\sigma'}}(r_\ell-r'_\ell) F_\ell)\cdot F_s\geq 0, \quad \text{for all}\quad s\in S_\sigma\cup S_{\sigma, \sigma'}.
\]
By the Negativity Lemma  we deduce that 
\[
\begin{split}
&r_\ell\leq 0 \quad \text {if}\quad \ell \in S_\sigma,\\
& r_\ell\leq r'_\ell  \quad \text {if}\quad \ell \in S_{\sigma, \sigma'}.
\end{split}
\]
Arguing in the same way on \eqref {eq:cc2}, we get
\[
\begin{split}
& r'_\ell\leq r_\ell  \quad \text {if}\quad \ell \in S_{\sigma, \sigma'},\\
&r'_\ell\leq 0 \quad \text {if}\quad \ell \in S_{\sigma'}.
\end{split}
\]
In conclusion
\[
\begin{split}
&r_\ell= 0 \quad \text {if}\quad \ell \in S_\sigma,\\
&r'_\ell= 0 \quad \text {if}\quad \ell \in S_{\sigma'},\\
& r_\ell= r'_\ell  \quad \text {if}\quad \ell \in S_{\sigma, \sigma'}.
\end{split}
\]
which proves \eqref {eq:invar}. 

\subsection{Step 3: conclusion} We have 
\[
R\geq \sum_{k=1}^nr_kE_k= \sum_{h=1}^mr'_hE'_h=R'
\]
and all the divisors contracted to points by $\sigma'$ are also contracted to points by $\sigma$. By the minimality assumption on $\sigma'$, we have that $\sigma'$ is the identity, hence $X=S'$ and $\sigma=\phi^{-1}$. 

Now consider the composite map $S'\stackrel {\sigma}\to S\stackrel {p}\to C$ and let $\Gamma$ be a general fibre of it. 
Note that, since the general fibre of $p:S\to C$ is connected, then the same happens for $\Gamma$. 
By taking into account the definition of the quasi effective threshold,  and by \eqref {eq:invad} proved in Step 2, we have
\[
0= \Gamma\cdot \sigma^*(K_S+\frac 1\mu \L)=\Gamma\cdot (K_{S'}+\frac 1\mu\L')=\frac 1\mu \Gamma \cdot p'^*(A')\]
This tells us that $\Gamma$ is contracted to a point by $p'$. This proves that there is a morphism $g: C\to C'$ which makes the following diagram commutative
\[
\xymatrix{ 
&S \ar[d]_{p}&S'\ar[l]_{\phi^ {-1}}\ar[d]^{p'} \\
&C\ar[r]^{g}& C'
}
\]
Now note that $\phi^ {-1}$ is a birational morphism. If it is not an isomorphism, there is some curve $E$ in $S'$ contracted to a point by $\phi^ {-1}$. But then $E$ is contracted to a point by $p'$. Since the curves in fibres of $p'$ are all numerically equivalent, we clearly get a contradiction. This proves that $\phi^ {-1}$, and hence $\phi$ is an isomorphism, as wanted. \end{proof}

\subsection{Sarkisov's algorithm}\label{ssec:sarkalg}

In this section we prove the Noether--Castelnuovo's theorem. The proof consists in performing an algorithm which lowers the Sarkisov's degree, which is supposed to be lexicographically ordered. This algorithm is called \emph{Sarkisov's algorithm} or also the \emph{untwisting process}. 

We start with a diagram like \eqref {eq:star}, where the vertical arrows are Mori fibre spaces and $\phi$ is a birational map,  with Sarkisov degree $(\mu, \lambda, \ell)$. We have the dichotomy:\\
\begin{inparaenum}
\item [Case 1:] $\lambda \leq \mu$,\\
\item [Case 2:] $\lambda >\mu$.
\end{inparaenum}

In case 1 we have the further dichotomy:\\
\begin{inparaenum}
\item [Case 1.1:] $K_S+\frac 1\mu\L$ nef,\\
\item [Case 1.2:] $K_S+\frac 1\mu\L$ non--nef.
\end{inparaenum}

In case 1.1, by Noether--Fano--Iskovskih Theorem we conclude that $\phi$ is an isomorphism, and the algorithm ends here.

In case 1.2, $p: S\to C$ does not coincide with $\PP^2\to p$, where $p$ is a point, because on $\PP^2$ we would clearly have $K_{\PP^2}+\frac 1\mu\L\equiv 0$, hence $K_{\PP^2}+\frac 1\mu\L$ would be nef, a contradiction. Hence $p: S\to C$ is a $\PP^1$--bundle over the curve $C$ and $\Num(S)\cong \mathbb Z^ 2$ generated by the classes $H$ of a section of $p: S\to C$, existing by the Noether--Enriques' theorem, and $F$ of a fibre of $p: S\to C$. Then 
\[
(K_S+\frac 1\mu \L)\cdot F=0.
\]
Since $K_S+\frac 1\mu \L$ is not nef, then 
there is a class $D$ of a curve generating an extremal ray, such that 
\begin{equation}\label{eq:opl}
(K_S+\frac 1\mu \L)\cdot D<0, \quad \text{hence} \quad K_S\cdot D<0.
\end{equation}
So there is another extremal contraction, the one contracting the curves in $R$, which we denote by $q: S\to T$. Since $S\neq \PP^2$, then either $q: S\to T$ is the contraction of a $(-1)$--curve or it is another $\PP^1$--bundle. \medskip

Subcase 1.2.1: $q: S\to T$ is the contraction of a $(-1)$--curve. In this case $T$ is a surface and $\Num(T)\cong \mathbb Z$. Consequently $T\cong \PP^2$ and 
\[
\xymatrix{ 
&S=\mathbb F_1 \ar[d]_{p}\ar[r]^q&T=\PP^2:=S_1\ar[d]^{p_1} \\
& C= \PP^1\ar[r]&  x:=C_1
}
\]
is a link of type III. \medskip

Subcase 1.2.2: $q: S\to T$ is a $\PP^1$ bundle different from $p: S\to C$. Then, since $T$ is dominated by the fibres of $p: S\to C$, then $T\cong \PP^1$. For the same reason $C\cong \PP^1$. Moreover $D$ is the class of a fibre of $q: S\to \PP^1$, hence $D^2=0$. 

\begin{claim}\label{cl:111} One has $S= \PP^1\times \PP^1$ and 
\[
\xymatrix{ 
&S=\PP^1\times \PP^1\ar[d]_{p}&=&\PP^1\times \PP^1=S:=S_1\ar[d]^{q:=p_1} \\
&C= \PP^1\ar[r]& x&\ar[l]  \PP^1:=C_1
}
\]
is a link of type IV. 
\end{claim}

\begin{proof} The final assertion follows once we prove that  $S= \PP^1\times \PP^1$. To prove this, assume, to the contrary, that $S=\mathbb F_a$, with $a> 0$. Then $S$ has a unique section $E$ of $p: S\to \PP^1$  such that $E^2=-a<0$. This is an extremal ray as well as $F$, hence, we should have that the ray spanned by $E$ is the same as the ray spanned by $D$, hence $E^2=D^2=0$, a contradiction. \end{proof}

\begin{claim}\label {cl:112} In both subcases 1.2.1 and 1.2.2, we have a new Mori fibre space $p_1: S_1\to C_1$ for which $\mu_1<\mu$.
\end{claim}

\begin{proof} Let $\L_1$ be the transformed linear system on $S_1$ and let $F_1$ be the generator of the extremal ray contracted by $p_1$. 

In subcase 1.2.1, $q:S= \mathbb F_1\to \PP^2=S_1$ is the contraction of a $(-1)$--curve $E$. Let $L$ be the class of a line in $\PP^2$, then the class of a fibre of $p:S\to \PP^1$ is
$F=q^*(L)-E$ and $K_S=q^*(-3L)+E$, whereas $\L$ is of the form $\L=q^*(aL)-bE$, with $a,b$ non--negative integers. We have
\[
K_S+\frac 1\mu\L=q^*\Big ((-3+\frac a \mu)L\Big )+(1-\frac b\mu)E \quad \text{hence}\quad 
(K_S+\frac 1\mu\L)\cdot E=\frac b\mu -1<0,
\]
the last inequality holding because of \eqref {eq:opl}. Then 
\[
0=(K_S+\frac 1\mu\L)\cdot F=\Big (q^*\Big ((-3+\frac a \mu)L\Big )+(1-\frac b\mu)E\Big)\cdot (q^*(L)-E)=-3+\frac a\mu+1-\frac b\mu>-3+\frac a\mu,
\]
hence
\[
(K_{S_1}+\frac 1\mu \L_1)\cdot F_1=(-3+\frac a \mu)L\cdot L=-3+\frac a \mu<0
\]
which implies that $\mu_1<\mu$.

In subcase 1.2.2, we have $S=S_1=\PP^1\times \PP^1$ and $F$ and $F_1$ are the two rulings.  By \eqref {eq:opl} we have
\[
(K_S+\frac 1\mu\L)\cdot F_1<0
\]
which again implies that $\mu_1<\mu$. \end{proof}

In conclusion, in case 1, either the algorithm ends or, by making a Sarkisov's link, we drop the Sarkisov's degree. 

Next, we have to discuss case 2, in which $\lambda > \mu$. Here we have two subcases:\\
\begin{inparaenum}
\item [Case 2.1:] $S\cong \PP^2$;\\
\item [Case 2.2:] $p:S\to C$ is a $\PP^1$--bundle.
\end{inparaenum}

Let $x$ be a (proper) point of $S$ which realizes the maximal multiplicity $\lambda$ of $\L$. 

In case 2.1, blow--up $x$ with exceptional divisor $E$, so that we have a type I Sarkisov's link
\[
\xymatrix{ 
&S=\PP^2\ar[d] \ar[r]^{f}&\mathbb F_1=S_1 \ar[d]_{p_1}\\
&y &\PP^1\ar[l]
}
\]
Then, if $L$ is the class of a line in $\PP^2$, we have $\L=aL$, with $a$ a positive integer, and, with the usual notation, we have 
\[
\begin{split}
& K_{S_1}=f^*(K_S)+E=f^*(-3L)+E.\\
&\L_1=f^*(\L)-bE=f^*(aL)-bE, \quad \text{with}\quad b=\lambda>0,\\
&F_1=f^*(L)-E.
\end{split}
\]
Hence
\[
K_{S_1}+\frac 1\mu \L_1=f^*(-3L)+E+\frac 1\mu\Big (f^*(aL)-bE\Big)=(-3+\frac a\mu)f^*(L)+(1-\frac b\mu)E
\]
and therefore
\[
(K_{S_1}+\frac 1\mu \L_1)\cdot F_1=\Big ((-3+\frac a\mu)f^*(L)+(1-\frac b\mu)E\Big)\cdot (f^*(L)-E)=-3+\frac a\mu-\frac b\mu+1<0
\]
because
\[
0=(K_S+\frac 1\mu\L)\cdot L=(-3+\frac a\mu)L\cdot L=-3+\frac a\mu
\]
and $b=\lambda>\mu$. So, in this case too, we have $\mu_1<\mu$. 

In case 2.2, we perform an elementary transformation at $x$
\[
\xymatrix{ 
&S \ar[d]_{p} \ar@{-->}[r]^f&S_1\ar[d]^{p_1}\\
&C& C_1
}
\]
Note the commutative diagram
\[
\xymatrix{ 
&Z\ar_{\alpha}[d]\ar^{\beta}[dr] &\\
&S\ar@{-->}[r]^f &  S_1&
}
\]
If $F$ denotes the fibre of $p$ such that $x\in F$, the total transform on $Z$ of $F$ is of the form $F'+E$, with $F'$ the strict transform of $F$ and $E$ the exceptional divisor. Then on $S_1$, $F'$ is contracted to a point $x_1$ and the image $F_1$ of $E$ is the fibre of $p_1$ through $x_1$. 

Note that
\[
0=(K_S+\frac 1\mu \L)\cdot F=-2+\frac 1\mu \L\cdot F, \quad \text{thus}\quad \L\cdot F=2\mu.
\] 
Hence $x_1$ is a point of multiplicity $2\mu-\lambda<\mu<\lambda$. So $\lambda_1\leq \lambda$ and if equality holds, we decreased $\ell$. On the other hand we claim that $\mu_1=\mu$. Indeed, if we set $G=\alpha^*(F)=\beta^*(F_1)$,  with obvious meaning of the notation, we have
\[
\begin{split}
&K_Z=\alpha^*(K_S)+E,\\
&\L_Z=\alpha^*(\L)-\lambda E,\\
&K_Z+\frac 1\mu \L_Z=\alpha^*(K_S+\frac 1\mu \L)+(1-\frac \lambda \mu)E,\\
\end{split}
\]
hence
\[
(K_Z+\frac 1\mu \L_Z)\cdot G= \alpha^*(K_S+\frac 1\mu \L)\cdot \alpha^*(F)=(K_S+\frac 1\mu \L)\cdot F=0
\]
and, by the same computation
\[
0=(K_Z+\frac 1\mu \L_Z)\cdot G=\beta^*(K_{S_1}+\frac 1\mu \L_1)\cdot \beta^*(F_1)=
(K_{S_1}+\frac 1\mu \L_1)\cdot F_1
\]
which proves that $\mu_1=\mu$, as wanted. 

So, also in this case we lower the Sarkisov's degree. The final observation is that the Sarkisov's algorithm ends, because $\lambda$ and $\ell$ are non--negative integers and $\mu\in \frac 16\mathbb Z$.

In conclusion Sarkisov's algorithm proves the Noether--Castelnuovo's theorem, as wanted.

\begin{example}\label{ex:quad} 
Consider the \emph{standard quadratic transformation}
\[\varphi: [x,y,z]\in \PP^2\dasharrow  [yz,xz,xy]\in \PP^2.\]
We want to apply to $\varphi$ the Sarkisov's algorithm. 

Let $L'$ be the class of a line in the target $\PP^2$ and $L$ the class of a line in the source $\PP^2$. The transformed system of $|L'|$ is the linear system $\L$ of conics through the fundamental points $A=[1,0,0], B=[0,1,0], C=[0,0,1]$. Clearly we have $\lambda =1$ and $\ell=3$, and $\mu$ is computed by
\[
0=(K_S+\frac 1\mu \L)\cdot L=(-3L+\frac 1\mu \cdot 2L)\cdot L=-3+\frac 2\mu, \quad \text{hence}\quad \mu=\frac 23.
\]
So we have $\lambda>\mu$.\medskip

\noindent {\bf Step 1: blow--up $A$}. This is a type I link
\[
\xymatrix{ 
&S=\PP^2 \ar[d]_{p} &\mathbb F_1=S_1\ar[l] \ar[d]^{p_1}\\
&x& \PP^1
}
\]
If $L_1$ is the class of the pull back of a line of $\PP^2$, and $F_1$ a fibre of $p_1$, we have that $\L_1$ consists of the curves in $|F_1+L_1|$ passing through the points $B_1$ and $C_1$ the images of $B$ and $C$ on $S_1$, so $\L_1\cdot F_1=1$ hence
\[
(K_{S_1}+2\L_1)\cdot F_1=0,
\] 
thus $\mu_1=\frac 12<\frac 23=\mu$, so we lowered the Sarkisov's degree. Note that $\lambda_1=1$ and $\ell_1=2$. \medskip 

\noindent {\bf Step 2: make an elementary transformation at $B_1$}. This is a type II link
\[
\xymatrix{ 
&S_1=\mathbb F_1 \ar[d]_{p_1} \ar@{-->}[r]&\mathbb F_0=S_2 \ar[d]^{p_2}\\
&\PP^1& \PP^1
}
\]
Let $F_2$ be the fibre of $p_2$ and $G$ the other ruling of $\mathbb F_0$. The transform $\L_2$ on $S_2$ of $\L_1$ is now the linear system of the curves in $|F_2+G|$ passing through the point $C_2$ image of $C_1$ on $S_2$. Hence we see that the new Sarkisov's degree is $(\frac 12,1,1)$, which has been lowered again. \medskip

\noindent {\bf Step 3: make an elementary transformation at $C_2$}. This is a type II link
\[
\xymatrix{ 
&S_2=\mathbb F_0 \ar[d]_{p_2} \ar@{-->}[r]&\mathbb F_1=S_3 \ar[d]^{p_3}\\
&\PP^1& \PP^1
}
\]
If $L_3$ is the pull--back of a line in $\PP^2$ to $\mathbb F_1$, we have that $\L_3=|L_3|$ and the Sarkisov's degree now becomes $(\frac 12, 0, *)$, where $*$ stays for undeterminate.\medskip 

\noindent {\bf Step 4: blow down the $(-1)$--curve in $S_3=\mathbb F_1$}. This is a type III link
\[
\xymatrix{ 
&S_3=\mathbb F_1 \ar[d]_{p_3} \ar@{-->}[r]&\PP^2=S_4 \ar[d]^{p_4}\\
&\PP^1& x
}
\]
and the linear system $\L_4$ is the linear system of the lines in $\PP^2$. Then the Sarkisov's degree is now $(\frac 13, 0, *)$. 
\end{example}

\section{The classical Noether--Castelnuovo's theorem}\label{sec:class}

The classical Noether--Castelnuovo's theorem is as follows:

\begin{thm}[Classical Noether--Castelnuovo's theorem]\label{thm:classNC}
The group of birational transformations of $\PP^2$ is generated by the projective transformations and by the standard quadratic transformation.
\end{thm}

This theorem can be proved as a consequence of Noether--Castelnuovo's theorem \ref {thm:NC}. The idea is that elementary transformations appearing in the sequence of Sarkisov's links composing a given birational transformation of $\PP^2$ can be rearranged in such a way as to give quadratic transformations as in Example \ref {ex:quad}. However we will not do this, but we will give a direct proof of Theorem \ref {thm:classNC} which goes back to Castelnuovo and Alexander. This proof requires a number of preliminaries which we will now introduce.

\subsection{Infinitely near points}\label{ssec:inp}

Let $\sigma: S'\to S$ a birational morphism of smooth, irreducible, projective surfaces. Then $\sigma$ is the composition of finitely many blow--ups
\[
\sigma: S'=S_n\stackrel{\sigma_n}\to S_{n-1} \stackrel{\sigma_{n-1}}\to\ldots \stackrel{\sigma_2}\to S_1\stackrel{\sigma_1}\to S_0=S
\] 
at points $p_i\in S_i$, for $0\leq i\leq n-1$.
A point $p\in S$ will be called a \emph{proper point}. The point $q\in S_i$ is called an \emph{infinitely near point to $p$ of order $i$}, and we write $q>^ip$, if there are points $p_j\in S_j$ such that $\sigma (p_j)=p_{j-1}$ for $0\leqslant j\leq n$ and $p_0=p$ and $p_j=q$.

For $1\leq i\leq n$, let $E_i=\sigma_i^ {-1}(p_{i-1})$ be the exceptional divisor of the blow-up $\sigma_i: S_i\to S_{i-1}$. We denote by $E'_{i,j}$ the proper transform of $E_i$ on $S_j$, for $1\leq i\leq j\leq n$. 

We say that $q$ is \emph{proximate} to $p$, and we write $q\to p$, if either $q>^1p$ or $q>^ip$, with $i>1$ and $q\in E_{1,j}$. In this later case we say that $q$ is \emph{satellite} to $p$ and we write $q\odot p$. 

If $C$ is a curve on $S$, we say that $C$ \emph{passes through} the infinitely near point $q>^ip$ if the strict transform of $C$ on $S_i$ contains $q$. If $C$ passes with multiplicity $m$ through $p$ and multiplicity $m_i$ through a point $p_i\to p$, one has the obvious
 \emph{proximity inequality}
 \[
 m\geq \sum_{p_i\to p} m_j.
 \]
 
 Moreover, if a curve $C$ on $S$ passes through $p$ and through a satellite point $q\odot p$, then $C$ is singular at $p$. For details, see \cite {Cal}. 
  
 \subsection{Homalodail nets}\label {ssec:hom}
 
 Let $p_1,\ldots, p_n$ be points in the plane, which can be proper or infinitely near. Fix positive integers $m_1,\ldots, m_n$, and consider the linear system 
 \[
 \L:=\L(d;p_1^{m_1},\ldots, p_n^{m_r})
 \] 
of plane curves of degree $d$ having multiplicity at least $m_i$ at $p_i$, for $1\leq i\leq r$.
If the points are understood, we may simply write
\[
\L=\L(d;{m_1},\ldots, {m_r})
\] 
and if some multiplities are repeated we may sometimes use the exponential notation, i.e. $m_i^j$ stays for $j$ times $m_i$. 
Since the \emph{base points} can be infinitely near, we can understand the linear system $\L$ as living on a suitable birational model $\sigma: S\to \PP^2$. 

Assume now $\dim(\L)=2$, $\L$ with  no fixed components, and the map
\[
\phi_\L: \PP^2\dasharrow \PP^2
\]
determined by $\L$ to be birational, i.e., a \emph{Cremona transformation}. Then we call $\L$ a \emph{homaloidal net} and $\L$ is the transform on the source $\PP^2$ of the linear system of lines on the target $\PP^2$. For a homaloidal net we have
\[
\begin{split}
& d^2-\sum_{i=1}^r m_i^2=1\\
& {{d-1}\choose 2} - \sum_{i=1}^r {{m_i}\choose 2}=0
\end{split}
\]
where the first equality means that the curves in $\L$ meet in one variable point off the base points and the second means that the curves in $\L$ have geometric  genus 0. The above relations also read
\begin{equation}\label{eq:rel}
\begin{split}
& d^2-\sum_{i=1}^r m_i^2=1\\
& 3(d-1)=\sum_{i=1}^r m_i
\end{split}
\end{equation}

\begin{example}\label{ex:DJ} For any $d\geq 2$, the linear system
\[
\L=\L(d; d-1, 1^ {2d-2})
\]
is a homaloidal net called a \emph{De Jonqui\`eres net} and $\phi_\L$ is called a \emph{De Jonqui\`eres transformation}. Note that such a map sends the pencil of lines through the point of multiplicity $d-1$ to a pencil of lines. For $d=2$ this is a \emph{quadratic transformation}.
\end{example}

\subsection{The simplicity}\label{ssec:simpl}

Let $\L=\L(d;m_0,\ldots, m_r)$ be a homaloidal net, where we assume $m_0\geq \ldots \geq  m_r$. If there is no base point, we set $r=-1$. Moreover we set $m_{-1}=\infty$ and $m_s=0$ for all integer $s>r$. 

The \emph{simplicity} of $\L$ is the triplet $(k_\L,h_\L,s_\L)$ of integers defined as follows:
\[
\begin{split}
&k_\L:= d-m_0;\\ 
& m_{h_\L}>  \frac{k_\L}2 \geq m_{h_\L+1};\\
&s_\L= \quad \text{the number of satellite points among}\quad p_0,\ldots , p_{h_\L}.
\end{split}
\]

If $\phi=\phi_\L:\PP^2\dasharrow \PP^2$ is a Cremona transformation, the \emph{simplicity} of $\phi$ is, by definition, the simplicity of $\L$.

One says that $\L'$ [resp., the Cremona transformation $\phi'$] is \emph{simpler} than $\L$ [resp., than $\phi$]  if the simplicity of $\L'$ [resp., of $\phi'$] is lexicographically 
smaller than the simplicity of $\L$ [resp., of $\phi$]. 

\begin{example}\label{ex:DJ2} If $\L(d;d-1, 1^ {2d-2})$ is a De Jonqui\`eres net, then its simplicity is $(1, 2d-2, s)$, where $s$ is the number of satellite points among the points of multiplicity 1 proximate to the point of multiplicity $d-1$. By the proximity inequality, one has $s\leq d-1$.  \end{example}

\subsection{The proof of the classical Noether--Castelnuovo's theorem} \label{ssec:NCT}

We can now give the:

\begin{proof}[Proof of Theorem \ref {thm:classNC}]
 Let $\phi=\phi_\L:\PP^2 \dasharrow\PP^2$ be a Cremona transformation. Assume it is neither a linear map nor a quadratic transformation. We will prove that there is a quadratic transformation $\gamma: \PP^2\dasharrow \PP^2$ such that $\gamma\circ \phi$ is simpler that $\phi$. This will prove the theorem.
 
 Let 
 \[
 \L=\L(\delta; p_0^{\alpha_0},\ldots, p_r^{\alpha_r}) \quad \text{with}\quad \alpha_0\geq \ldots \geq \alpha_r. 
 \]
 By \eqref {eq:rel} we have
 \[
 \begin{split}
 &0\geq \sum_{i=0}^r \alpha_i(\alpha_i-\alpha _0)= \sum_{i=0}^r \alpha_i^2-\alpha_0\sum_{i=0}^r \alpha_i\\
 &=\delta^2-1-3\alpha_0(\delta-1)=(\delta-1)(\delta+1-3\alpha_0),
 \end{split}
 \]
 hence
 \[
 \begin{split}
 &\delta<3\alpha_0, \quad \text {i.e.,}\quad k_\L=\delta-\alpha_0<2\alpha_0,\quad \text {i.e.,} \\
& \alpha_0>\frac {k_\L}2, \quad \text{hence}\quad h_\L\geq 0.
 \end{split}
 \]
 Set $m:=\frac {k_\L}2$. Then we have
 \begin{equation}\label{eq:ut}
 \begin{split}
 &(\delta-1)(\delta-3m+1)=\delta(\delta-3m)+3m-1=\\
 &=\alpha_0(\alpha_0-m)+\sum_{i=1}^r\alpha_i(\alpha_i-m).
  \end{split}
 \end{equation}
 
 Indeed, by subtracting term by term
 \[
 (\delta-1)(\delta-2)=\sum_{i=0}^r\alpha_i(\alpha_i-1)
 \]
 from
 \begin{equation}\label{eq:pal1}
 \delta^2-1=\sum_{i=0}^r\alpha_i^2
 \end{equation}
 we obtain
 \begin{equation}\label{eq:pal}
 3(\delta-1)=\sum_{i=0}^r\alpha_i.
 \end{equation}
 Then subtracting term by term \eqref {eq:pal} multiplied by $m$ form \eqref {eq:pal1}, we obtain \eqref {eq:ut}.

 By \eqref{eq:ut} and since $2m=\delta-\alpha_0$, we have
 \[
 2m(\alpha_0-m)=(\delta-\alpha_0)(\alpha_0-m)=\sum_{i=1}^r\alpha_i(\alpha_i-m)-3m+1.
 \]
 Since $m=\frac {k_\L}2\geq \frac 12$, we have
 \[
 \begin{split}
& 2m(\alpha_0-m)= (\delta-\alpha_0)(\alpha_0-m)<\sum_{i=1}^r\alpha_i(\alpha_i-m)\leq \\
&\leq \sum_{i=1}^{h_\L}\alpha_i(\alpha_i-m)\leq 2m \sum_{i=1}^{h_\L}(\alpha_i-m)
\end{split}
 \]
 where the last inequality follows from the fact that $\alpha_i\leq 2m=\delta-\alpha_0$, which is equivalent to $\delta\geq \alpha_0+\alpha _i$, which in turn holds because otherwise the line through $p_0$ and $p_i$ would split off $\L$, which is impossible, because the general curve of $\L$ is irreducible.  In conclusion we have
 \begin{equation}\label{eq:ppp}
 \alpha_0-m<\sum_{i=1}^{h_\L}(\alpha_i-m)
 \end{equation}
 and this implies that:\\
 \begin{inparaenum}
 \item [$\bullet$] $h_\L\geq 2$;\\
  \item [$\bullet$] $p_1,\ldots, p_{h_\L}$  cannot be all proximate to $p_0$, otherwise we would have
 \[
 \alpha_0\geq \sum_{i=1}^{h_\L}\alpha_i
 \]
 which contradicts \eqref {eq:ppp}. 
 \end{inparaenum}
 
 Note also that for $0<i<j\leq h_\L$ the points $p_0, p_i,p_j$ cannot be on (the strict transform) of a line, otherwise, if this happens, since
 \[
 \alpha_0+\alpha_i+\alpha_j>\alpha_0+2m=\alpha_0+\delta-\alpha_0=\delta,
 \]
the line would split off $\L$, a contradiction. \medskip

{\bf Case 1:} there are two points $p_i,p_j$ among $p_1,\ldots, p_{h_\L}$ such that there is an irreducible conic through $p_0, p_i,p_j$. This means that none of $p_i,p_j$ is satellite. Then we can consider the quadratic transformation \emph{based} at $p_0,p_1,p_0$, i.e., determined by the homalodal net $\L(2;p_0,p_i,p_j)$. Then $\psi:=\gamma\circ \phi$ is determined by the homaloidal net $\L'$ transformed of $\L$ via $\gamma$. This is a linear system of degree 
\[
\delta-\epsilon \quad \text {with}\quad \epsilon=\alpha_0+\alpha_i+\alpha_j-\delta=\alpha_i+\alpha_j-2m>0
\]  
since $\alpha_i, \alpha_j>m$. Then $\L'$ has the following multiplicities at the points $p'_0,p'_i,p'_j$, where the lines $p_ip_j, p_0p_j, p_0p_i$ are contracted by $\gamma$:
\[
\begin{split}
&\delta-\alpha_i-\alpha_j=\alpha_0-\epsilon<\alpha_0\\
&\delta-\alpha_0-\alpha_j=2m-\alpha_j\leq m\\
&\delta-\alpha_0-\alpha_i=2m-\alpha_i\leq m,
\end{split}
\]
whereas all the other multiplicities stay the same. 

If $p'_0$ is still the point with the highest multiplicity, then $k_{\L'}=k_{\L}$, but we see that $h_{\L'}=h_{\L}-2$, so the simplicity went down. If $p'_0$ is no longer the point with the highest multiplicity, then the highest multiplicity is $\mu>\alpha_0-\epsilon$, then
\[
k_{\L'}=(\delta-\epsilon)-\mu<(\delta-\epsilon)-(\alpha_0-\epsilon)=\delta-\alpha_0=k_{\L},
\]
and again the simplicity went down.\medskip

{\bf Case 2:} for any two points  $p_i,p_j$ among $p_1,\ldots, p_{h_\L}$  there is no  irreducible conic through $p_0, p_i,p_j$. This means that $p_i>p_0$, for all $1\leq i\leq h_\L$, and moreover we can find $1\leq i<j\leq h_\L$ such that 
\[
p_j>^1p_i>^1p_0 \quad \text{and} \quad p_j\odot p_0.
\]
Choose $q\in \PP^2$ general, and consider the quadratic transformation $\gamma$ based at $p_0, p_i, q$. Set $\psi=\gamma\circ \phi$. This is determined by the homalodail net $\L'$ which is transformed of $\L$ by $\gamma$. Then $\L'$ has degree 
\[
2\delta-(\alpha_0+\alpha_i)=\delta+(\delta-\alpha_0-\alpha_i)=\delta+(2m-\alpha_i)=\delta+\epsilon
\]
where $\epsilon=\delta-\alpha_0-\alpha_i=2m-\alpha_i$ and $0\leq \epsilon\leq m$, because, as we saw, $\alpha_i\leq 2m$ and $\alpha_i>m$. The multiplicities at the points $p'_0,p'_i,q'$, where the lines $p_iq, p_0q, p_0p_i$ are contracted by $\gamma$  are:
\[
\begin{split}
&\delta -\alpha_i=\alpha_0+\epsilon\geq \alpha_0\\
&\delta-\alpha_0=2m\\
&\delta-\alpha_0-\alpha_i=\epsilon< m,
\end{split}
\]
whereas all the other multiplicities stay the same. One has $\alpha_0+\epsilon\geq 2m$ which is equivalent to $\alpha_0\geq \alpha_i$. Hence $k_\L=k_{\L'}$, $m$ stays the same and also $h_\L=h_{\L'}$. However $p_j$ is no longer satellite, therefore $s_{\L'}=s_\L-1$, and the simplicity went down. 

To finish the proof, we have to show that all quadratic transformation are composed by projective transformations and by the standard quadratic transformations. To do this, we remark that the quadratic transformations, corresponding to a homaloidal net of the form $\L=\L_2(1^3)$, are of three types:\\
\begin{inparaenum}
\item [(Type I)] the three base points $p_0,p_1,p_2$ of $\L$ are  proper and  distinct;\\
\item [(Type II)] $p_0$ and $p_2$ are proper and distinct and $p_1>^1p_0$;\\
\item [(Type III)] $p_0$ is proper and $p_2>^1p_1>^1p_0$, and $p_2$ is not satellite of $p_0$.
\end{inparaenum}

The quadratic transformations of type I are clearly composed of a projective transformation and of the standard quadratic transformation. 

\begin{claim}\label{cl:quad} A quadratic transformation of type II [resp. of type III] is the product of two [resp. of four] quadratic transformations of type I.
\end{claim}

\begin{proof} [Proof of the Claim \ref {cl:quad}] Consider the quadratic transformation $\alpha$ of type II, with base points as $p_0,p_1,p_2$ with $p_1>^1p_0$. Consider the quadratic transformation $\gamma$ of type I with base points $p_0,p_2, q$ where $q$ is a general point of the plane. The homaloidal net determining $\alpha$ is then transformed by $\gamma$ in a homaloidal net $\beta$ of conics with three distinct proper points: the points in which the lines $qp_0$ and $qp_2$ are contracted by $\gamma$ and the point corresponding to $p_1$ which is now a proper point on the line which is contracted to $p_0$ by the inverse transformation of $\gamma$. This proves the claim in this case.

Similarly,  consider the quadratic transformation $\alpha$ of type III, with base points $p_0,p_1,p_2$ with $p_2>^1p_1>^1p_0$. Consider the quadratic transformation $\gamma$ of type II, with base points $p_0,p_1,q$, where $q$ is a general point in the plane. As above one sees that the homaloidal net determining $\alpha$ is  transformed by $\gamma$ in a homaloidal net $\beta$ of conics with two distinct proper points and one point infinitely near to one of the two (we leave the details to the reader). Since the claim holds for type II transformations, this prove that transformations of type III are the product of four transformations of type I. \end{proof}

The classical Noether--Castelnuovo's theorem is now completely proved.\end{proof}

\section{Examples}\label{sec:ex}
In this section we collect a number of interesting examples which illustrate various phenomena concerning the Kleiman--Mori cone.

\subsection {Negative curves} \label{ssec:neg}

\begin{lemma}\label{lem:neg} let $S$ be a smooth surface.\\
\begin{inparaenum}
\item [(i)] If $R=\mathbb R^+[\alpha]$ is an extremal ray of $\NE(S)$, then $\alpha^2\leq 0$;\\
\item [(ii)] if $C$ is an irreducible curve on $S$ such that $C^2<0$, then $\mathbb R^+[C]$ is an extremal ray of $\NE(S)$. 
\end{inparaenum}
\end{lemma} 
\begin{proof} (i) Assume, by contradiction, that $\alpha^2>0$. If $H$ is ample, we have $H\cdot \alpha>0$ by Kleiman's criterion of ampleness. So there is an open neighborhood $U$ of $\alpha$ in $\Num(S)\otimes \mathbb R$ such that for all $\beta\in U$, one has $\beta^2>0$ and $H\cdot \beta>0$. For all rational $\beta\in U$, we have a suitable multiple $k\beta$ of $\beta$ such that $k\beta=D$ is a divisor and for $m\gg 0$ we have
\[
h^ 0(mD)+h^ 0(K_S-mD)\geq \chi(mD)=\chi(\mathcal O_S)+\frac {(mD)\cdot (mD-K_S)}2>0.
\]
But $h^ 0(K_S-mD)=0$ because $H\cdot (K_S-mD)=H\cdot K_S-mkH\cdot \beta<0$   for $m\gg 0$, so $mD$ is effective which implies that $\alpha$ is in the interior of $\NE(S)$.\medskip

\noindent (ii) Write $[C]= \sum_{i=1}^n \beta_i$, where $\mathbb R^+[\beta_i]$ are extremal rays for $1\leq i\leq n$. We have
\[
0>C^ 2=C\cdot \sum_{i=1}^n \beta_i,
\]
hence there is some $j\in \{1,\ldots, n\}$ such that $C\cdot \beta_j<0$. Set $\beta=\beta_j$.
Then $\beta=\lim_{n\to \infty} \beta_n$, with $\beta_n$ in the interior of $\NE(S)$, and we have $C\cdot \beta_n<0$ for $n\gg 0$. Hence for $n\gg 0$ we have $\beta_n=\lambda_nC+\beta'_n$, with $\lambda _n>0$  and $C\cdot \beta'_n\geq 0$. Thus $\beta=\lambda C+\beta'$, with $\lambda=\lim_{n\to \infty} \lambda _n$, $\beta'=\lim_{n\to \infty} \beta' _n$, and $C\cdot \beta'\geq 0$. We claim that $\beta'=0$. If not, 
since $\mathbb R^+[\beta]$ is an extremal ray, we have that 
$\mathbb R^+[\beta]=\mathbb R^+[\beta']=\mathbb R^+[C]$, which is impossible, because 
$C\cdot \beta<0$ whereas $C\cdot \beta'\geq 0$. Hence $\beta'=0$ and 
$\mathbb R^+[\beta]=\mathbb R^+[C]$ as wanted. 
\end{proof}

\begin{example}\label{ex:rns} Consider the $\PP^1$--bundle $S=\mathbb F_n\to \PP^1$. One has $\Num(S)\cong \mathbb Z^2$, generated by the classes $f$ of a fibre $F$ and the class $e$ of the section $E$ that $E^2=-n$. This section is unique if $n\neq 0$ and in this case it generates an extremal ray. The other extremal ray is generated by $f$. In fact, for every irreducible curve $C\neq F$ on $S$, we have $C\sim \lambda F+\beta E$, and $\beta\geq 0$, otherwise $C\cdot F<0$, which is impossible. We claim that also $\lambda \geq 0$. Otherwise, $C\cdot E=\lambda-n\beta<0$, so we can write $C=\alpha E+ C'$, with $C'$ effective not containing $E$, hence $E\cdot C'\geq 0$. But then $C'\sim \lambda F+(\beta-\alpha) E$ and $\beta\geq \alpha$ because $C'$ is effective. So $E\cdot C'=\lambda-n(\beta-\alpha)<0$, a contradiction. This proves that $\NE(\mathbb F_n)=\mathbb R^+[f]+
\mathbb R^+[e]$ if $n\neq 0$.

If $n=0$, one sees that for all curves $C$ on $\mathbb F_0$, different from $E$ and $F$, one has $C^2>0$, which again implies that $f$ and $e$ are the two extremal rays. 
\end{example}

\begin{example}\label{ex:pp} Let $\pi: X\to \PP^2$ be the blow--up of the plane at three distinct points $p_1,p_2,p_3$  sitting on a line $\ell$. On $X$ we have the three $(-1)$--curves $E_1,E_2,E_3$ corresponding to  $p_1,p_2,p_3$ respectively. There is also the strict transform $C$ of the line $\ell$, which has $C^2=-2$. These four curves are extremal rays of $\NE(X)$. 

The anticanonical system $|-K_X|$ is the strict transform on $X$ of the $6$--dimensional linear system of cubics through $p_1,p_2,p_3$, hence it is nef with positive self--intersection equal to $6$,  and it is base point free. So all $\alpha\in \NE(X)$ have $\alpha\cdot K_X\leq 0$.  

The curve $C$ is the only irreducible curve such that $C\cdot K_X=0$, i.e., $\mathbb R^+[C]$ is the intersection of $\NE(X)$ with the hyperplane $K_X^\perp$. Indeed, suppose there is some other irreducible curve $D$ such that $D\cdot K_X=0$, thus $D^2\geq -2$.
Suppose the image of $D$ on $\PP^2$  has degree $d>0$ and multiplicities $m_1,m_2,m_3$ at $p_1,p_2,p_3$. Then $D\cdot K_X=0$
reads 
\begin{equation}\label{eq:kotz1}
3d-(m_1+m_2+m_3)=0.
\end{equation}
Moreover $(D+C)\cdot K_X=0$, then by the Hodge index theorem, one has 
\[
0>(D+C)^2=D^2+C^2+2C\cdot D\geq-4+2C\cdot D,
\]
which implies $0\leq C\cdot D\leq 1$, that in turn reads
\begin{equation}\label{eq:katz}
0\leq d-(m_1+m_2+m_3)\leq 1
\end{equation}
which is incompatible with \eqref {eq:kotz1}. 

Moreover there is no other $(-1)$--curve $E$ on $X$ besides $E_1,E_2,E_3$. In fact the image of $E$ on $\PP^2$ would have degree $d>0$ and multiplicities $m_1,m_2,m_3$ at 
$p_1,p_2,p_3$. Since $E\cdot K_X=-1$, we have 
\begin{equation}\label{eq:kotz}
3d-(m_1+m_2+m_3)=1.
\end{equation}
Moreover  $|-K_X|$ defines a morphism $\phi_{|-K_X|}$ which contracts $C$ to a double point of the image, and $E$ is mapped to a line. Therefore $0\leq C\cdot E \leq 1$, which reads again \eqref {eq:katz},
which is incompatible with \eqref {eq:kotz}. 

If $D$ is any irreducible curve on $X$, different from $E$ and $E_1,E_2,E_3$, then $D^2\geq 0$. In fact, one has $D\cdot K_X< 0$. If $D^2<0$, we would have  $D\cdot K_X=D^2=-1$, which, as we saw, is impossible. 

Suppose $D^2=0$ and
the image of $D$ on $\PP^2$ has degree $d>0$ and multiplicities $m_1,m_2,m_3$. Then $D\cdot K_X=-2$, i.e.,
\begin{equation}\label{eq:kotz2}
3d-(m_1+m_2+m_3)=2,
\end{equation}
and $|D|$ is a pencil of rational curves, which is mapped by  $\phi_{|-K_X|}$ to a pencil of conics. Then, as above, we  have   $0\leq C\cdot D \leq 1$, i.e., \eqref {eq:katz} holds. Then by \eqref{eq:kotz2}, we see that the only possibility is $d=1$, and $m_i=1$, $m_j=m_k=0$, with $\{i,j,k\}=\{1,2,3\}$. This implies that $D\sim C+E_j+E_k$, so that $D$ cannot span an extremal ray.

In conclusion  $\NE(X)$ is the polyhedral cone spanned by the extremal rays $\mathbb R^+[C]$, $\mathbb R^+[E_1]$, $\mathbb R^+[E_2]$, $\mathbb R^+[E_3]$.
\end{example}

\subsection{The blown--up plane}\label{ssec:BUP}

Consider $\pi: X_n\to \PP^2$ the blow--up of the plane at $n$ general points $p_1\ldots,p_n$, with $(-1)$--curves $E_1,\ldots, E_n$ corresponding to $p_1\ldots,p_n$ respectively. A linear system $\L$ on $X_n$ of the form
\[
|dL-\sum_{i=1}^nm_iE_i|
\]
where $L$ is the transform on $X_n$ of a general line of $\PP^2$, is mapped to the plane to a linear system of the type  $\L(d;m_1,\ldots,m_n)$.
The integer $d$ is called the \emph{degree} of $\L$. 

Note that the homaloidal nets with base points at $p_1\ldots,p_n$ determine 
Cremona transformations which form a group acting on $\Pic(X)$. We extend this group with the permutations over $p_1\ldots,p_n$ and call this group $K_n$.  

If we have $\L=\L(d;m_1,\ldots,m_n)$ we set
\[
\begin{split}
&\nu_\L=d^2-\sum_{i=1}^nm_i^2 \quad \text{called the self--intersection of}\quad \L;\\
&g_\L={{d-1}\choose 2}-\sum_{i=1}^n {{m_i}\choose 2}\quad \text{called the arithmetic genus of}\quad \L.
\end{split}
\]
we often write $\nu$ and $g$ instead of $\nu_\L$ and $g_\L$ if there is no danger of confusion. Note that $\nu$ and $g$, as well as $\dim(\L)$ are invariant for the $K_n$--action.

\begin{lemma}\label{lem:profl} Assume $n\geq 9$ and $\L$ with  $\nu\geq 2g-2$. The the $K_n$--orbit of $\L$ is infinite (hence $K_n$ is infinite), unless $n=9$, $d=3m$, $m_1=\cdots=m_9=m$.
\end{lemma}

\begin{proof} Assume $m_1\geq \cdots \geq m_n$. We claim that
\begin{equation}\label{eq:fot}
m_n+m_{n-1}+m_{n-2}<d.
\end{equation}
If so, the quadratic transformation based at $p_n,p_{n-1},p_{n-2}$ transforms $\L$ in a linear system with the same $\nu$ and $g$ and higher degree. By iterating this, we see the conclusion holds. 

To prove \eqref {eq:fot}, we remark that
\[
0\leq \nu-(2g-2)=3d-\sum_{i=1}^n m_i.
\]
If \eqref {eq:fot} does not hold, then
\[
m_n+m_{n-1}+m_{n-2}\geq d
\]
and, since $n\geq 9$, we have $\sum_{i=1}^n m_i\geq 3d$, hence
\[
0\leq \nu-(2g-2)=3d-\sum_{i=1}^n m_i\leq 0,
\]
so that
\[
\nu=2g-2, n=9, m=m_1\ldots, =m_9, d=3m,
\]
proving the assertion. 
\end{proof}

\begin{corollary}\label{cor:fot} If $n\geq 9$ there are on $X_n$ infinitely many $(-1)$--curves.
\end{corollary}
\begin{proof} For any $(-1)$--curve one has $\nu=-1>-2=2g-2$. The assertion follows by applying Lemma \ref {lem:profl}. 
\end{proof}

\begin{remark}\label{rem:otto} If $n\leq 8$ there are only finitely many $(-1)$--curves on $X_n$. Indeed, the anticanonical system is $|-K_{X_n}|=\L_3(1^8)$ and it has dimension $9-n\geq 1$. If $E$ is a $(-1)$--curve on $X_n$, one has $-K_n\cdot E=1$. If $n\leq 7$, then $\dim(|-K_{X_n}|)\geq 2$, so the curve $E$ is contained in a curve of $|-K_{X_n}|$. Hence the degree of the image of a $(-1)$--curve on $\PP^2$ is bounded by $3$ and we are done. If $n=8$, note that  $|-2K_{X_n}|=\L_6(2^8)$, which has dimension at least 3. Again any $(-1)$--curve is contained in a curve of $|-2K_{X_n}|$ and we conclude as before.
\end{remark}

Consider now $\L=\L(d;m_1,\ldots, m_n)$ with $m_1\geq \ldots \geq m_n$, and $m_3\geq 1$. 

\begin{lemma}\label{lem:ccrr}
Let $\L$ be a linear system as above. If
\[
(m_3-1)\nu\geq 2m_3(g-1),
\]
and either the strict inequality holds or $m_3<m_1$, then there is a Cremona transformation in $K_n$ which lowers the degree of $\L$. If the equality holds and $m_1=m_2=m_3=m$ the same is true, unless $d=3m$.
\end{lemma}

\begin{proof} We want to prove that
\begin{equation}\label{eq:lop}
d<m_1+m_2+m_3,
\end{equation}
so that the quadratic transformation based at $p_1,p_2,p_3$ lowers the degree of $\L$. 

Consider the two equalities
\[
\begin{split}
&d(d-1)-\sum_{i=1}^nm_i(m_i-1)-(2g-2)=0\\
&d^2-\sum_{i=1}^nm_i^2-\nu=0.
\end{split}
\]
by multiplying the first by $-m_3$, the second by $m_3-1$, and then adding, we find the identity
\[
d(3m_3-d)+m_1^2+m_2^2-m_3(m_1+m_2)-\Big [ (m_3-1)\nu -2m_3(g-1)   \Big ]=
m_3 \sum _{i=4}^n m_i-  \sum _{i=4}^n m^2_i
\]
where the right hand side is clearly non--negative. The term $d(3m_3-d)$ reaches its maximum for $d=\frac {3m_3}2$, then monotonically decreases for larger $d$. Hence if  
\eqref {eq:lop} does not hold, we have
\[
(m_1+m_2+m_3)(2m_3-m_1-m_2)+m_1^2+m_2^2-m_3(m_1+m_2)-\Big [ (m_3-1)\nu -2m_3(g-1)   \Big ]\geq 0
\]
hence
\[
2(m_3^2-m_1m_2)-\Big [ (m_3-1)\nu -2m_3(g-1)   \Big ]\geq 0
\]
whence we get a contradiction. \end{proof}

\begin{corollary}\label{cor:prof} The action of $K_n$ on $(-1)$--curves is transitive.
\end{corollary}
\begin{proof} If $E$ is a curve of $X_n$ different from one of the curves $E_i$, for $1\leq i\leq n$, and also different from one of the curves $E_{ij}=\L_1(p_i^1,p_j^1)$,  for 
$1\leq i<j\leq n$, then $E=\L(d;m_1,\dots,m_n)$ with $m_1\geq \ldots \geq m_n$, and $m_3\geq 1$. Then we can apply Lemma \ref {lem:ccrr} to $E$ and lower its degree $d$. Apply this repeatedly till we arrive at degree 0 or 1. In the degree 0 case, we have one of the curves $E_i$, for $1\leq i\leq n$. In the degree 1 case $E$ is one of the curves $E_{ij}$ as above, whose degree can be lowered to 0 by the quadratic transformation based at $p_i,p_j, p_k$, with $k\neq i,j$. The assertion follows.\end{proof}

\subsection{Products of elliptic curves}\label{ssec:prod}

Let $E$ be a general curve of genus 1. Set $X=E\times E$. Let $f_1,f_2,\delta$ be the classes of the fibres of the projections of $X$ to the two factors and of the diagonal. One has the intersection numbers
\[
f_1^2=f_2^2=\delta^2=0, f_1\cdot f_2=f_1\cdot \delta=f_2\cdot \delta=1
\]
hence the intersection matrix has determinant 2. So $f_1,f_2,\delta$ are independent and it is a classical fact that they span $\Num(X)$ (e.g., this follows from \cite [Lemma 2.2]{CG}). 

The surface $X$ is an abelian surface, and $K_X$ is trivial.
For any irreducible curve on $X$, one has $C^2 \geq 0$, because $X$ is abelian and the translations move any curve in a 2--dimensional algebraic family. Moreover $C^2=0$ if and only if $C$ has genus 1. There are infinitely many such curves, e.g., the graphs of all multiplication maps $E\stackrel{\cdot n}\rightarrow E$.  The divisors $C$ such that $C^2>0$ and $C\cdot H>0$, with $H$ a fixed ample divisor (e.g., $H=f_1+f_2$), are effective. Hence $\NE(X)$ is the cone spanned by curves $C$ such that
\[
C\cdot H>0 \quad \text {and}\quad C^2\geq 0.
\]   
It has infinitely many extremal rays, actually uncountably many, corresponding to classes $\alpha$ with $\alpha\cdot H>0$  and $\alpha^2=0$, and only countably many of them are rational and provide contractions $E\times E\to E$.

%%%%%%%%%%%%%%%%%%%%%%%%%%%%%(BIBLIOGRAPHY)%%%%%%%%%%%%%%%%%%%%%%%%%%%%%%%
%
%
%%%%%%%%%%%%%%%%%%%%%%%%%%%%%%%%%%%%%%%%%%%%%%%%%%%%%%%%%%%%%%%%%%%%%%%

\end{document}